\documentclass[11pt]{article} 
\usepackage{ifthen}
\newboolean{arxiv}
\setboolean{arxiv}{true}
\newcommand{\citep}[1]{\cite{#1}}

\usepackage{ifthen}

\ifthenelse{\boolean{arxiv}}{
\usepackage{arxiv_def}
}
{
}

\title{Universality of max-margin classifiers}

\author{Andrea Montanari\thanks{Department of Electrical Engineering
    and Department of Statistics, Stanford University}
    \;\;\;\;Feng
  Ruan\thanks{Department of Statistics and Data Science, Northwestern University} \;\;\;\; 
  Basil Saeed\thanks{Department of Electrical Engineering, Stanford University} \;\;\;\; 
  Youngtak Sohn\thanks{Department of Mathematics, Massachusetts Institute of Technology}}

\begin{document}

\maketitle

\begin{abstract}%
Maximum margin binary classification is one of the most fundamental algorithms in machine learning, yet the role of featurization maps and the high-dimensional asymptotics of the misclassification error for non-Gaussian features are still poorly understood. We consider settings in which we observe binary labels $y_i$ and either 
$d$-dimensional covariates $\bz_i$ that are mapped to a $p$-dimension space via a randomized featurization map $\bphi:\R^d \to\R^p$, or $p$-dimensional features of non-Gaussian independent entries. In this context, we study two fundamental questions: $(i)$~At what overparametrization ratio $p/n$
do the data become linearly separable? $(ii)$ What is the generalization error of the max-margin classifier?

Working in the high-dimensional regime in which the number
of features $p$, the number of samples $n$
and the input dimension $d$ (in the nonlinear featurization setting) 
diverge, with ratios of order one,
we prove a universality result establishing that the asymptotic behavior is
completely determined by the expected covariance of feature vectors 
and by the covariance between features and labels.
In particular, the overparametrization threshold and generalization error can 
be computed within a simpler Gaussian model. 

The main technical challenge  lies in the fact that max-margin is not the 
maximizer (or minimizer) of an empirical average,  but the maximizer of a minimum over the samples.  We address this by representing the classifier as an average over support vectors. Crucially, we find that in high dimensions, the support vector count is proportional to the number of samples, which ultimately yields universality. 
\end{abstract}

\newcommand{\opt}{^\star}
\def\hkappa{{\widehat\kappa}}
\def\sRF{\mbox{\tiny\rm RF}}
\def\sind{\mbox{\tiny\rm IND}}
\def\RF{{\sf RF}}
\def\ind{\sf ind}
\newcommand{\minimize}{{\rm minimize}}
\newcommand{\subjectto}{\text{{\rm subject to}}}

\section{Introduction}
\label{section:intro}
\subsection{Motivation and overview of results}

Given linearly separable samples $\{(\bz_i,y_i)\}_{i\in[n]}$, $\bz_i\in\reals^d$, $y_i\in\{+1,-1\}$
the maximum margin classifier  maximizes the margin 
$\min_{i\in[n]} y_i \<\bz_i,\btheta\>$ over $\|\btheta\|_2\le 1$.
This algorithm is one of the most fundamental in machine learning
and theory developed in this setting has often provided the foundation for understanding more complex methods~\cite{shalev2014understanding}.  
 
If the data  $\{(\bz_i,y_i)\}_{i\in[n]}$ are not linearly separable,  
we can nevertheless perform max-margin classification by first mapping the covariates $\bz_i$ to 
a higher-dimensional space via a nonlinear featurization map $\bphi:\R^d \to\R^p$. We can  
then learn a classifier by
\begin{equation}
 \widehat\btheta_\MM = \argmax_{\norm{\btheta}_2 = 1} 
 \min_{i\in[n]} y_i \langle\bphi(\bz_i),\btheta\rangle .\label{eq:FirstMaxMargin}
\end{equation}
While such ideas date back to several decades ago, the role of the featurization map in the performance of the algorithm is poorly understood. 
For instance, the most basic question that one might 
ask is:
\begin{enumerate}
 \item[{\sf Q1.}] \emph{How large does $p$, the dimension of the features, need to be for the data to become linearly separable?}
\end{enumerate}
The classical theory replaces the finite dimensional space $\R^p$ by a Hilbert space $\cH$,
hence effectively taking $p=\infty$. More precisely $\cH$ is
(isomorphic to) a reproducing kernel Hilbert space (RKHS).  If the associated kernel
is ``non-degenerate'' (it is strictly positive definite) then it is dense in $L^2$, and
therefore any dataset is linearly separable in this space (provided $\bz_i\neq\bz_j$
for $i\neq j$). 
In contrast with the finite-$p$ case, this classical treatment is non-quantitative.

The importance of the finite-$p$ setting comes from at least two contexts.
\emph{First,} solving the max-margin problem \eqref{eq:FirstMaxMargin}  in the RKHS $\cH$ requires time at least
$n^2$, which is not feasible for massive datasets. Randomized, finite-dimensional, featurization
maps of the type studied here were introduced precisely to alleviate this computational bottleneck
\cite{balcan2006kernels,rahimi2007random,rahimi2008random}. \emph{Second,}
recent work has established that deep learning models can be approximated, in 
the so called `lazy training' regime, by linear models with randomized featurization maps, see e.g.
\cite{du2018gradient,oymak2019overparameterized,bartlett2021deep}.
The number of dimensions $p$ of these featurization maps 
(`neural tangent features') is equal to the number of parameters
in the original deep learning model. Therefore it can be large but never infinite.
In this latter context, maximum margin plays a special role since gradient descent with respect
to the cross-entropy (a.k.a. logistic) loss is known to converge to the max-margin
classifier \citep{soudry2018implicit}.

Beyond {\sf Q1}, a number of other questions are of interest for the max margin 
classifier. Among others:
\begin{enumerate}
 \item[{\sf Q2.}] \emph{How does the margin} 
\begin{equation}
\label{eqn:definition-of-max-margin-bD}
\hkappa_n := \max_{ \norm{\btheta}_2 \le 1}
		\min_{1\le i \le n} y_i \langle \bphi(\bz_i), \btheta\rangle\, 
\end{equation}
\emph{behave?}
In particular, how large does overparameterization ratio $p/n$ need to be for
the margin to approach its $p=\infty$ limit?
\item[{\sf Q3.}] \emph{How does the test classification error of the max margin classifier}
\begin{equation}
R(\widehat\btheta_{\MM}) = \P_{(\bz_\new, y_\new)}\left( y_\new \langle\widehat \btheta_\MM , \bphi(\bz_\new)\rangle \le 0\right)
\end{equation}
\emph{behave?}
From a statistical viewpoint, this is the most important question, since its answer can guide the choice 
of the number of features used. 
\end{enumerate}

Let us emphasize that sharp answers to these questions cannot be obtained by adapting
classical RKHS arguments that require $p=\infty$. Evidence of this
is provided by the fact that the threshold for the data to
become linearly separable is for $p/n$ of order one. 
Indeed, under the model described below, the data become separable, with high probability, provided 
$\lim\inf (p/n)> \gamma_{\star}$, for a suitable constant $\gamma_{\star}$ (this is 
a consequence of our results).
When $p$ is comparable with $n$, even concentration of the empirical kernel matrix around expectation 
breaks down \cite{pennington2017nonlinear,mei2022generalization}, and therefore obtaining sharp results requires to go beyond concentration ideas.

Our approach is instead based on establishing a \emph{universality statement}.
Roughly speaking, the answer to questions ${\sf Q1}$, ${\sf Q2}$, ${\sf Q3}$,
depends on the details of the featurizazation map $\bphi$ only through its
second-order statistics $\bSigma := \E[\bphi(\bz_i)\bphi(\bz_i)^{\sT}]$ 
(for simplicity, we will assume $\E[\bphi(\bz)]=\bzero$), and through its correlation with $y$.
As a consequence of this universality, calculations can be carried out
in an equivalent Gaussian model, whereby $\bx_i=\bphi(\bz_i)$ is replaced by 
$\bg_i\sim \cN(0,\bSigma)$. This Gaussian model is amenable to a variety of analytical techniques.

To be precise, we will prove universality results for two models of the features $\bx_i$:
\begin{enumerate}
\item[$(I).$] The \emph{random features} map of  \cite{balcan2006kernels,rahimi2007random}
\begin{equation}
 \bphi(\bz_i) = \sigma(\bW \bz_i)\, ,\label{eq:RandomFeatures}
\end{equation}
where $\sigma$  is a non-linear activation that acts element-wise, and 
$\bW\in\R^{p\times d}$ is a a random weight matrix.
In words, $\bphi(\,\cdot\,)$ corresponds to the application of a one-layer
neural network with random weights. We will take the initialization where the rows $(\bw_i)_{i\leq p}$ of $\bW$ are drawn independently and uniformly on the unit sphere $\bw_i\sim \mathsf{Unif}(\mathbb{S}^{d-1}(1))$.
\item[$(II).$] An \emph{independent features} model, whereby $\bx_i$ has independent 
centered entries with variance bounded away from zero and infinity. While this model does not arise from a particularly motivated featurization map, it is nevertheless a canonical 
random matrix model, and provides an important validation for our universality claims.
\end{enumerate}

We briefly summarize our universality results for the random features model $(I)$ above, 
deferring a more complete and precise presentation (including model $(II)$) to Section
\ref{sec:Main}. We assume a stylized data model whereby the covariates vectors $\bz_i$ 
do not have special structure and labels depend on a one-dimensional projection of the covariates.
Namely, the  $(\bz_i,y_i)_{i\in[n]}$ are i.i.d.
with
\begin{equation}
\label{eq:law:y}
	 \P(y_i = +1 | \bz_i) =  f(\bz_i^{\sT}\bbeta^\star ),\quad \norm{\bbeta^\star}_2=1, \quad \bz_i \sim \normal(0, I_d)\, .
\end{equation}
We will consider the proportional asymptotics:
\begin{equation}
	n, d, p\to \infty~~\text{with}~~
 		\frac{p}{d} \to \gamma_1 ,\quad \frac{n}{d} \to \gamma_2\, ,\label{eq:ProportionalAsymptotics}
\end{equation} 
and study the max margin classifier of Eq.~\eqref{eq:FirstMaxMargin}
with the featurization map of Eq.~\eqref{eq:RandomFeatures}. 

We compare this data distribution with an equivalent Gaussian model whereby
the nonlinear features $\bx_i = \sigma(\bW\bz_i)$ are replaced by Gaussian features
\begin{align}
\bg_i = \mu_0 + \mu_1 \bW \bz_i+ \mu_{2} \bh_i\, ,\label{eq:EquivGauss}
\end{align}
where $\mu_j$ are expressed in terms of coefficients of the expansion of $\sigma(\,\cdot\,)$ in Hermite polynomials (see Eq.~\eqref{eq:hermite:coeff}) and $\bh_i\sim\normal(0,\id_p)$ is independent of $\bz_i$. The features 
$\bg_i$  are Gaussian conditionally on $\bW$, and 
are constructed so that  the mean and covariance
of $(\bz_i,\bg_i)$ match those of $(\bz_i,\bphi(\bz_i))$ (always conditionally on $\bW$). 
The coefficients of Eq.~\eqref{eq:EquivGauss} are chosen so that 
$\|\bSigma_{\bx}-\bSigma_{\bg}\|_{\op} = o_n(1)$ under the proportional
asymptotics \eqref{eq:ProportionalAsymptotics}.
In the equivalent model, we seek a max-margin classifier: $\max_{\norm{\btheta}_2 = 1} 
 \min_{i\in[n]} y_i \langle\bg_i,\btheta\rangle$. 
 
Assuming the activation function to be twice differentiable, we 
establish universality of the margin and test error. Namely, denoting by
$\hkappa_n^{\bX}, R^{\bX}_n$ the margin and risk in the random features model, and 
by $\hkappa_n^{\bX}, R^{\bX}_n$ the same quantities in the Gaussian model, we prove that
\begin{align}
\big|\hkappa_n^{\bX} - \hkappa_n^{\bG}\big| &= o_\P(1)\, ,\\
\big|R^{\bX}_n - R_n^{\bG}\big| &= o_\P(1)\,.
\end{align}

\subsection{Technical innovation}
\label{sec:Technical}

A rich line of work investigated universality in high-dimensional statistics and 
statistical learning in the past \citep{korada2011applications,oymak2018universality,MontanariNg17,HuLu22,MontanariSa22}. 
Universality results are particularly useful because they allow
leveraging powerful techniques that are available for Gaussian features models in order to 
analyze non-Gaussian ones. 
However, all this work considers  \emph{empirical risk minimization (ERM)} problems of the form
\begin{equation}
\label{eqn:erm_general}
 \widehat\btheta  = \argmin_{\btheta} \cL(\btheta; \bZ, \by), \quad \quad
 \cL(\btheta; \bZ,y) = \frac1n \sum_{i=1}^n \ell(\btheta^\sT\bphi(\bz_i) , y_i) + r(\btheta)
\end{equation}
for some loss $\ell$ and possibly a regularizer $r$.
The case of the max-margin  presents conceptual challenges 
 that are not present in the case of empirical risk minimization.

To understand the additional difficulty, consider the margin
$\widehat\kappa_n$ defined in~\eqref{eqn:definition-of-max-margin-bD}.
Unlike the ERM setting in~\eqref{eqn:erm_general} where the objective $\cL(\btheta; \bZ, \by)$ is 
an average over i.i.d. data points,
the corresponding objective in the margin $\widehat \kappa_n$ is
a  \emph{minimum} over the data points.
For instance, stating that $\widehat{\kappa}_n >0$ (with high probability) 
requires the existence of a vector $\btheta$ such that the margin is strictly positive 
(with high probability) for \emph{every} single data point, and not only on average.

In order to understand the challenge, it is instructive to first revisit some of the approaches 
in the literature.
Consider the following optimization problem used in the analysis of 
max-margin classifiers~\cite{Stojnic13, montanari2019generalization}: 
\begin{equation}\label{eq:surrogate}
	\widehat F_n(\bD; \kappa) =
		\min_{\btheta: \norm{\btheta}_2 \le 1}
		 	\frac{1}{n} \sum_{i=1}^n (\kappa - y_i \langle \bd_i, \btheta\rangle)_+^2,
\end{equation}
where $\kappa >0$ is a fixed parameter, and we denote by $\bD$
either one of the random features model $\bX$ or the equivalent Gaussian features model
$\bG$. 
The connection  between this function and the max-margin is given by
\begin{equation}
\label{eqn:relation-between-F-and-kappa}
	\hkappa^{\bD}_n = \inf\big\{\kappa: \widehat F_n(\bD; \kappa) > 0\big\}. 
\end{equation} 
One might be tempted to
apply universality results in the literature \citep{HuLu22,MontanariSa22} to this formuation as follows:
\begin{enumerate}
\item Establish the universality of the surrogate function 
$\kappa \mapsto \widehat F_n(\bD; \kappa)$.
This can be done as a corollary of past results since Eq.~\eqref{eq:surrogate} takes the form of an ERM problem.
\item Exploit the connection~\eqref{eqn:relation-between-F-and-kappa} to translate the universality 
	of $\widehat F_n$ to that of the max-margin $\widehat \kappa_n$.
\end{enumerate}
This scheme, however, faces important challenges.
Existing universality results of $\widehat F_n$ yield: 
\begin{equation}
\label{eqn:hypo-universality-of-F_n}
	\norm{\widehat F_n(\bG; \; \cdot\;) - \widehat F_n(\bX; \;\cdot\;)} \to 0 
\end{equation}
under a suitable norm on the space of functions of $\kappa$. 
The strongest such metric one could 
hope for is the $\ell_\infty$ metric.
However,  establishing  universality of $\widehat F_n$, even under the $\ell_\infty$ metric, 
does not lead to the desired universality of max-margin.
Indeed, from equation~\eqref{eqn:relation-between-F-and-kappa},   $\hkappa_n$ is not a 
continuous functional
$\widehat F_n$: two functions $\widehat F_n(\bG; \; \cdot\;)$, $\widehat F_n(\bX; \;\cdot\;)$
can be uniformly close and the corresponding margins can be far.

This paper 
supplements ingredients necessary towards a rigorous proof of  universality.
Our approach leverages the dual formulation of the max-margin problem.
The crucial step is to prove that the dual optimum has a large support,
namely support size of order $n$.

In the next section, we give a more detailed review of related work on this topic. 
In Section~\ref{sec:Main}, we state our main results, while Section~\ref{sec:ProofOutline} outlines the proofs. The complete proofs are given in Sections~\ref{sec:proof-of-lemma-restricted-strong-convexity}-\ref{sec:delocalization:MM}.
%
%
\section{Further related work}
\label{sec:Related}

The random features model of Eq.~\eqref{eq:RandomFeatures} 
has received ample attention in the last few years
as the simplest example of a neural network in the linear (or `lazy') regime.
In  particular, \citep{mei2022generalization} characterized the generalization 
error of ridge regression under the proportional asymptotics of Eq.~\eqref{eq:ProportionalAsymptotics}.
Several phenomena observed empirically in fully trained 
neural networks (e.g. benign overfitting or the double descent phenomenon)
could be reproduced analytically in this setting.

The analysis of \citep{mei2022generalization} was based on random matrix theory methods
that is difficult to generalize beyond ridge-regularized least squares. 
In order to move beyond ridge regression, several groups conjectured universality,
and in particular that the random features model of Eq.~\eqref{eq:RandomFeatures} 
could be replaced by the equivalent Gaussian features model of Eq.~\eqref{eq:EquivGauss}.
The importance of this conjecture can be appreciated by the number of interesting  
results that \emph{assumed the universality conjecture to hold} in order to
derive interesting characterizations using Gaussian techniques. A small subset of these works
 incluse \cite{gerace2020generalisation,goldt2020modeling,goldt2020gaussian,loureiro2021learning,
LiangSu22,javanmard2022precise}.
Among others, \cite{gerace2020generalisation} used random features to model data with a latent low-dimensional structure, 
\cite{LiangSu22} studied boosting and minimum-$\ell_1$-norm interpolation
in random features models, while
\cite{javanmard2022precise} characterized adversarial training,

Our work establishes the universality conjecture for max-margin classification,
thus allowing to transfer a series of results that were previously obtained
for Gaussian features models to nonlinear random features models, and to random features
with independent entries. 
The study of max-margin classification in the Gaussian features spans a third of a century, 
and was initiated by the seminal work of Elizabeth Gardner \citep{gardner1988space}
who used the non-rigorous replica method from spin glass theory to 
compute the asymptotics of the maximum margin in the case of random labels and
isotropic covariates. This result was rigorized in \cite{shcherbina2003rigorous,Stojnic13}. The separability threshold
for the case of labels correlated with a linear function was determined in \cite{candes2020phase} (see also \cite{salehi2019impact,sur2019modern} for related work
on logistic regression). The maximum margin and generalization error in the case of non-isotropic Gaussian features was determined in \cite{montanari2019generalization}.
The results of \cite{montanari2019generalization} provided the first examples of `benign overfitting' in a classification setting:
The test error can be made arbitrarily small, despite even in a regime in which the number 
of parameters $p$ is much larger than the sample size $d$. 

Several universality theorems were proven in last two years
that cover nonlinear random features model for an increasingly broad class of training losses. In particular:
\begin{itemize}[-]
    \item\cite{mei2022generalization} used random matrix theory to prove universality for ridge regression (square loss).
    \item  \cite{HuLu22} apply Lindeberg method established universality for 
    a broad class linear models under strongly convex losses.
    \item\cite{MontanariSa22}  used a finite-temperature relaxation to extend 
    universality to non-convex losses, with additional requirements to obtain universality of the test error.
    \item Recent work generalized universality for the case in which the 
    network input is a Gaussian mixture \cite{dandi2023}, for multi-layer networks
    \cite{schroder2023}, and for adversarial training \cite{hassani2022curse}.
\end{itemize}

As discussed in the previous section, extending universality to 
max-margin classification  poses mathematical challenge that are not addressed
in earlier work.  We believe that our proof strategy is of independent 
 interest and susceptible to be generalized to other loss functions and other featurization 
 maps.
 
%
%
\section{Main results}
\label{sec:Main}
Let us first present the precise assumptions we make in each of the feature distributions $(I)$ and $(II)$ discussed in Section~\ref{section:intro}. In each of these settings, we define the distribution of the features $\bX$ precisely, state the corresponding assumptions, and define the corresponding \emph{Gaussian equivalent}, i.e., the Gaussian features $\bG$ whose first and second moments are (asymptotically) the same as those of $\bX$. We do this in Sections~\ref{section:ass_RF} and~\ref{section:ass_ind} below. In Section~\ref{section:univ_results}, we present our universality results which state that the asymptotics of the non-Gaussian and the Gaussian problem are the same. This allows us to leverage the precise asymptotics of
max-margin classification with Gaussian features~\citep{montanari2019generalization}
and apply them to non-Gaussian features. We demonstrate an application of this in Section~\ref{section:benign_overfitting} where our universality results are used to conclude that the conditions for benign overfitting from~\citep{montanari2019generalization} in the Gaussian case are the same in the non-Gaussian setting we consider.

\subsection{Random features model and its Gaussian equivalent}
\label{subsec:results:RF}
\label{section:ass_RF}
We first fully specify the data distribution under the random features model.
 Given a non-linear activation function $\sigma$, the random features $\bx_i$ and their associated labels are defined by
\begin{align}\label{eq:def:RF}
\bx_i = \sigma(\bW\bz_i),
\quad
	 \P(y_i = +1 | \bz_i) =  f(\bz_i^{\sT}\bbeta^\star ),\quad \norm{\bbeta^\star}_2=1,
\end{align}
where $\bz_i\sim \normal(0,\id_d)$ and $\bW\in \R^{p\times d}$ are independent. Here, we assume that the rows $(\bw_i)_{i\leq p}$ of $\bW$ are i.i.d. with $\bw_i\sim \mathsf{Unif}(\mathbb{S}^{d-1}(1))$, and the parameters $\mu_j$ correspond to the decomposition of the activation function $\sigma$
in Hermite polynomials:
\begin{align}\label{eq:hermite:coeff}
\mu_0 := \<\sigma,1\>_{L^2(\nu_{\sf G},\reals)}, \;\; 
\mu_1 := \<\sigma,x\>_{L^2(\nu_{\sf G},\reals)}, \;\; 
\mu_2 := \|\sigma\|_{L^2(\nu_{\sf G},\reals)}^2-\mu_0^2-\mu_1^2\, .
\end{align}
 Here the inner product and norm are with respect to the $L^2$ space of the standard Gaussian measure 
 $\nu_{\sf G}$ on $\reals$, i.e. $\<\sigma,f\>_{L^2(\nu_{\sf G},\reals)}:=\E[\sigma (G)f(G)]$
 for $G\sim\normal(0,1)$.
 
We make the following assumptions on the parameters of the model.
\begin{assumption}\label{assumption:RF}
For the random features model, we assume the proportional asymptotics in Eq.~\eqref{eq:ProportionalAsymptotics} and the following:
\begin{itemize}
    \item[${\sf (A1)}$] The activation $\sigma:\R\to\R$ is non-linear, and $\mu_0\equiv \E[\sigma(G)]=0$. Further, it is second order differentiable with uniformly bounded derivatives and satisfies $\E[\sigma''(G)]=0$. Namely, there exists  $C_\sigma$ such that
\begin{equation}
	\|\sigma'\|_\infty \vee \|\sigma''\|_\infty  \le C_\sigma < \infty\, .
\end{equation}
   \item[${\sf (A2)}$] The link function $f:\R\to [0,1]$ in Eq.~\eqref{eq:def:RF} is pseudo-Lipschitz. Namely, there exists $C_f<\infty$ such that
\begin{equation}\label{eq:f:pseudo}
 |f(t) - f(s)| \le C_f(1 + |s|  + |t|)|t - s|~~~~\forall t, s\in\reals.
\end{equation}
\end{itemize}
\end{assumption}
For example, this assumption is satisfied if $\sigma$ is odd and twice differentiable with bounded derivatives, which is the same assumption that \cite{HuLu22} has made. We remark that the universality is still expected to hold without this assumption (see numerical simulations for the ReLU activation function done in \cite{montanari2019generalization}). Indeed, we believe that our current proof technique to hold without the $\E[\sigma(G)]=\E[\sigma''(G)] = 0$ assumption, which we leave for future work.

We define the Gaussian equivalent features of this model as the linearization of~\eqref{eq:def:RF}:
\begin{equation}
\bg_i = \mu_0 + \mu_1 \bW \bz_i+ \mu_{2} \bh_i
\end{equation}
where $\bW$ and $\bz_i$ are as defined previously and $\bh_i\sim\normal(0,\id_p)$, independent of everything else.
Note that the features $\bz_i, \bg_i$ are jointly Gaussian conditionally on $\bW$.

We define the design matrices for random features $\bX \equiv\bX_{\RF}$ as the $n\times p$ matrix with rows $(\bx_i)_{i\leq n}$ and for Gaussian features $\bG\equiv \bG_{\RF}$ as the $n\times p$ matrix with rows $(\bg_i)_{i\leq n}$. We use the notation $(\bX,\bG)=(\bX_{\RF}, \bG_{\RF})$ to indicate the features and the equivalent Gaussian features follow the distributions defined above.

\begin{remark}
Standard manipulations of Gaussian random variables show that $y_i$ in Eq.~\eqref{eq:law:y} given $(\bg_i)_{i\leq n}$ follows a generalized linear model. This fact  is leveraged to apply existing 
results about generalized linear models with Gaussian covariates to
derive asymptotics for random features models \cite{montanari2019generalization}.

Explicitly, conditional on $\bW$, the labels $\by\equiv (y_i)_{i\leq n}$ of the Gaussian features given $(\bg_i)_{i\leq n}$ are distributed as
\begin{equation}
	 \P(y_i=+1\mid\bg_i)=1- \P(y_i=-1\mid\bg_i)=\bar{f}( \bg_i^{\sT}\btheta^\star)\, ,\;\;\;\; \bg_i\sim\normal(0, \bSigma_{\RF})\,, \label{eq:AlternativeDistr}
\end{equation}
where $\bSigma_{\RF}$, $\btheta^{\star}$, and the function $\bar{f}$ are given as follows.
First of all,
\begin{equation}
\label{eq:def:Sigma}
 \bSigma_{\RF} := \mu_1^2\bW\bW^{\sT}+\mu_2^2\id_p.  
\end{equation}
By matching covariances, we get $\bz_i = \mu_1\bW^{\sT}\bSigma_{\RF}^{-1}\bg_i+
\bQ^{1/2}\tilde{\bh}_i$, where $\tilde{\bh}_i\sim \normal(0,\id_d)$ independently of $\bz_i$,
and $\bQ:= \mu_2^2(\mu_2^2\id_d+\mu_1^2\bW^{\sT}\bW)^{-1}$.   
We can therefore rewrite  $\bbeta^{\star\sT}\bz_i =\alpha \btheta^{\star \sT}\bg_i+\eps_i$ for 
$\eps_i \sim \normal(0, \tau^2)$ independent of $\bg_i$, where
\begin{equation}\label{eq:def:theta:star}
\begin{split}
    \alpha:= \mu_1  \norm{\bSigma_\RF^{-1} \bW \bbeta\opt}_{2},~~~
    \btheta^\star:=\alpha^{-1}\mu_1\bSigma_{\RF}^{-1}\bW \bbeta^{\star},~~~
    \tau^2 := 1- \alpha^2   \norm{\btheta\opt}_{\bSigma_\RF}^2\, .
\end{split}
\end{equation}
This yields Eq.~\eqref{eq:AlternativeDistr} with
$\bar{f}(x):= \E_{\eps\sim \normal(0,1)}[f(\alpha x+\tau \eps)]$.
\end{remark}

%
%
\subsection{Independent features model and its Gaussian equivalent}\label{subsec:results:indep}
\label{section:ass_ind}
Our second result considers the design matrices
$\bX \equiv \bX_{\ind}\in \R^{n\times p}$, where the rows $(\bx_i)_{i\leq n}\in \R^{p}$ are i.i.d. with independent entries $\bx_i=(x_{ij})_{j\leq p}$ and labels $\by^{\bX}\equiv (y^{\bX}_{i})_{i\leq n} \in \{\pm 1\}^n$ defined by
\begin{equation}\label{eq:y:dist:ind:X}
    \P(y_i^{\bX} = +1 \mid \bx_i)=f(\bx_i^{\sT}\btheta^\star),
\end{equation}
where $\btheta^\star\in \R^p$ satisfies $\big\|\btheta^\star\big\|_2=1$. 
We make the following assumption on this model.
\begin{assumption}\label{assumption:ind}
For the independent features model, we assume the proportional asymptotics 
\begin{equation}\label{eq:ProportionalAsymptotics:ind}
n, p\to \infty~~\text{with}~~
 		\frac{p}{n} \to \gamma,
\end{equation}
and the following:
\begin{itemize}
    \item[${\sf (B1)}$]  The entries $(x_{ij})_{i\leq n, j\leq p}$ of $\bX_{\ind}$ are independent and subgaussian with parameter $\nu^2>0$. Further, the covariance matrix $\bSigma_{\ind}\equiv \textup{Cov}(\bx_i)$ has bounded condition number: \begin{equation*}
        c\leq \lambda_{\min}(\bSigma_{\ind})\leq\lambda_{\max}(\bSigma_{\ind})\leq C\,.
    \end{equation*}
    The constants $\nu^2,c,C>0$ do not depend on $n,p$.
    \item[${\sf (B2)}$]
    There exists a measure $\mu \in \mathcal{P}(\R_{>0}\times \R)$ such that the empirical distribution of $(\la_i, \sqrt{p} \theta_i^{\star})$ converges to $\mu$ in 
    Wasserstein-2 sense:

    \begin{equation}\label{eq:empirical:convergence}
     \frac{1}{p}\sum_{i=1}^{p}\delta_{(\la_i,\sqrt{p}\theta^{\star}_i)}\stackrel{W_2}{\Longrightarrow}\mu\,.
    \end{equation}

    \item[${\sf (B3)}$] The link function $f:\R\to (0,1)$ takes values in $(0,1)$ and is again pseudo-Lipschitz with constant $C_f >0$,
    cf. Eq.~\eqref{eq:f:pseudo}.
\end{itemize}
\end{assumption}

The corresponding Gaussian equivalent feature matrix $\bG \equiv \bG_{\ind}\in \R^{n\times p}$ is defined to have 
 i.i.d. rows $\bg_i \sim \normal(0,\bSigma_{\ind})$, $i\le n$.
In contrast with the random features setting, the labels for the Gaussian features 
 $\by^{\bG}\equiv (y^{\bG}_{i})_{i\leq n} \in \{\pm 1\}^n$ have a different marginal distribution than those for the non-Gaussian features as the notation indicates. Namely, we have
\begin{equation}\label{eq:y:dist:ind:G}
\P(y_i^{\bG} = +1 \mid \bg_i)=f(\bg_i^{\sT}\btheta^\star)
\end{equation}
for $\btheta^\star$ as above.
We use the notation $(\bX,\bG)=(\bX_{\ind},\bG_{\ind})$ to indicate that the non-Gaussian features and their equivalent follow the \emph{independent features model} defined above.

\begin{remark}\label{rmk:assumption:ind}
 Our results hold under the following weaker assumption, alternative to ${\sf (B2)}$:
\begin{itemize}
\item[${\sf (B2^\prime)}$] The following tightness condition holds:
\begin{align}
\lim_{R\to\infty}\limsup_{p\to\infty}\frac{1}{p} \sum_{i=1}^{p}
\big[\lambda_i^2+p(\theta^\star_i)^2\big]
\one_{\lambda_i^2+p(\theta^\star_i)^2\geq R}=0.\label{eq:Tightness}
\end{align}
\end{itemize}
Indeed, under this condition, by Prokhorov's theorem the empitical distribution 
$\frac{1}{p}\sum_{i=1}^{p}\delta_{(\la_i,\sqrt{p}\theta^{\star}_i)}$
converges weakly on subsequences. By  \cite[Definition 6.8]{villani2008optimal},
it also converges `weakly in $W_2$', which by 
 \cite[Theorem 6.9]{villani2008optimal} is equivalent to the convergence of 
 Eq.~\eqref{eq:empirical:convergence}. Therefore we can apply
 the universality theorems below under to these subsequences, and hence
 apply them under the weaker assumption \eqref{eq:Tightness}.
 \end{remark}

\begin{remark}\label{rmk:InftyNorm}
Either condition ${\sf (B2)}$ or the weaker condition 
${\sf (B2')}$ imply $\big\|\btheta^\star\|_{\infty}=o_n(1)$. 
This is  a necessary to guarantee the $\bx_i^{\sT}\btheta^\star$ is distributed asymptotically the same as $\bg_i^{\sT}\btheta^\star$.
 See~\cite[Section 6]{villani2008optimal} for other equivalent definitions of the convergence Wasserstein $2$-distance. 
 \end{remark}

\subsection{Universality theorems}
\label{section:univ_results}

Recall the definition of maximum margin 
\begin{align*}
 \hkappa^{\bX}_n = \max_{ \norm{\btheta}_2 \le 1}
		\min_{1\le i \le n} y_i \<\bx_i, \btheta\>\, ,\;\;\;\;\;  \hkappa^{\bG}_n = \max_{ \norm{\btheta}_2 \le 1}
		\min_{1\le i \le n} y_i \<\bg_i, \btheta\>\,.
\end{align*}
Denote by $\bx_\new,\bg_\new,y_\new$ a fresh sample. 

Then, the test errors of $\btheta\in \R^p$ are defined by
\begin{align}\label{eq:def:test:error}
R_n^{\bx}(\btheta)=\P\left(y_\new \langle \btheta, \bx_\new\rangle \leq 0\right)\, ,\;\;\;\;\; R_n^{\bg}(\btheta)=\P\left(y_\new \langle \btheta, \bg_\new\rangle \leq 0\right)\,.
\end{align}
Recall that the max-margin solutions are defined by
\begin{equation*}
    \widehat{\btheta}_{\MM}^{\bX}:=\argmax_{\norm{\btheta}_2\leq 1}\Big\{\min_{i\leq n}y_i\langle \bx_i,\btheta\rangle\Big\}\, ,\;\;\;\;\;  
    \widehat{\btheta}_{\MM}^{\bG}:=\argmax_{\norm{\btheta}_2\leq 1}\Big\{\min_{i\leq n}y_i\langle \bg_i,\btheta\rangle\Big\}\,.
\end{equation*}
We denote the corresponding test errors by $R^{\bX}_n := R_n^{\bx}(\widehat{\btheta}_{\MM}^{\bX})$ and $R^{\bG}_n := R_n^{\bg}(\widehat{\btheta}_{\MM}^{\bG})$. It is important to emphasize that these are random variables, because of the randomness coming from the training data, but as discussed below, we expect them to concentrate in the high-dimensional asymptotics we study. We also note that they depend on the data distribution both on the 
test point and (in a more intricate way) on
the training data.\\

Our first result concerns the limit of the margin under the distribution classes defined in Subsections~\ref{section:ass_RF} and~\ref{section:ass_ind}.
\begin{theorem}[Universality of the margin]\label{theorem:universality-of-the-margin}
Consider either $(\bX,\bG)=(\bX_{\RF},\bG_{\RF})$ under Assumption~\ref{assumption:RF} or $(\bX,\bG)=(\bX_{\ind},\bG_{\ind})$ under Assumption~\ref{assumption:ind}. Then, for any Lipschitz function $\psi:\R\to \R$,
\begin{equation*}
\lim_{n\to\infty}\big|\E[\psi(\hkappa_n^{\bX})]-\E[\psi(\hkappa_n^{\bG})]\big|=0\,.
\end{equation*}
\end{theorem}
\cite{montanari2019generalization} computed the exact asymptotics for the
max-margin $\hkappa_n^{\bG}$ in the Gaussian design, which we include in Section~\ref{sec:formula} of the Appendix for completeness.
In particular, for $\bG=\bG_{\RF}$, there exists a curve of thresholds $\tau^\star_{\RF}:\R_{+}\to\R_{+}$ such that if $\gamma_1>\tau^\star_{\RF}(\gamma_2)$, then $\hkappa_n^{\bG}\pto \kappa^\star_{\RF}(\gamma_1,\gamma_2)$ holds for some $\kappa^\star_{\RF}(\gamma_1,\gamma_2)>0$. Conversely, if $\gamma_1<\tau^\star_{\RF}(\gamma_2)$, then $\hkappa_n^{\bG}\pto 0$ holds.  A similar statement holds for $\bG = \bG_{\ind}$ concerning the corresponding overparmeterization ratio $\gamma$. 
By Theorem~ \ref{theorem:universality-of-the-margin}, we can transfer these results to the random features model and the independent features model. 
For any $\kappa>0$, this allows us to characterize the regime in which the data, under both distribution classes of Subsections~\ref{section:ass_RF} and~\ref{section:ass_ind}, are separable with a margin $\kappa >0$ with high probability.
See Section~\ref{sec:formula} for the complete definition of these functions.

Our next result concerns the limit of the test error of the max-margin classifier.
\begin{theorem}[Universality of the test error]\label{thm:test}
Consider either $(\bX,\bG)=(\bX_{\RF},\bG_{\RF})$ under Assumption~\ref{assumption:RF} or $(\bX,\bG)=(\bX_{\ind},\bG_{\ind})$ under Assumption~\ref{assumption:ind}. In the separable regime, for any Lipschitz function $\psi:\R\to \R$,
 \begin{equation}\label{eqn:ultimate-goal-of-test-error}
 	\lim_{n \to \infty}
		\left|\E[\psi(R_n^{\bX})] - \E[\psi(R_n^{\bG})]\right| = 0.
 \end{equation}
\end{theorem}

We remark again that~\cite{montanari2019generalization} computed the exact asymptotics for the test error $R_n^{\bG}$ in the Gaussian design.
Section~\ref{sec:formula} also includes these asymptotics.
\newline
 
 Let us emphasize two important  aspects of our results. 
 \emph{First,} by~\cite{montanari2019generalization}, the margin and error in the Gaussian model concentrate around
 non-random values, and therefore Theorems \ref{theorem:universality-of-the-margin} and ~\ref{thm:test}
 imply $|\hkappa_n^{\bX} -\hkappa_n^{\bG}|=o_{\P}(1)$, $|R_n^{\bX} -R_n^{\bG}|=o_{\P}(1)$.
 \emph{Second,} in the proportional
  asymptotics, the test error and the margin are of constant order. In fact even the excess error 
  (over the Bayes risk)
  is of constant order, although it can be made an arbitrarily small in suitable settings. 
  As a consequence, our result $|R_n^{\bX} - R_n^{\bG}|=o_\P(1)$ allows characterizing 
  the excess error up to sub-dominant errors.

\subsection{Implications for benign overfitting}
\label{section:benign_overfitting}
As an implication of our main results, we characterize the sequences of 
$(\bSigma_{\ind},\btheta^\star)$ along which the independent features 
model displays the so-called \emph{benign overfitting} phenomenon. 
We use this term to refer to cases in which 
the excess error (defined as the difference between the test error and 
the Bayes error) vanishes as $n, p$ tends to infinity,
despite  the fact that the training error is much smaller than the 
Bayes error (in our case, vanishes as $n,p$ diverge).

More precisely, let us first define the Bayes error for the independent 
feature model. Following the notation in Section~\ref{section:ass_ind}, for a data pair 
$(\bx,\by^{\bX})\equiv(\bx_1, \by^{\bX}_1)$ satisfying Eq.~\eqref{eq:y:dist:ind:X}, we define
\begin{equation*}
    {\rm Bayes}_n^{\bX} 
        =\inf_{\hy:\reals^p\to\{\pm 1\}} \P\big(\hy(\bx)\neq y^{\bX}\big)\, .
\end{equation*}
where the infimum is taken over all possible predictors $\hat{y}:\R^{p}\to \{+1,-1\}$. Similarly, for a data pair $(\bg,\by^{\bG})\equiv (\bg_1,\by^{\bG}_1)$ satisfying Eq.~\eqref{eq:y:dist:ind:G}, we define the Bayes error for the Gaussian  equivalent 
independent feature model 
\begin{equation*}
    {\rm Bayes}_n^{\bG} 
        =\inf_{\hy:\reals^p\to\{\pm 1\}} \P\big(\hy(\bg)\neq y^{\bG}\big)\, .
\end{equation*}

Theorem 2 of \cite{montanari2019generalization}
characterizes the excess error under the Gaussian 
feature model. This result requires an additional regularity assumption 
on the link function $f$. 
\begin{assumption}
\label{assumption:f-more-regularity}
$f$ is differentiable, monotonically increasing with $f(0)=1/2$ and $f^\prime(0)>0$
\end{assumption}
Let $\rho:=\big(\int \lambda \theta^2 \mu(\de \la, \de \theta)\big)^{1/2}$ be the limit of $\big(\langle \btheta^\star, \bSigma_{\ind}\btheta^\star\rangle \big)^{1/2}$, and $\lambda_{\sf M} > 0$ be an upper bound on the maximum
eigenvalue of $\bSigma_{\ind}$:
$\lambda_{\max}(\bSigma_{\ind})\leq \lambda_{\sf M}$. 
Given a link function $f$ satisfying Assumption~\ref{assumption:ind}
and~\ref{assumption:f-more-regularity}, \cite[Theorem 2]{montanari2019generalization} proves the following bound
on the excess error for the Gaussian equivalent model. With 
probability tending to one as $n \to \infty$: 
\begin{equation}
\label{eqn:excess-error-bound}
c\cdot \frac{n}{p}\leq R_n^{\bG}-{\rm Bayes}_n^{\bG} \leq C\cdot\Big(\mathcal{B}_n(\la)+\mathcal{V}_n(\la)\Big)\,.
\end{equation}
Here, the constants $c, C > 0$ depend solely on $f$, $\rho$, $\lambda_{\mathsf M}$.
The quantities $\mathcal{B}_n(\la)$ and $\mathcal{V}_n(\la)$ are given as follows. Let $\la_1\geq \ldots\ge \la_p$ be the diagonal entries of the matrix $\bSigma_{\ind}$, i.e., $\bSigma_{\ind}=\diag(\la_i)_{1\leq i\leq p}$, and let $\btheta^\star=(\theta^\star_i)_{1\leq i\leq p}$. Then,
\begin{equation}\label{eq:def:B:V}
\begin{split}
    &\mathcal{B}_n(\lambda):= 
    \frac{1}{\<\btheta^\star,\bSigma_{\ind} \btheta^\star\>}
    \left\{
    \Big(\frac{\lambda r_1(\lambda)}{n}\Big)^2 \sum_{i:\lambda_i>\lambda}
    \frac{\big(\theta^\star_i\big)^2}{\lambda_i}+ \sum_{i:\lambda_i\le \lambda}
    \lambda_i\big(\theta^\star_i\big)^2
    \right\}
    \, ,\\
    &\mathcal{V}_n(\lambda):= \frac{r_0(\lambda)}{n}+\frac{n}{\overline{r}(\lambda)}\, ,
\end{split}
\end{equation}
where we defined, for $q\ge 0$,
\begin{align*}
r_q(\lambda) :=\sum_{i:\lambda_i\le \lambda}\Big(\frac{\lambda_i}{\lambda}\Big)^q \, ,
\;\;\;\;\;\;\; \overline{r}(\lambda) := \frac{r_1(\lambda)^2}{r_2(\lambda)}\, .
\end{align*}
In particular, if the sequence of parameters
$(\bSigma_{\ind},\btheta^\star)$ obeys $\lim_{n \to \infty} B_n(\lambda) = 
\lim_{n \to \infty} V_n(\lambda) = 0$ for some $\la>0$, then for the Gaussian feature model, 
the excess error vanishes: 
\begin{equation}
    \lim_{n \to \infty} \P(R_n^{\bG} \le {\rm Bayes}_n^{\bG} + \delta) = 1
    ~~\text{for every $\delta > 0$}.
\end{equation}
Sequences of $(\bSigma_{\ind},\btheta^\star)$ that obey 
these requirements are given in \cite[Example 1,2]{montanari2019generalization}. 

To summarize, while an excess error bound of the form~\eqref{eqn:excess-error-bound} is critical to the understanding of benign overfitting
for the max-margin classifiers, such a result was previously only known for Gaussian features. Armed with our universality result (Theorem~\ref{thm:test}), we 
show that the analogous result holds for the general independent 
feature model. 
\begin{corollary}\label{cor:benign}
Consider the independent features model $\bX=\bX_{\ind}$ under Assumptions~\ref{assumption:ind} and~\ref{assumption:f-more-regularity}. Denote by $\lambda_{\sf M}>0$ a constant such that $\lambda_{\max}(\bSigma_{\ind})\leq \lambda_{\sf M}$ and $\rho:=\big(\int \lambda \theta^2 \mu(\de \la, \de \theta)\big)^{1/2}$ the limit of $\big(\langle \btheta^\star, \bSigma_{\ind}\btheta^\star\rangle \big)^{1/2}$. Then, there exists constants $c,C>0$ that only depends on $\rho, \lambda_{\sf M}$, and $f$ such that for any $\la>0$,
\begin{equation*}
    c\cdot \frac{n}{p}\leq R_n^{\bX}-{\rm Bayes}_n^{\bX} \leq C\cdot\Big(\mathcal{B}_n(\la)+\mathcal{V}_n(\la)\Big)\,,
\end{equation*}
where $\mathcal{B}_n(\la)$ and $\mathcal{V}_n(\la)$ are defined in Eq.~\eqref{eq:def:B:V}.
\end{corollary}

This corollary  follows straightforwardly from Theorem~\ref{thm:test} and the bound on the excess error for the Gaussian equivalent model in Eq.~\eqref{eqn:excess-error-bound}. 
Indeed, under Assumptions~\ref{assumption:ind} and \ref{assumption:f-more-regularity}, $|{\rm Bayes}_n^{\bX}-{\rm Bayes}_n^{\bG}|=
|R_n^{\bx}(\btheta^\star)-R_n^{\bg}(\btheta^\star)|\to 0$ as $n\to\infty$ since $\|\btheta^\star\|_{\infty}\to 0$ as $n\to\infty$ (see Remark~\ref{rmk:InftyNorm}).
We refer to Proposition~\ref{proposition:asymptotic-equivalence-for-regular-estimators} for details.

Corollary~\ref{cor:benign} shows that overparametrization is necessary for max-margin estimators to achieve near Bayes error. Further, it shows that the excess error is small if the eigenvalues of the population covariance $\bSigma_{\ind}$ is sufficiently slowly decaying and the coordinates of the signal $\btheta^\star$ corresponding to small eigenvalues has small magnitude.

Benign overfitting has been demonstrated empirically and established rigorously
in a number of settings \cite{belkin2018overfitting,belkin2019two,bartlett2020benign,koehler2021uniform,hastie2022surprises,mei2022generalization,montanari2022interpolation,zhou2022non, tsigler2023benign}. Notably, the bounds of Corollary~\ref{cor:benign} align closely with results for ridge(less) regression in \cite{bartlett2020benign, tsigler2023benign}.


Our current contribution is focused on max-margin classification.
Non-asymptotic bounds for this problem were proven in \cite{chatterji2021finite},
implying benign overfitting for mixture of well separated data distribution. This analysis was generalized to two-layer ReLU networks in  \cite{frei2022benign,frei2023benign}.
In contrast our characterization does not require such separation
to hold, and is the first one that holds beyond the Gaussian
setting of  \cite{montanari2019generalization}.

\section{Proofs of main results}
\label{sec:ProofOutline}
%
%


In this section we state the main technical lemmas required to prove our
main results for Theorem  
\ref{theorem:universality-of-the-margin} and Theorem \ref{thm:test}.
The proofs of these lemmas are deferred to later sections: here we show 
that they indeed imply the statements of the theorems. 

Throughout, we assume that either $(\bX,\bG)=(\bX_{\RF},\bG_{\RF})$ holds under Assumption \ref{assumption:RF} or $(\bX,\bG)=(\bX_{\ind},\bG_{\ind})$ holds under Assumption \ref{assumption:ind}. For example, in the statements of the type ``for $\bX$, we have...'', we implicitly assume that either $\bX=\bX_{\RF}$ holds under Assumption \ref{assumption:RF} or $\bX=\bX_{\ind}$ holds under Assumption \ref{assumption:ind}. Moreover, for simplicity, we use the notation $\bD\in \R^{n\times p}$ to denote a generic matrix ensemble $\bD\in \{\bX_{\RF},\bG_{\RF},\bX_{\ind},\bG_{\ind}\}$ under these assumptions.

\subsection{Proof of Theorem~\ref{theorem:universality-of-the-margin}: Universality of Max-Margin} 
\label{sec:proof-of-universality-max-margin}

\subsubsection{The max-margin as an average over support vectors}
As discussed in Section~\ref{section:intro}, the main difficulty in proving universality 
is that the definition in~\eqref{eq:FirstMaxMargin} is an 
\emph{extremum} rather than an 
\emph{average} over the i.i.d. data points. 
We overcome this by reformulating the problem as an average over the support vectors. To deduce universality from such a formulation, we show that, under the proportional asymptotics,
the number of support vectors is proportional to the sample size $n$.


Formally, consider the dual representation
\begin{equation}
\label{eqn:definition-of-max-margin-in-u}
	\widehat \kappa_n^\bX = \max_{\btheta: \norm{\btheta}_2 \le 1}
		\min_{\bu: \bu \ge 0, \langle \one, \bu \rangle = 1} \sum_{i=1}^n u_i 
			y_i \langle \bx_i, \btheta\rangle. 
\end{equation}
Here, $\bu$ is the dual parameter.
By convex-concave duality, there exists an optimal dual parameter $\bu^*$ such that 
\begin{equation}
\label{eqn:optimal-dual-variable}
	\bu^* = \argmin_{\bu: \bu \ge 0, \langle \one, \bu \rangle = 1} \widehat G_n^{\bX}(\bu)
	~~\text{where}~~\widehat G_n^{\bX}(\bu) = 
		\norm{\sum_{i=1}^n u_i y_i \bx_i}_2. 
\end{equation} 
We establish that $\bu \mapsto \widehat G_n^{\bX}(\bu)$ 
satisfies a form of \emph{restricted strong convexity} (with high probability) stated as follows.

\begin{lemma}[Restricted Strong Convexity]
\label{lemma:restricted-strong-convexity} 
For $\bX\in \{\bX_{\RF},\bX_{\ind}\}$, there exist 
constants $c, c_1, c_2 > 0$ independent of $n$ such that 
with probability at least $1-e^{-cn}$,  
\begin{equation}
	\widehat G_n^{\bX}(\bu) \ge c_1 \sqrt{n} \norm{\bu}_2 -  c_2 \norm{\bu}_1~\text{holds for all $\bu \in \R^n$}.
\end{equation} 
\end{lemma}

 The proof 
of Lemma~\ref{lemma:restricted-strong-convexity} is deferred to Section~\ref{sec:proof-of-lemma-restricted-strong-convexity}.
Similar results have appeared in the literature of compressed sensing \cite{CandesTa06, Donoho06}, and, more broadly, 
high dimensional statistics~\cite{Wainwright19}, albeit with different random matrix ensembles. 

Using this lemma, we establish that there exists a constant $c>0$ such
that the optimal dual parameter $\bu\opt$ defined by 
equation~\eqref{eqn:optimal-dual-variable} has support size at least $cn$ (with high probability) by the following argument.
First, note that the max-margin, with probability $1-e^{-cn}$, is upper bounded by a constant since
\begin{equation}
\label{eqn:max-margin-upper-bound}
	\widehat \kappa_n^{\bX} =  \widehat G_n^{\bX}(\bu^*) \le \widehat G_n^{\bX}(\one/n) \le c_3,
\end{equation}
where the last inequality uses the fact that the random matrix $\bX$ has operator norm 
$\norm{\bX}_{\op} \le c_3\sqrt{n}$ with high probability (Lemma~\ref{lem:basic:op:norm:RF}). 
Since $\bu\opt$ satisfies $\norm{\bu\opt}_1 = 1$ by definition, 
Lemma~\ref{lemma:restricted-strong-convexity} shows a lower bound
of the max-margin that holds with probability $1-e^{-cn}$: 
\begin{equation}
\label{eqn:max-margin-lower-bound}
		\widehat \kappa_n^{\bX} = \widehat G_n^{\bX}(\bu^*) \ge c_1 \sqrt{n} \norm{\bu\opt}_2 -  c_2.
\end{equation}
Combining the upper and lower bounds in 
~\eqref{eqn:max-margin-upper-bound} and~\eqref{eqn:max-margin-lower-bound},
we conclude that
\begin{equation}
	\frac{\norm{\bu^*}_2}{\norm{\bu^*}_1} \le \frac{1}{c_4 \sqrt{n}}~~\text{and hence}~~
	\norm{\bu^*}_0 \ge \frac{\norm{\bu^*}_1^2} {\norm{\bu^*}_2^2} \ge c_4^2 n
\end{equation}
where $c_4 = (c_2 + c_3)/c_1$.
In other words, as claimed, the number of 
support vectors is at least $cn$. This implies that there exists optimal $\bu^*$ with 
$\norm{\bu^*}_\infty \le C/n$.
We thus proved the following proposition.
\begin{proposition}
\label{proposition:well-controlled-u}
There exist constants  $c, C > 0$ such that for $\bX\in \{\bX_{\RF},\bX_{\ind}\}$, the following identity holds with probability at least $1-e^{-cn}$: 
\begin{equation}
	\widehat \kappa_n^{\bX} = \max_{\btheta: \norm{\btheta}_2 \le 1}
		\min_{\substack{\bu: \bu \ge 0, \langle \one, \bu \rangle = 1; 
			\\\norm{\bu}_\infty \le C/n}} \sum_{i=1}^n u_i 
			y_i \langle \bx_i, \btheta\rangle. 
\end{equation}
\end{proposition}

Consequently, the max-margin problem enjoys the dual representation 
\begin{equation}
\label{eqn:max-margin-dual-representation-averaging}
	\widehat \kappa_n^{\bX} = \max_{\btheta: \norm{\btheta}_2 \le 1} 
		\sum_{i=1}^n u_i^* y_i \langle \bx_i, \btheta\rangle
\end{equation}
for some $\bu^*$ whose support size is at least proportional to $n$, i.e.,
$\norm{\bu^*}_0 \ge cn$ and $\norm{\bu^\star}_\infty \le C/n$.
That is, the objective is a weighted empirical average over at least $cn$ 
support vectors, whose weights are precisely determined by the optimal dual $\bu^*$
(note $\bu^*\ge 0$, $\langle\one, \bu^* \rangle = 1$). 

\subsubsection{Relation to empirical risk minimzation}
Using the result of previous section, we will lower and upper bound the margin by quantities resembling an \emph{empirical risk minimization}, which in turn can be shown to be universal using  a general universality theorem for empirical risk minimization 
from~\cite{MontanariSa22}.
To this end, define
\begin{equation}
\label{eq:def_F_n}
    \widehat F_n(\bX; \kappa) := \min_{\norm{\btheta}_2 \le 1}\frac1n \sum_{i=1}^n \left(\kappa - y_i \inner{\bx_i , \btheta}\right)_+^2.
\end{equation}
Note that
\begin{equation}\label{eq:equiv:F:kappa}
     \widehat\kappa_{n}^{\bX}\geq \kappa \iff \widehat F_n(\bX; \kappa)=0.
\end{equation}
Further, we  have the following universality result. 
\begin{proposition}
\label{sec:proof-proposition-universality}
Fix $\kappa > 0$, and any Lipschitz function $\psi:\R\to\R$. Then for $(\bX,\bG)=(\bX_{\RF},\bG_{\RF})$ and $(\bX,\bG)=(\bX_{\ind},\bG_{\ind})$, we have
\begin{equation}
	\lim_{n \to \infty} 
		\Big|\E\left[\psi(\widehat F_n(\bX; \kappa))\right] - \E\left[\psi(\widehat F_n(\bG; \kappa))\right]\Big| = 0.
\end{equation}
\end{proposition}
The proof of Proposition \ref{sec:proof-proposition-universality} is deferred to Section~\ref{sec:ERM}, where we apply \cite[Theorem 1]{MontanariSa22}.
\begin{remark}
Let us emphasize that applying  \cite[Theorem 1]{MontanariSa22} to the present problem is highly non-trivial and requires additional work. Indeed, the crucial assumption of \cite[Theorem 1]{MontanariSa22} is that the minimizer of~\eqref{eq:def_F_n}, denoted by $\widehat\btheta$, satisfies a delocalization condition. Namely, that $\widehat\btheta$ is contained in the (deterministic) set of $\btheta$ which $\langle \btheta, \bx_i\rangle$ has the same asymptotic distribution as $\langle \btheta, \bg_i\rangle$. We achieve this by showing that $\|\widehat\btheta\|_{\infty}=o_n(1)$ holds w.h.p. by a leave-one-out argument. We refer to Lemma~\ref{lemma:inf_norm_ERM} and its proof for details.
\end{remark}

Next, we want to use Proposition \ref{sec:proof-proposition-universality} to prove universality of the margin.
Proposition~\ref{proposition:well-controlled-u} in the previous section implies a lower bound 
on the margin in terms of $\widehat F_n(\bX; \kappa)$ as follows.
Fixing any $\kappa > 0$, with probability tending to one,
\begin{equation}
\label{eq:lower_bound}
\begin{split}
	\widehat \kappa_n^{\bX} &= \max_{\btheta: \norm{\btheta}_2 \le 1}
		\sum_{i=1}^n u_i\opt y_i \langle \bx_i, \btheta\rangle
		\stackrel{(a)}{=}  \kappa - \min_{\btheta: \norm{\btheta}_2 \le 1}
		\sum_{i=1}^n u_i\opt (\kappa - y_i \langle \bx_i, \btheta\rangle) \\
 		&\stackrel{(b)}{\ge}  \kappa -  \min_{\btheta: \norm{\btheta}_2 \le 1} 
			\frac{C}{n} \sum_{i=1}^n (\kappa - y_i \langle \bx_i, \btheta\rangle)_+
		\stackrel{(c)}{\ge} \kappa - C \cdot (\widehat F_n(\bX; \kappa ))^{1/2}. 
\end{split}
\end{equation}
Here, part $(a)$ follows from $\langle \one, \bu\opt \rangle = 1$, part 
$(b)$ follows from $0 \le \bu_i^* \le C/n$ due to Proposition \ref{proposition:well-controlled-u}, and part $(c)$ follows from Cauchy-Schwartz.

\subsubsection{Universality of the margin}

We can now complete the proof of Theorem~\ref{theorem:universality-of-the-margin}.
Let us first recall the result of~\cite{montanari2019generalization} on the Gaussian equivalent 
problem. 

\begin{proposition}[{\cite[Theorem 3]{montanari2019generalization}}]
\label{proposition:max-margin-gaussian}
There exists $\kappa^\star=\kappa^\star_{\RF}(\gamma_1,\gamma_2)\in [0,\infty)$ for $\bG=\bG_{\RF}$, and $\kappa^\star = \kappa^\star_{\ind}(\gamma;\mu)\in [0,\infty)$, where $\mu$ is given in Eq.~\eqref{eq:empirical:convergence}, for $\bG=\bG_{\ind}$ such that the following holds. 
\begin{itemize}
\item If $\kappa > \kappa^\star$, then there exists $\delta \equiv \delta(\kappa)>0$ such that $\liminf_{n \to \infty} \widehat F_n(\bG_{\RF}; \kappa) \ge \delta$ holds with probability at least $1-e^{-cn}$.
\item If $0\leq \kappa<\kappa^\star$, then $\lim_{n\to\infty}\widehat F_n(\bG;\kappa)=0$ with probability at least $1-e^{-cn}$.
 \end{itemize}
 Further, there exists a curve of thresholds $\tau^\star_{\RF}:\R_{+}\to\R_{+}$ such that $\kappa^\star_{\RF}(\gamma_1, \gamma_2)>0$ holds for $\gamma_1<\tau^\star_{\RF}$, and $\kappa^\star_{\RF}(\gamma_1,\gamma_2)=0$ holds for $\gamma_1<\tau^\star_{\RF}(\gamma_2)$. Similarly, there exists $\gamma^\star_{\ind}(\mu)$ such that $\kappa^\star_{\ind}(\gamma;\mu)>0$ holds for $\gamma>\gamma^\star_{\ind}(\mu)$ and $\kappa^\star_{\ind}(\gamma;\mu)=0$ holds for $\gamma<\gamma^\star_{\ind}(\mu)$.
\end{proposition}

\begin{proof}[Proof of Theorem \ref{theorem:universality-of-the-margin}]For $(\bX,\bG)=(\bX_{\RF},\bG_{\RF})$, let $\kappa^\star \equiv \kappa^\star(\gamma_1, \gamma_2)$ denote the asymptotic max-margin of the Gaussian 
equivalent as described in Proposition~\ref{proposition:max-margin-gaussian}. Similarly, for $(\bX,\bG)=(\bX_{\ind},\bG_{\ind})$, let $\kappa^\star\equiv \kappa^\star_{\ind}(\gamma;\mu)$. 
Our goal is to show that $\widehat{\kappa}_n^{\bX}$ converges to the same 
value $\kappa^{\star}$ under the proportional asymptotics. 
We divide our proof into the upper bound and the lower bound for $\widehat\kappa_n^{\bX}$.
\paragraph{Upper bound.}
By Proposition \ref{proposition:max-margin-gaussian},  for any $\eps >0$, there exists $\delta\equiv \delta(\eps)$ such that w.h.p.,
\begin{equation}
\liminf_{n\to \infty} \widehat F_n(\bG; \kappa^\star+\eps) \ge \delta > 0.
\end{equation}
By Proposition
\ref{sec:proof-proposition-universality}, we can translate this bound to $\bX$ 
and conclude that $\widehat F_n(\bX; \kappa^\star+\eps) > \delta/2> 0$, 
which further implies that 
$\widehat \kappa_n(\bX) < \kappa^\star+\eps$ (cf. Eq.~\eqref{eq:equiv:F:kappa}), with probability 
tending to one.

\paragraph{Lower bound.}
For the lower bound, it is sufficient to assume $\kappa^\star > 0$.
For any $\eps  < \kappa^\star$, set 
 $\kappa = \kappa^\star-\eps> 0$ 
in Equation~\eqref{eq:lower_bound} to conclude that
\begin{equation}
	\widehat \kappa_n^{\bX} \ge \kappa^\star-\eps - C \cdot 
	(\widehat F_n(\bX; \kappa^\star-\eps ))^{1/2}.
\end{equation}

Proposition~\ref{proposition:max-margin-gaussian} and Eq.~\eqref{eq:equiv:F:kappa} indicate that 
$\lim_{ n\to \infty} \widehat F_n(\bG; \kappa^\star-\eps) = 0$.
Then by universality of $\widehat F_n(\cdot)$ in Proposition \ref{sec:proof-proposition-universality},
 this 
implies that
for any $\delta > 0$, $\widehat F_n(\bX; \kappa^\star-\eps) \le \delta$ holds 
with high probability, hence we conclude that $\widehat \kappa_n(\bX) > 
\kappa^\star-\eps- C\delta^{1/2}$ for any $\delta > 0$, with high probability.
\end{proof}

%
%

\subsection{Proof of Theorem~\ref{thm:test}: Universality of the test error}
\label{subsec:proof:test:error}


\subsubsection{Asymptotic equivalence of the test errors for delocalized estimators}
Call a sequence of vectors $\{\btheta_n\}_{n\in \N}$, possibly random, \emph{delocalized} if it satisfies
the following condition:
		\begin{equation}
			\lim_{n \to \infty} \frac{\norm{\btheta_n}_\infty}{\norm{\btheta_n}_2} = 0.
		\end{equation}
Uniformly over sequences $\{\btheta_n\}_{n\in \N}$ that are \emph{delocalized}, we show that the test 
errors---$R_n^{\bx}(\btheta_n)$ and $R_n^{\bg}(\btheta_n)$ defined in Eq.~\eqref{eq:def:test:error}---are asymptotically equal.  

\begin{proposition}
\label{proposition:asymptotic-equivalence-for-regular-estimators} 
For any sequence $\{\delta_n\}_{n\geq 1}$ such that $\delta_n\to 0$ as $n\to\infty$, we have the limit 
	\begin{equation}
		\sup_{\norm{\btheta_n}_{\infty}/\norm{\btheta_n}_2 \leq \delta_n} \big|R_n^{\bx}(\btheta_n) - R_n^{\bg}(\btheta_n)\big| \to 0.
	\end{equation} 
\end{proposition} 	
It is 
worth mentioning that this delocalization condition on the sequence is necessary to establish 
universality of the test errors; e.g., if $\btheta_n = e_i$ is a standard 
basis vector, then $|R_n^{\bx}(\btheta_n) - R_n^{\bg}(\btheta_n)|= \Omega_n(1)$ as long as 
the activation $\sigma$ is nonlinear.

Proposition~\ref{proposition:asymptotic-equivalence-for-regular-estimators} is achieved by proving that projections of the non-Gaussian features onto delocalized sequences are asymptotically Gaussian. From a technical 
viewpoint, Proposition~\ref{proposition:asymptotic-equivalence-for-regular-estimators}
 is an improvement over the existent results in the literature, where it does not assume the 
 rate of decay of $\delta_n$ (e.g., the analog of \cite[Theorem 2]{HuLu22} would require 
 $\delta_n=O\big((\textnormal{polylog }p)^{-1}\big)$). Such improvement is necessary
  for our proof of the universality of the test error. The proof of 
  Proposition~\ref{proposition:asymptotic-equivalence-for-regular-estimators}
is based on the second-order Poincar\'e inequality \citep{Chatterjee09}, which
 might be of independent interest. 

\subsubsection{Delocalization of the max-margin estimator}
As discussed in the previous section, delocalization of the max-margin estimator is crucial
to prove universality of the test error. We establish this property next. Recall that we have defined the separable regime as the regime $\gamma_1>\tau^\star_{\RF}(\gamma_2)$ for $\bD\in \{\bX_{\RF},\bG_{\RF}\}$ and $\gamma>\gamma^\star_{\ind}(\mu)$ for $\bD\in \{\bX_{\ind},\bG_{\ind}\}$.

\begin{proposition}[Delocalization of the Max-Margin Estimator for General Random Matrix Ensemble]
\label{proposition:delocalization-for-general-max-margin}
In the separable regime, the corresponding 
sequence of max-margin estimators $\widehat{\btheta}_{n, \MM}^{\bD}$ is
delocalized. Namely $\normsmall{\widehat{\btheta}_{ \MM}^{\bD}}_2 = 1$
holds for all $n$ and 
\begin{equation}
	 \normsmall{\widehat{\btheta}_{ \MM}^{\bD}}_\infty \pto 0\, .
\end{equation} 
\end{proposition} 
The proof of Proposition \ref{proposition:delocalization-for-general-max-margin} is 
presented in Section \ref{sec:delocalization:MM} and consists of two steps. First, 
we establish the delocalization in the case of Gaussian where 
$\bD \equiv \bG$ by leveraging tools from the theory Gaussian processes, 
namely the Gordon's inequaltiy \cite{Gordon88, thrampoulidis2015, montanari2019generalization}. 
Second, we extend non-Gaussian designs by proving that the $\ell_\infty$ norm of the estimator
is universal. The proof of the last statement is analogous to the 
proof of the universality of the margin described in Section
~\ref{sec:proof-of-universality-max-margin}. 



\subsubsection{Universality of geometric properties of the max-margin estimator}

Having established Propositions~\ref{proposition:asymptotic-equivalence-for-regular-estimators} and
 \ref{proposition:delocalization-for-general-max-margin}, it remains to show that $R_n^{\bg}(\bthetaMM^{\bX})=R_n^{\bg}(\bthetaMM^{\bG})+o_{\P}(1)$. To that end, we recall the basic
fact that, in the case where the features of the data are Gaussian, the test error depends only 
on the \emph{angle} between the estimator and the truth:  
 \begin{equation}
 	R_n^{\bg}(\btheta)= Q\left(\frac{\langle \btheta, \btheta^\star \rangle_{\bSigma_{\bg}}}
		{\norm{\btheta}_{\bSigma_{\bg}}\norm{\btheta^\star}_{\bSigma_{\bg}}}\right)\,,
 \end{equation} 
 where $Q(\, \cdot\, )$ is an explicit bounded Lipschitz function (cf. Eq.~\eqref{eq:def:Q}), 
 and we denoted $\norm{\btheta}_{\bSigma_{\bg}}:=(\btheta^{\sT}\bSigma_{\bg}\btheta)^{1/2}$
  and $\langle \btheta, \btheta^\star \rangle_{\bSigma_{\bg}}:= \btheta^{\sT}\bSigma_{\bg}\btheta^\star$, where $\bSigma_\bg$ is the covariance of the Gaussian equivalent features.
Thus, it suffices to establish the universality of the functionals $\pi_{\ell}: \R^p \mapsto \R, \ell=1,2,$ defined by 
\begin{equation}
\label{eqn:all-possible-G}
	\pi_1(\btheta) = \langle \btheta, \btheta^\star \rangle_{\bSigma_{\bg}},~~~
	\pi_2(\btheta) = \norm{\btheta}_{\bSigma_\bg}^2.~~~
\end{equation}
The asymptotics of these functionals under the Gaussian case is known~\citep{montanari2019generalization}. 
\begin{proposition}[{\cite[Proposition 6.4]{montanari2019generalization}}]\label{proposition:gaussian-angles}
Consider the max-margin estimator $\bthetaMM^{\bG}$ under the Gaussian design. Then, there exist $\pi_{\ell}^\star=\pi_{\ell,\RF}^\star(\gamma_1,\gamma_2)$ for $\bG=\bG_{\RF}$ and $\pi_{\ell}^\star=\pi_{\ell,\ind}^\star(\gamma;\mu)$ such that the following limit holds for $\ell = 1, 2$ in the separable regime: 
\begin{equation}
		\pi_{\ell}(\bthetaMM^{\bG})\pto \pi_\ell^\star\,.
\end{equation} 
\end{proposition}
We show that the limits in the last proposition hold for non-Gaussian designs.
In order to achieve this, we consider the tilted objective
\begin{equation}
\text{$H_\ell^{\bD}(s;\kappa)$:=}~~~
\begin{cases}
	\minimize_{\btheta}& \norm{\btheta}_2^2 + s\pi_{\ell}(\btheta)\\
	~~\subjectto~~& \norm{\btheta}_{\infty}\leq \delta_n \quad\textnormal{and}\quad y_i \langle \bd_i, \btheta\rangle \ge \kappa~~\forall 1\le i\le n. 
\end{cases}
\end{equation}
Here, $\delta_n=o_n(1)$ is such that $\normsmall{\widehat{\btheta}_{ \MM}^{\bD}}_\infty\leq \delta_n$ with probability tending to one (cf. Proposition \ref{proposition:delocalization-for-general-max-margin}). 
By definition of the max-margin problem, the solution $\widehat\btheta_{\ell}(\bD;s,\kappa)$ 
satisfies $\widehat\btheta_{\ell}(\bD;0,\widehat\kappa_n^\bD)=\bthetaMM^{\bD}$ for $\ell=1,2$. 
Thus, it follows from the envelope theorem that 
\begin{equation*}
    \frac{\partial H_{\ell}^{\bD}}{\partial s}(0;\widehat\kappa_n^\bD)=\pi_{\ell}(\bthetaMM^{\bD})\,.
\end{equation*}
By leveraging this fact, and establishing the universality of
 $H_\ell^{\bD}(s;\kappa)$ for every $\kappa>0$ and small enough $|s|$,
  we prove the following proposition.

\begin{proposition}
\label{proposition:universality-of-the-geometric-structure}
Fix $\ell \in \{1, 2\}$ and any Lipschitz function $\psi:\R\to \R$. Then for $(\bX,\bG)=(\bX_{\RF},\bG_{\RF})$ and $(\bX,\bG)=(\bX_{\ind},\bG_{\ind})$, we have
\begin{equation}
	\lim_{n\to \infty}  \left|\E[\psi( \pi_{\ell}(\hat{\btheta}_\MM^\bX))]
		- \E[\psi( \pi_{\ell}(\hat{\btheta}_\MM^\bG))]\right| = 0
\end{equation}
\end{proposition} 
Finally, we prove Theorem \ref{thm:test} based on Propositions \ref{proposition:asymptotic-equivalence-for-regular-estimators}, \ref{proposition:delocalization-for-general-max-margin}, \ref{proposition:gaussian-angles}, and \ref{proposition:universality-of-the-geometric-structure}.
\begin{proof}[Proof of Theorem \ref{thm:test}]
In the separable regime,
Theorem \ref{theorem:universality-of-the-margin}, shows that $\widehat\kappa_n^{\bD}>0$ with probability tending to one, thus $\big\|\bthetaMM^{\bD}\big\|_2=1$ holds w.h.p.. Further, by Proposition \ref{proposition:delocalization-for-general-max-margin}, $\big\|\bthetaMM^{\bD}\big\|_{\infty}\pto 0$ holds. Thus there exists $\delta_n=o_n(1)$ such that $\big\|\bthetaMM^{\bD}\big\|_{\infty}\leq \delta_n$ with probability tending to one as $n\to\infty$. Consequently, Proposition \ref{proposition:asymptotic-equivalence-for-regular-estimators} shows that as $n\to\infty$,
\begin{align}\label{eq:R:equiv:1}
	\left|R_n^{\bx}(\bthetaMM^{\bX})-R_n^{\bg}(\bthetaMM^{\bX})\right| \pto  0\,.
\end{align}
Moreover, by Proposition \ref{proposition:gaussian-angles}, $\pi_{\ell}(\bthetaMM^{\bG})\pto \pi_{\ell}^\star(\gamma_1,\gamma_2)$ holds. Applying Proposition \ref{proposition:universality-of-the-geometric-structure}, this implies that $\pi_{\ell}(\bthetaMM^{\bX})\pto \pi_{\ell}^\star(\gamma_1,\gamma_2)$ holds as well. Since $R_n^{\bg}(\bthetaMM^{\bD})$ is a function of $\big(\pi_1(\bthetaMM^{\bD}),\pi_2(\bthetaMM^{\bD})\big)$, it follows that
\begin{equation}\label{eq:R:equiv:2}
\left|R_n^{\bg}(\bthetaMM^{\bG})-R_n^{\bg}(\bthetaMM^{\bX})\right|\pto 0.
\end{equation}
Therefore, we have by \eqref{eq:R:equiv:1}, \eqref{eq:R:equiv:2}, and a triangle inequality that $\big|R_n^{\bx}(\bthetaMM^{\bX})-R_n^{\bG}(\bthetaMM^{\bG})\big|\pto 0$, which concludes the proof.
\end{proof}

\newcommand{\smin}{\sigma_{\rm min}}
\newcommand{\lmin}{\lambda_{\rm min}}
\newcommand{\lmax}{\lambda_{\rm max}}
\section{Restricted Strong Convexity: Proof of Lemma~\ref{lemma:restricted-strong-convexity}} 
\label{sec:proof-of-lemma-restricted-strong-convexity}

The goal in this section is to establish a generalization of Lemma~\ref{lemma:restricted-strong-convexity} that will also be useful for establishing the universality of test error in the proof of Theorem~\ref{thm:test}.
Let us begin by making the following definition:
for a vector $\bv \in \R^p$ and $\eps \in [0, \infty]$,  define
$$
    \norm{\bv}_{2,\eps}\equiv \max_{\btheta:\norm{\btheta}_2\leq 1\,,\, \norm{\btheta}_{\infty}\leq \eps}\big\langle \btheta,\bv\big\rangle\, .
$$
Note that $\norm{\cdot}_{2,\eps}$ defines a norm. Moreover, 
$\eps \mapsto  \norm{\bv}_{2,\eps}$ is increasing for every $\bv$, and 
$\norm{\bv}_{2,\infty} = \norm{\bv}_2$.
We state and prove the following.
\begin{lemma}[Restatement of Lemma~\ref{lemma:restricted-strong-convexity}]
\label{lemma:RSC_inf_norm}
Given any deterministic sequence $\{\eps_n\}_{n\geq 1}$ such that $\sqrt{n}\eps_n \to\infty$, there are 
constants $c, c_1, c_2 > 0$, and $N_0$ independent of $n$ such that for every $n \ge N_0$, the event
\begin{equation}
\label{eqn:proof-restricted-strong-convexity-eps}
	\norm{\bX^\sT \bu}_{2,\eps_n} \ge c_1\sqrt{n} \norm{\bu}_2 - c_2 \norm{\bu}_1~~\text{for all $\bu \in \R^n$}
\end{equation} 
holds with probability at least $1- o_n(1)$. 
\end{lemma}

As mentioned in the previous section, we show that $\widehat\btheta_\MM$ has a small $\ell^\infty$-norm and so it's sufficient for us to look at sets of $\btheta$ with this property for our constraint sets, so we modify Lemma~\ref{lemma:restricted-strong-convexity} to account for this additional constraint.

\label{sec:main-geometric-argument}
To establish this lemma, let us first fix constants $c_0,c_4,c_3 >0$ and define the events
\begin{equation}
    \Omega_1 := \left\{\norm{\bX}_\op \le c_0 \sqrt{n} \right\}
\end{equation}
and
\begin{equation}
    \Omega_2 := \left\{
\norm{\bX_S^\sT \bu_S }_{2,\epsilon_n} \ge c_4 \sqrt{n} \norm{\bu_S}_2  \textrm{ for every } S \subseteq [n] \textrm{ of size equal to } c_3 n
    \right\},
\end{equation}
where $\bX_S$ for a given $S\subseteq [n]$ is the $|S|\times p$ submatrix of $\bX$ whose rows are determined by the indices in $S$, and $\bu_S = (u_i)_{i\in S}$. 
We define and work on the event $\Omega_0 := \Omega_1 \cap \Omega_2$. Namely, we'll show that Eq.~\eqref{eqn:proof-restricted-strong-convexity-eps} holds on $\Omega_0$, which then reduces the proof of the lemma to establishing the following claim.
\begin{lemma}
\label{lemma:Omega_0_whp}
   There exists constants $c_0,c_3,c_4$ independent of $n$ such that $\P(\Omega_0) > 1 -o_n(1)$.
\end{lemma}
The majority of this section will be dedicated to proving this lemma (namely, that $\Omega_2$ holds with high probability).
For now, let us use it to complete the proof of Lemma~\ref{lemma:RSC_inf_norm}
\begin{proof}[Proof of Lemma~\ref{lemma:RSC_inf_norm}]
Let us take $c_1 = c_4 \sqrt{c_3}$ and $c_2 = c_0/\sqrt{c_3}$, and work on the event .
    Fix $\bu \in \R^n$ and order its coordinates as
\begin{equation*}
	|\bu_{i_1}| \ge |\bu_{i_2}| \ge \ldots |\bu_{i_n}|.
\end{equation*}
Let $S = \{i_1, i_2, \ldots, i_k\}$ where $k = c_3 n$. Note that we have
\begin{equation}
    \bX^\sT \bu = \bX_S^\sT \bu_S + \bX_{S^c}^\sT \bu_{S^c}.
\end{equation}
Then, by the triangle inequality,
\begin{equation}
\label{eqn:initial-triangle-inequality}
\begin{split}
	\norm{\bX^\sT \bu}_{2,\eps_n}  &\ge \norm{\bX_S^\sT \bu_S}_{2,\eps_n} - \norm{\bX_{S^c}^\sT \bu_{S^c}}_{2,\eps_n} \\
		&\ge c_4 \sqrt{n}\norm{\bu_{S}}_2 - c_0 \sqrt{n}\norm{\bu_{S^c}}_2.
\end{split} 
\end{equation} 
where the second inequality follows from the definition of $\Omega_1, \Omega_2$, and the fact that 
$\norm{\bv}_{2, \eps} \le \norm{\bv}$.

We lower and upper bound $\norm{\bu_{S}}_2$ and $\norm{\bu_{S^c}}_2$, respectively, as follows:
since $\bu_S$ contains coordinates with largest absolute entries, 
\begin{equation*}
	\norm{\bu_S}_2 \ge \sqrt{|S|/n} \norm{\bu}_2 = \sqrt{c_3} \norm{\bu}_2.
\end{equation*} 
Similarly, since $\bu_{S^c}$ contains coordinates with smallest absolute entries,
\begin{equation*}
	\norm{\bu_{S^c}}_2  \le \norm{\bu_{S^c}}_1^{1/2} \norm{\bu_{S^c}}_\infty^{1/2}
		\le \norm{\bu_{S^c}}_1^{1/2} \cdot (\norm{\bu_S}_1/|S|)^{1/2} = \norm{\bu}_1/\sqrt{c_3 n}.
\end{equation*} 
Substituting the above estimates 
into equation~\eqref{eqn:initial-triangle-inequality}, we obtain 
$
	\norm{\bX \bu}_{2,\eps_n} \ge c_4 \sqrt{c_3 n}\norm{\bu}_2 - c_0/\sqrt{c_3} \norm{\bu}_1. 
$
As the bound holds for all $\bu$, this proves the statement of the lemma on the event $\Omega_0$. Invoking Lemma~\ref{lemma:Omega_0_whp} then yields the desired claim.
\end{proof}

We now move on to proving Lemma~\ref{lemma:Omega_0_whp}. In the next section, we recall and establish some basic properties regarding matrix concentration for the feature distributions we consider. Then in Section~\ref{sec:bound-on-operator-norm}, we prove the lower bound on $\P(\Omega_1)$, which follows from standard matrix tail inequalities. Finally, in Section~\ref{sec:bound-on-restricted-eigenvalue}, we establish the lower bound on $\P(\Omega_2)$ completing the proof of Lemma~\ref{lemma:Omega_0_whp}.

\subsection{Concentration and subgaussianity properties of feature matrix}
We'll consider each of $\Omega_1$ and $\Omega_2$ separately. 
The main tool used in establishing a high probability bound is the 
following matrix concentration result for subgaussian random variables. 
For a review of subgaussianity and related definitions, see for example~\cite{Wainwright19, vershynin2018high}.
\begin{theorem}[Matrix Concentration:~Theorem 6.5 of \cite{Wainwright19}]
\label{theorem:matrix-concentration-covariance-operator}
Let $\bA \in \R^{p \times m}$ be a random matrix whose rows $\{\ba_i\}_{i\leq p}\in \R^m$ are i.i.d drawn from a 
subgaussian distribution with parameter $\nu^2 > 0$. Denote the covariance matrix by $\bSigma \equiv \E \ba_i\ba_i^{T}$. Then there exist absolute constants $c, C \in (0, \infty)$ such that for every $\delta > 0$
\begin{equation*}	
	\bigg\|\frac{1}{p}\bA^{T} \bA - \bSigma\bigg\|_\op \le 
		C\cdot \nu^2 \bigg(\sqrt{\frac{m}{p} + \delta}+ \frac{m}{p} + \delta\bigg)~~\text{holds with probability at least $1-e^{-c p \delta}$}.
\end{equation*}
\end{theorem}

We'll first need to establish the applicability of this result; namely, that the feature distributions in question are subgaussian with covariance whose spectrum is lower bounded.
Let us consider the two feature distributions separately.
\paragraph{Random Feature distribution, $\bX = \bX_\RF$:} 

Recall the random feature ensemble $\bX = \sigma (\bW \bZ) \in \R^{p \times n}$ where 
$\bW \in \R^{p \times d}$ has i.i.d. rows uniformly distributed on $\S^{d-1}(1)$ and 
$\bZ \in \R^{d \times n}$ has i.i.d. entries $\normal(0, 1)$. 


Lemma~\ref{lemma:crazy-lemma-random-feature-models} below shows that the random matrix 
ensemble $\bX$ has independent subgaussian rows when conditional on $\bZ$. Additionally,
after this conditioning, the subgaussianity parameter as well as the minimum  and maximum 
eigenvalue of the covariance of the row of $\bX$ are properly behaved with high probability. 
For clarity, Lemma~\ref{lemma:crazy-lemma-random-feature-models} is stated in a nonasymptotic setting
where the size of the matrix $n, p, d$ are fixed. We note that similar claims have been established for this model previously, but in slightly different settings~\cite{HuLu22}. 

\begin{lemma}[Properties of Random Feature Ensemble]
\label{lemma:crazy-lemma-random-feature-models}
There are universal constants $c, C > 0$ such that the following holds. 
For every $n, p, d$, the event 
	\begin{equation*}
		\Omega_{\bZ} = \left\{\norm{\bZ}_\op \le C(\sqrt{d}+\sqrt{p}),~~\max_i \norm{\bz_i}_2 \le C\sqrt{\log n}, 
			~~\max_{i, j\le n, i \neq j}|\langle \bz_i, \bz_j\rangle| \le C\log n\right\}
	\end{equation*}
satisfies the following three properties:  
\begin{enumerate}[(i)]
\item $\Omega_{\bZ}$ has probability at least $1-n^{-1} - e^{-cp}$. 	
\item Conditional on $\bZ$, the rows of $\bX$ are i.i.d.
	subgaussian with parameter $\nu^2(\bZ)$. Furthemore, on the event $\Omega_{\bZ}$,
	\begin{equation*}
		\nu(\bZ) \le C (1+\sqrt{p/d}) \norm{\sigma}_{{\rm lip}}.
	\end{equation*} 
\item Let $\bSigma_{\bZ}$ denote the conditional covariance of $\bx_i$ given $\bZ$.
	Then on the event $\Omega_{\bZ}$,
	\begin{equation*}
		\mu_2 - \delta_n 
			\le \lambda_{\min}(\bSigma_{\bZ}) \le \lambda_{\max}(\bSigma_{\bZ}) \le C\mu_1 \cdot (1+\sqrt{p/d}) + \mu_2 
			+ \delta_n
	\end{equation*}
	where $|\delta_n| \le \frac{C}{\sqrt{n}}((1+\sqrt{p/d})^4 + \log^4 n)$. 
\end{enumerate}
\end{lemma}

The proof of Lemma~\ref{lemma:crazy-lemma-random-feature-models} makes use of standard Lipschitz concentration results of Gaussian vectors and is deferred to Appendix~\ref{section:aux_lemmas}.

\paragraph{Independent Feature Ensemble, $\bX = \bX_{\ind}$:} 
By our assumptions, the feature matrix $\bX$ has independent rows, which are subgaussian with 
parameter at most $\nu^2$ by definition. Each row has covariance 
whose eigenvalue is lower bounded by $\lambda_{\min}(\bSigma_{{\sf ind}})$.

\subsection{Lower bound on $\mathbb{P}(\Omega_1)$}
\label{sec:bound-on-operator-norm}
The lower bound on $\P(\Omega_1)$ follows from standard results given the subgaussian property we established for the features. To be formal, we state the result in the following lemma.

\begin{lemma}[Operator Norm]
\label{lemma:operator-norm-bound-on-Z}
Let $p/d= \kappa$. For some universal constant $c > 0$, and constant $c_0 > 0$ depending only on 
$\kappa, \norm{\sigma}_{{\rm lip}}, \mu_1, \mu_2$, with probability at least $1-e^{-cp}$ the matrix $\bX$ 
satisfies
\begin{equation*}
	\norm{\bX}_\op \le c_0 (\sqrt{n} + \sqrt{p})
\end{equation*}
for both $\bX = \bX_{\ind}$ and $\bX = \bX_\RF$.
\end{lemma}


\newcommand{\lip}{{\rm lip}}

\begin{proof}
For $\bX = \bX_\RF$, let $\Omega_Z$ denote the event in Lemma~\ref{lemma:crazy-lemma-random-feature-models}.
By Theorem~\ref{theorem:matrix-concentration-covariance-operator}, there are universal constants $c, C > 0$ such that
\begin{equation*}\label{eq:approx:D}
    \bigg\|\frac{1}{p}\bX^{T}\bX-\bSigma_{\bZ}\bigg\|_{\op}\leq  C \nu^2(\bZ)
    	\bigg(\sqrt{\frac{n}{p} + 1}+ \frac{n}{p} + 1\bigg)~~\text{holds with probability at least $1-e^{-c p}$}, 
\end{equation*} 
where on $\Omega_\bZ$
\begin{equation*}
	\nu(\bZ) \le C \left(1+\sqrt{\kappa}\right) \norm{\sigma}_{{\rm lip}}~~\text{and}~~
	 \lambda_{\max}(\bSigma_{\bZ}) \le C \mu_1 \left(1+\sqrt{\kappa}\right) + \mu_2 
			+ \delta_n
\end{equation*}
for some $\delta_n \le \frac{C}{\sqrt{n}} ((1+\sqrt{\kappa})^4 + \log^4 n)$ which tends to $0$ as $n \to \infty$. 
This result implies Lemma \ref{lemma:operator-norm-bound-on-Z} for the random features distribution.

Meanwhile, for $\bX=\bX_\RF$, $\bX$ has independent rows with subgaussianity parameter $\nu^2$ where $\nu < \infty$.
Lemma \ref{lemma:operator-norm-bound-on-Z} then
follows immediately  from Theorem~\ref{theorem:matrix-concentration-covariance-operator}.
 
\end{proof} 

\subsection{Lower bound on $\P(\Omega_2)$} 
\label{sec:bound-on-restricted-eigenvalue}
Our goal here is to establish the following lemma.

\newcommand{\norms}[1]{\|#1\|}

\begin{lemma}
\label{lemma:restricted-minimum-singular-value}
Given any $\eps_n>0$ with $\sqrt{n}\eps_n\to\infty$ as $n\to\infty$, there are constants $c_3, c_4 > 0$
depending only on $\sgamma_1, \sgamma_2, \mu_1, \mu_2$, such that the event
\begin{equation*}
	\min_{\bv:\norm{\bu}_2=1}\norm{\bX_{S}\bu}_{2,\eps_n}\geq c_4 \sqrt{n}~~\text{for every subset 
		$S \subseteq [n]$, $|S| = c_3 n$}
\end{equation*} 
holds with probability tending to one as $n \to \infty$.
\end{lemma}

Our proof is based on two general statements. The first is a lower bound on $\norm{\cdot}_{2, \eps}$ (in Section~\ref{sec:lower-bound-norm-two-eps}), and the second is a generalization of Kashin's 
theorem, providing sufficient conditions on a random matrix $\bA$ with independent rows such that
with high probability $\norm{\bv}_1 \approx \frac{1}{\sqrt{p}} \norm{\bv}_2$ holds uniformly over 
$\bv$ in the column space of $\bA_S$ for proper size $S$ (in Section~\ref{sec:equivalence-between-ell-one-ell-two-random-subspaces}).
In Section~\ref{section:proof_rmsv} below, we establish Lemma~\ref{lemma:restricted-minimum-singular-value} using these.

\subsubsection{Lower bounding $\norm{\cdot}_{2, \eps}$}
\label{sec:lower-bound-norm-two-eps}
\begin{lemma}
\label{lemma:elementary}
For every $\bv \in \R^p, \eps > 0$, there is 
$\norm{\bv}_{2, \eps} \ge \norm{\bv}_2 \cdot \left(\frac{\norm{\bv}_1}{\sqrt{p}\norm{\bv}_2} - \frac{1}{\sqrt{p}\eps}\right)_+^2$.
\end{lemma}
\begin{proof}
Write $\bv=(v_i)_{i\leq p} \in \R^p$.  Since every norm is homogeneous, 
it suffices to prove the lemma when $\norm{\bv}_2 = 1$.

Let $T(\eps) = \{i: |v_i| \le \eps\}$. Then $\norm{\bv}_{2, \eps}  \ge  \norms{\bv_{T(\eps)}}_2^2$.
Indeed, $\norm{\bv}_{2, \eps} \ge \langle\bv, \btheta \rangle$ where $\btheta_i= v_i$
for $i \in T(\eps)$ and $\btheta_i = 0$ for $i \in T^c(\eps)$.
Note $ \norms{\bv_{T(\eps)}}_2 \ge \frac{1}{\sqrt{p}}  \norms{\bv_{T(\eps)}}_1$ where
$\norms{\bv_{T(\eps)}}_1 = \norm{\bv}_1 - \norms{\bv_{T^c(\eps)}}_1$ by Cauchy-Schwartz. The result follows 
since $\norms{\bv_{T^c(\eps)}}_1 \le \norms{\bv_{T^c(\eps)}}_2^2/\eps$ by Markov's inequality, and
$\norms{\bv_{T^c(\eps)}}_2\le \norm{\bv}_2 = 1$. 
\end{proof}

\subsubsection{Equivalence between $\ell_1$ and $\ell_2$ norm in random subspaces} 
\label{sec:equivalence-between-ell-one-ell-two-random-subspaces}
Below is the main result of the Section. 
\begin{lemma}
\label{lemma:equivalence-between-l1-and-l2-norm}
Let $\bA\in \R^{p\times n}$ be a random matrix whose rows $\{\ba_i\}_{i\leq p}\in \R^n$ are i.i.d. subgaussian
with parameter $\nu^2$. Suppose, for some $\kappa > 0$, the covariance $\bSigma_J \equiv \E [\ba_J \ba_J^{T}]$ 
satisfies $\lmin(\bSigma_J)\ge \kappa \nu^2$ for every $J$ where
$|J| \ge N$. Then there's a constant $c' = c'(\kappa)$ depending
only on $\kappa$ such that 
for $m = c' p$, given any subset $S$ with $N \le |S| \le m$, 
\begin{equation}
\label{eqn:equivalence-between-l1-l2}
		 \frac{1}{\sqrt{p}}\norm{\bA_S u}_2 \ge \norm{\bA_S u}_1 \ge  \frac{c}{\sqrt{p}} \norm{\bA_S u}_2
		 	~~\text{for every $u \in \R^{|S|}$}
\end{equation}
holds with probability at least $1-e^{-cp}$. Here $c > 0$ is a universal constant. 
\end{lemma} 

The first inequality in equation~\eqref{eqn:equivalence-between-l1-l2} always holds by Cauchy-Schwartz. 
The second inequality that holds uniformly over $u \in \R^{cn}$ is what we will need to show. 
This follows from two results stated below;
Lemma~\ref{lemma:lower-bound-on-l-1-norm} and
Lemma~\ref{lemma:upper-bound-on-l2-norm}.

\begin{lemma}
\label{lemma:lower-bound-on-l-1-norm}
Let $\bA\in \R^{p\times m}$ be a random matrix whose rows $\{\ba_i\}_{i\leq p}\in \R^m$ are i.i.d. subgaussian
with parameter $\nu^2$. Suppose the covariance $\bSigma\equiv \E [\ba \ba^{T}]$ satisfies 
$\lmin(\bSigma)\ge \kappa \nu^2$ where $\kappa > 0$. Then, there are constants $C_1, C_2 > 0$ depending only on 
$\kappa$, and a universal constant $c > 0$, such that with probability at least $1-e^{-cp}$, 
\begin{equation}
	\norm{\bA \bv}_{1} \ge p \nu (C_1 - C_2 \sqrt{m/p})_+ \cdot \norm{\bv}_2~~\text{holds for every $\bv$}.
\end{equation}
\end{lemma}
\begin{proof}
Let $\sigma > 0$ be determined later. Note by Markov's inequality, for every vector $\bv$
\begin{equation*}
	\norm{\bA \bv}_{1}
		\ge \sigma \cdot \sum_{j} \mathbf{1}_{|\langle \ba_j, \bv\rangle|\ge \sigma}. 
\end{equation*}
A standard VC argument~\cite[Section 4]{Wainwright19} yields that, for some universal constant $C, c > 0$: 
\begin{equation*}
	\sup_{\bv} |\frac{1}{p}\sum_{j} \mathbf{1}_{|\langle \ba_j, \bv\rangle|\ge \sigma}
		- \P(|\langle \ba, \bv\rangle| \ge \sigma)| \le C \sqrt{m/p}
\end{equation*}
with probability at least $1- e^{-cp}$. So with triangle inequality, this implies with at least 
the same probability we have 
\begin{equation*}
	 \norm{\bA \bv}_{1} \ge 
		p\sigma \cdot (\min_{\bv: \norm{\bv}_2 = 1} \P(|\langle \ba, \bv\rangle| \ge \sigma) - C \sqrt{m/p})_+ 
			\cdot \norm{\bv}_2~~\text{for every $\bv$}.
\end{equation*} 
Below we lower bound $\P(|\langle \ba, \bv\rangle| \ge \sigma)|$ for every $\sigma > 0$ and $\bv$ with 
$\norm{\bv}_2 = 1$. 

Write $z = \langle \ba, \bv \rangle$ where  $\norm{\bv}_2 = 1$. Note $\E[z^2] \ge  \kappa \nu^2$, and 
$\E[z^4] \le C\nu^4$ for a universal constant $C < \infty$.  Hence, 
$\E[z^2 \mathbf{1}_{|z| \ge \sigma}] \le \frac{C \nu^4}{\sigma^2}$ by Markov's inequality,
which yields $\E[z^2\mathbf{1}_{|z| \le \sigma}] \ge \kappa \nu^2- \frac{C \nu^4}{\sigma^2}$. 
Hence $\P(|z| \le \sigma) \ge \frac{\kappa \nu^2}{\sigma^2} - \frac{C \nu^4}{\sigma^4}$. 
As a result, this means there are $C_1, C_2 > 0$ depending only on $\kappa$ such that 
$\P(|\langle \ba, \bv\rangle| \ge C_1 \nu) \ge C_2$ .

This proves that with probability at least $1- e^{-cp}$,
\begin{equation*}
\text{$\norm{\bA \bv}_{1} \ge C_1 p \nu (C_2 - C \sqrt{m/p})_+ \cdot \norm{\bv}_2$ holds for every 
$\bv$.}
\end{equation*}

\end{proof}

\begin{lemma}
\label{lemma:upper-bound-on-l2-norm}
Let $\bA\in \R^{p\times m}$ be a random matrix whose rows $\{\ba_i\}_{i\leq p}\in \R^m$ are i.i.d. subgaussian
with parameter $\nu^2$. Then, there are universal constants $c, C > 0$ such that with probability at least $1-e^{-cp}$
\begin{equation*}
	\norm{\bA \bv}_2 \le C \sqrt{p} \nu ( 1+ \sqrt{m/p}) \cdot \norm{\bv}_2~~\text{holds for every $\bv$}. 
\end{equation*}
\end{lemma}
\begin{proof}
This follows from results in 
nonasymptotic random matrix theory (Theorem~\ref{theorem:matrix-concentration-covariance-operator}).
\end{proof}

\subsubsection{Proof of Lemma~\ref{lemma:restricted-minimum-singular-value}}
\label{section:proof_rmsv}

By Lemma~\ref{lemma:elementary}, we know 
\begin{equation}
\label{eqn:initial-lower-bound}
	\norm{\bX_S \bu}_{2, \eps_n} \ge \norm{\bX_S \bu}_2 \cdot \left(\frac{\norm{\bX_S \bu}_1}
		{\sqrt{p}\norm{\bX_S \bu}_2} - \frac{1}{\sqrt{p} \eps_n}\right)_+^2
\end{equation}
We shall lower bound the RHS by $c_4 \sqrt{p}$, 
uniformly over all indices $S$ with cardinality $c_3 n$, and vector $u \in R^{|S|}$, with 
probability tending to one.
Here $c_3, c_4 > 0$ are constants to be determined in the rest of the proof. 

We divide our proof into two cases, depending on the distribution of the feature matrix $\bX$. Let $\kappa = p/d, \kappa'=n/p$. 

\paragraph{Random Feature Ensemble, $\bX = \bX_\RF$:}
By Lemma~\ref{lemma:crazy-lemma-random-feature-models}, we can choose a proper event 
	\begin{equation*}
		\Omega_{\bZ} = \left\{\norm{\bZ}_\op \le C(\sqrt{d}+\sqrt{p}),~~\max_i \norm{\bz_i}_2 \le C\sqrt{\log n}, 
			~~\max_{i, j\le n, i \neq j}|\langle \bz_i, \bz_j\rangle| \le C\log n\right\}
	\end{equation*}
which holds with probability at least $1-n^{-1} - e^{-cp}$, where $c, C > 0$ are universal constants. 
There exists an $N_0$ such that for every index $S$ such that $|S| \ge N_0$, we have 
\begin{enumerate}	
\item conditional on $\bZ$, the rows of $\bX_S$ are i.i.d.
	subgaussian with parameter $\nu^2(\bZ)$. Furthemore, on the event $\Omega_{\bZ}$,
	\begin{equation*}
		\nu(\bZ) \le C (1+\sqrt{\kappa}) \norm{\sigma}_{{\rm lip}},
	\end{equation*} 
 and
\item letting $\bSigma_{\bZ, S}$ denote the conditional covariance $\bX_S$ given $\bZ$, where $\bX$ stands for a row of $\bX$, on the event $\Omega_{\bZ}$,
	\begin{equation*}
		\frac{1}{2} \mu_2 
			\le \lambda_{\min}(\bSigma_{\bZ, S}) \le \lambda_{\max}(\bSigma_{\bZ, S}) \le C\mu_1 \cdot (1+\sqrt{\kappa}) + 2\mu_2. 
	\end{equation*}
\end{enumerate}

By Lemma~\ref{lemma:equivalence-between-l1-and-l2-norm} and
Theorem~\ref{theorem:matrix-concentration-covariance-operator}, 
there is a constant $c_3$ depending only on 
$\mu_1, \mu_2, \kappa, \kappa'$ such that the following two bounds hold 
for every given individual subset $S$ where $N_0 \le |S| \le c_3 n$ on the event $\Omega_{\bZ}$.
The first is the bound below holds with probability at least $1-e^{-cp}$ for some universal constant $c > 0$: 
\begin{equation}
\label{eqn:key-bound-high-probability-one}	
	\inf_{\bu: \bu \in \R^{|S|}} \frac{\norm{\bX_S \bu}_1}{\sqrt{p}\norm{\bX_S \bu}_2} \ge c_3,
\end{equation}
and the second is the bound below holds with probability at least $1-e^{-cp}$ for some $c > 0$
depending only on $\mu_2$: 
\begin{equation}
\label{eqn:key-bound-high-probability-two}
	\inf_{\bu: \bu \in \R^{|S|}, \norm{\bu}_2 = 1} \norm{\bX_S \bu}_2 \ge \frac{\mu_2}{4}\sqrt{p}. 
\end{equation} 

We make these lower bounds uniform over $|S| = c_3 n$ when $c_3$ is sufficiently small. Note the 
number of sets $S$ with cardinality $|S| = c_3 n$ is bounded by $\exp(n H(c_3))$, where 
$H(x) = -x\log_2 (x) - (1-x) \log_2 (1-x)$ is the entropy. Note $H(c_3) \to 0$ as $c_3 \to 0$.
This means that we can take $c_3, c$ depending 
on $\mu_1, \mu_2, \kappa, \kappa'$ such that with probability at least $1-e^{-cp}$, 
the above bounds~\eqref{eqn:key-bound-high-probability-one}
and~\eqref{eqn:key-bound-high-probability-two} hold uniform over set $S$ with cardinality $|S| = c_3 n$
on the event $\Omega_{\bZ}$.

Finally, since $\sqrt{p} \eps_n \to 0$ as $n \to \infty$, we may assume $\frac{1}{\sqrt{p}\eps_n} \le \frac{c_3}{2}$
for $n \ge N_0$.

These uniform bounds, in conjunction with~\eqref{eqn:initial-lower-bound}, allow us to conclude that for every 
$n \ge N_0$, 
\begin{equation*}
	\inf_{S: |S| = c_3 n}\inf_{\bu: \bu \in \R^{|S|}, \norm{\bu}_2 = 1}\norm{\bX_S \bu}_{2, \eps_n} \ge c_4 \sqrt{p}
\end{equation*}
holds with probability at least $1-e^{-cp}$ on the event $\Omega_{\bZ}$. 
In the above, $c, c_3, c_4, N_0$ depend only on $\mu_1, \mu_2, \kappa, \kappa'$. The result now 
follows since $\Omega_{\bZ}$ holds with probability tending to one, as $n \to \infty$. 

\paragraph{Independent Feature Ensemble, $\bX = \bX_{\ind}$: } 

We start by noticing the very basic property of $\bX_S$ for every index subset $S$. 
Note the rows $\bX_S$ are independent and subgaussian with parameter at most $\nu^2$, 
with the eigenvalue of the covariance lower bounded by $\gamma > 0$. Here, $\nu, \gamma$
are constants independent of $n, d, p$.

By Lemma~\ref{lemma:equivalence-between-l1-and-l2-norm} and
Theorem~\ref{theorem:matrix-concentration-covariance-operator}, 
there are constants $c_3, c$ depending only on 
$\nu, \gamma, \kappa, \kappa'$ such that  
for every given individual subset $S$ where $|S| \le c_3 n$,
the following bounds hold with probability at least $1-e^{-cp}$. The first bound is 
\begin{equation}
\label{eqn:key-bound-high-probability-one-ind}	
	\inf_{\bu: \bu \in \R^{|S|}} \frac{\norm{\bX_S \bu}_1}{\sqrt{p}\norm{\bX_S \bu}_2} \ge c_3,
\end{equation}
and the second bound is: 
\begin{equation}
\label{eqn:key-bound-high-probability-two-ind}
	\inf_{\bu: \bu \in \R^{|S|}, \norm{\bu}_2 = 1} \norm{\bX_S \bu}_2 \ge \frac{\gamma}{2}\sqrt{p}. 
\end{equation} 
The remainder of the proof is very similar to that of the random feature ensemble. 

We make these lower bounds uniform over $|S| = c_3 n$ for sufficiently small $c_3$. This 
is doable, as the number of  sets $S$ with cardinality $|S| = c_3 n$ is bounded by $\exp(n H(c_3))$
where $H(c_3) \to 0$ as $c_3 \to 0$. Thus, we can say the above 
bounds~\eqref{eqn:key-bound-high-probability-one-ind}
and~\eqref{eqn:key-bound-high-probability-two-ind} hold uniform over set $S$ with probability 
tending to one. 

Finally, note $\sqrt{p} \eps_n \to 0$. This allows us to conclude that for some $c_3, c_4$: 
\begin{equation*}
	\inf_{S: |S| = c_3 n}\inf_{\bu: \bu \in \R^{|S|}, \norm{\bu}_2 = 1}\norm{\bX_S \bu}_{2, \eps_n} \ge c_4 \sqrt{p}
\end{equation*}
holds with probability tending to one as $n \to \infty$, where, $c_3, c_4$ depend only on 
$\gamma, \nu, \kappa, \kappa'$. The proof is thus complete.

\newpage 
\section{Universality of the ERM-like problem}
\label{sec:ERM}

The goal of this section is to prove the universality of the minimum of the ERM-like problem defined in Eq.\eqref{eq:def_F_n} over appropriate ``perturbed'' constraint sets, namely, Proposition~\ref{prop:univ_fixed_kappa} below.
The need for these constraint sets comes in when proving the universality of the test error, as briefly outlined in Section~\ref{subsec:proof:test:error}.
Before precisely stating the result of this section, let us make some definitions.
Let
\begin{equation}\label{eq:def:pi}
  \pi_0(\btheta) := \norm{\btheta}_4^4,
  \quad 
  \pi_1(\btheta) := \btheta^\sT\bSigma_\bg \btheta^\star,
  \quad 
  \pi_2(\btheta) := \norm{\btheta}_{\bSigma_\bg}^2
\end{equation}
where $\bSigma_\bg$ is the covariance of the Gaussian equivalent.
We will be interested in constraint sets of the form
\begin{equation}
  \{\btheta : \norm{\btheta}_2^2 + s\pi_l(\btheta)\le b\}
\end{equation}
for values of $s$ near $0$. 
However, our task will be easier if we ensure that our optimization problem of interest is convex (note that the set above with $l=0$ is not convex for small negative values of $s$).
To remedy this, we define
\begin{equation}
    \cC_0(s,b) := \{\btheta : \norm{\btheta}_2^2 + s \pi_0(\btheta) \le b, \norm{\btheta}_2^2 \le 2 \}.
\end{equation}
Defining $\os := 1/12 \wedge (2\oslambda)^{-1}$,
one can check that if $\norm{\bSigma_\bg}_\op \le \oslambda$,
then the set $C_0(s,b)$ is convex for $|s| \le \os$.
Meanwhile, for $l\in\{1,2\}$, define the constraint sets
\begin{equation}
    \label{def:cl}
    \cC_l(s,b) := \{
    \btheta : \norm{\btheta}_2^2 + s\pi_l(\btheta) \le b, \norm{\btheta}_\infty \le \epsilon_n
    \}
\end{equation}
to restrict the $l^\infty$ norm of $\btheta$.
Now define the perturbed ERM-like minimization problems
\begin{equation} 
F^\sur_{l}(\bD;\kappa, s ,b) := \min_{\btheta \in \cC_l(s,b)} 
\frac1n \sum_{i=1}^n \left(\kappa - y_i\btheta^\sT\bd_i \right)_+^2
\end{equation}
for $l\in\{0,1,2\}$, $s \in\R$  and $b > 0$. Under these definitions, we will prove the following claim.
\begin{proposition}
\label{prop:univ_fixed_kappa}

Let $\os = 1/12 \wedge (2\oslambda)^{-1}$ where 
$ \norm{\bSigma_\bg}_\op  \le \oslambda $ w.h.p.
Fix $|s| \le \os, b > 0,\kappa >0$. 
Let $\cC_l \equiv \cC_l(s,b)$ be defined as above
for $l\in\{0,1,2\}$.
We have for all bounded Lipschitz functions
$\varphi$,
\begin{equation}
\label{eq:main_univ_sep_eq}
\lim_{n\to\infty}
\left|\E\left[\varphi\left( F^\sur_{l}(\bX;\kappa, s, b)\right) \right]
-
\E\left[\varphi\left(
 F^\sur_{l}(\bG;\kappa, s, b)\right) \right]
\right| = 0
\end{equation}
as $n\to\infty$ with $p/n \to \sgamma$, for both $(\bX,\bG) = (\bX_\RF,\bG_\RF)$ and 
$(\bX,\bG) = (\bX_{\ind}, \bG_{\ind})$.
\end{proposition}

\subsection{Proof of Proposition~\ref{prop:univ_fixed_kappa}}
We will establish this proposition as a corollary of Theorem 1 of~\cite{MontanariSa22} by reducing it to their assumptions. Let us begin by recalling this result.
\begin{theorem}[Adaptation of Theorem~1,~\cite{MontanariSa22}]
\label{thm:MS22}
Let $k \equiv k_n$ be a sequence of integers such that $k_n/n \to c\in(0,\infty)$ as $n\to\infty$.
Consider $(\bxi_i,y_i)\in \R^{k}\times \R$ i.i.d. with 
\begin{equation}
    \P(y_i = +1 |\bxi_i^\sT \bzeta^\star) = g( \bxi_i^\sT \bzeta^\star)
\end{equation}
 for $g:\R\to\R$ a pseudo Lipschitz function of order 1 and
 $\bxi$ a $O(1)$-subgaussian random variable satisfying
\begin{equation}
\label{eq:pointwise_gaussian_ass}
   \lim_{k\to\infty} \sup_{\btheta \in \cS} |\E[\varphi(\btheta^\sT \bxi) ] - \E[\varphi(\btheta^\sT \bgamma)]| = 0
\end{equation}
for any Lipschitz function $\varphi:\R\to\R$ and some set $\cS\subseteq \R^k$, where $\bgamma \sim \cN(0,\bSigma_\bgamma)$.
Consider a nonnegative loss $\ell :\R\times \R \to \R$ satisfying
\begin{equation}
\label{eq:PL_ass}
    |\ell(v, y) - \ell(\widetilde v, y)| \le C (1 + |y|) |v -\widetilde v|,\quad
    |\ell(v, y) - \ell(v,\widetilde  y)| \le C (1 + |v|) |y - \widetilde y|
\end{equation}
and any locally Lipschitz regularizer $r: \R^k \to\R$.
If $\bzeta^\star \in \cS$, then the quantity
\begin{equation}
   F(\bxi_1,\dots,\bxi_n) := \min_{\bzeta \in \cC } \sum_{i=1}^n \ell(\bxi_i^\sT \bzeta, y_i) + r(\bzeta)
\end{equation}
is universal for any compact subset $\cC \subseteq \cS$ in the sense that for any Lipschitz function $\varphi$,
\begin{equation}
    \lim_{n \to\infty} \left| \E[\varphi( F(\bxi_1, \dots,\bxi_n))] - \E[\varphi( F(\bgamma_1, \dots,\bgamma_n))] \right| = 0.
\end{equation}
\end{theorem}

The main step in applying this theorem to our setting is approximating the problem in $F^\sur_{l}$ by a minimzation over a Lipschitz loss.
To do so,
fix the parameters $\kappa,s,b$ from the statement of Proposition~\ref{prop:univ_fixed_kappa} in their respective domains in what follows, and let us suppress the dependence on these parameters in the notation.
For $\rho \in (0, \infty]$, define the \emph{Lipschitz} huber loss function
\begin{equation}
\label{eq:huber_def}
h_\rho(t) := \begin{cases}
                  t^2 & |t| < \rho\\
                  \rho(2|t| - \rho)& |t| \ge \rho.
                 \end{cases}
\end{equation}
For a feature matrix $\bD\in\R^{n\times p}$ with rows $\bd_i$, define the associated huber objective for a fixed $\rho$ in its domain and $\lambda >0$
\begin{equation}
\label{eq:def_hub_obj}
	L^\hub_\rho (\btheta; \bD, \lambda) := \frac1n \sum_{i=1}^n h_\rho\left((\kappa - y_i \bd_i^\sT \btheta)_+ \right) 
+ \frac\lambda2\norm{\btheta}_2^2,
\quad
F^\hub_{\rho,l}(\bD; \lambda,) := \min_{\btheta\in\cC_l} L_\rho^\hub(\btheta;\bD, \lambda).
\end{equation}
In Section~\ref{section:huber_reduction}
below, we prove the following lemma that allows us to reduce the universality of $F^\sur_l$ to $F_{\rho,l}^\hub$ for a fixed $\rho$.
\begin{lemma}
\label{lemma:huber}
Fix $l\in\{0,1,2\}$, $\lambda > 0$ and $\kappa,b,s$ in their respective domains (as in the statement of Proposition~\ref{prop:univ_fixed_kappa}). For any bounded Lipschitz function $\varphi$, we have  
\begin{equation}
\lim_{\lambda\to0}\lim_{\rho \to \infty} \sup_{n \in \N} \left|
\E\left[
\varphi\left(F_{\rho,l}^\hub(\bD; \lambda)\right)
\right] 
-
\E\left[
\varphi\left(  F^\sur_{l}(\bD)\right)
\right] 
 \right| =0.
\end{equation}
\end{lemma}

To complete the proof, we specialize to each of the distribution classes considered in Sections~\ref{section:ass_RF} and~\ref{section:ass_ind}.
\paragraph{Random features distribution}
In the case where $(\bX,\bG)=(\bX_\RF,\bG_\RF)$,
let $\Omega_1 := \{\norm{\bW}_\op \le \overline{W}\}$ where $\overline W$ is chosen as in Lemma~\ref{lem:basic:op:norm:RF} so that $\Omega_1$ holds with high probability.
Take 
\begin{equation}
    \bxi_i = (\bx_i^\sT,\bz_i^\sT)^\sT \in\R^{p+d}, \quad \bgamma_i = (\bg_i^\sT,\bz_i^\sT)^\sT\in \R^{p+d}.
\end{equation}
Let $\bXi,\bGamma\in\R^{n\times (p+d)}$ be the feature matrices with rows $\bxi_i,\bgamma_i$ respectively.
We will apply Theorem~\ref{thm:MS22} to conclude universality of
\begin{equation} 
\label{def:F_aug}
F^\hub_{\rho,l}(\bXi;\lambda) := \min_{\bzeta \in \widetilde\cC_l} 
\frac1n \sum_{i=1}^n h_\rho\left( (\kappa - y_i\bzeta^\sT\bxi_i )_+\right)^2 + \frac\lambda2 \norm{\bzeta}_2^2
\end{equation}
and $F^\hub_{\rho,l}(\bGamma)$ (defined similarly)
\emph{conditionally} on $\bW\in\Omega_1$, where the constraint sets are modified as
\begin{equation}
\widetilde\cC_l := \{\bzeta\in\R^{p+d} : \bzeta = (\btheta^\sT, {\mathbf 0}^\sT)^\sT, \btheta\in\cC_l\}  \subseteq \R^{p+d}
\end{equation}
for $l\in\{1,2\}$ and 
\begin{equation}
 \widetilde \cC_{0} :=
 \{\bzeta\in\R^{p+d} : \bzeta = (\btheta^\sT, {\mathbf 0}^\sT)^\sT, \btheta\in\cC_0\} \cap \{\bzeta \in\R^{p+d}, \norm{\bzeta}_\infty \le\delta'_n \},
\end{equation}
and $y_i$ are defined by
\begin{equation}
    \P(y_i = +1 | \bxi_i) = f(\bxi_i^\sT \bzeta^\star)
\end{equation}
for $\bzeta^\star := (\bzero^\sT ,\bbeta^{\star\sT})^\sT$.

To do so, first note that the huber loss $h_\rho$ satisfies the psuedo Lipschitz assumption in~Eq.~\eqref{eq:PL_ass}.
Furthermore, $\bxi_i$ are $i.i.d, O(1)$ subgaussian random vectors conditionally on $\bW\in\Omega_1$.

Finally, for any positive $\delta_n \to 0$ and positive constant $C$, 
Lemma~\ref{lemma:pointwise-gaussianity} in Appendix~\ref{section:aux_lemmas} asserts that condition~\ref{eq:pointwise_gaussian_ass} holds for the set
 \begin{equation}
    \cS := \{\bzeta \in\R^{p+d} : \bzeta^\sT = (\btheta^\sT ,\bbeta^\sT), \norm{\btheta}_\infty \le \delta_n, \norm{\bzeta}_2 \le C\},
 \end{equation}
conditionally on $\bW\in\Omega_1$.
Now observe that for some choices of $C,\delta_n$, we have that $\zeta^\star \in\cS$ and that $\widetilde\cC_l$ are compact subsets of $\cS$ for all $l\in\{0,1,2\}$. By Theorem~\ref{thm:MS22}, we therefore have that
\begin{equation}
 \lim_{n\to\infty} \left| \E\left[\varphi\left(F_{\rho,l}^\hub(\bXi;\lambda) \right)\Big| \bW \right]
- 
\E\left[\varphi\left(F_{\rho,l}^\hub(\bGamma;\lambda)\right)\Big| \bW\right]
\right| = 0
\end{equation}
for any $\bW \in\Omega_1$, and any bounded Lipschitz function $\varphi$.
Since $\varphi$ is bounded and $\Omega_1$ holds with high probability, we deduce by dominated convergence that
\begin{equation}
 \lim_{n\to\infty} 
 \left| \E\left[\varphi\left( F_{\rho,l}^\hub(\bXi;\lambda) \right)\right]
- 
\E\left[\varphi\left(F_{\rho,l}^\hub(\bGamma;\lambda) \right)\right]
\right| = 0.
\end{equation}

Now note that for $l\in\{1,2\}$, 
$F_{\rho,l}^\hub(\bXi;\lambda)= F_{\rho,l}^\hub(\bX;\lambda)$ directly from the definitions (and similarly for the problems with $\bGamma,\bG$). Meanwhile, to take care of the added $\ell^\infty$ norm constraint in $\widetilde \cC_0$, Lemma~\ref{lemma:inf_norm_ERM} shows that
the minimizers $\widehat\btheta^\bX, \widehat\btheta^\bG$ of $F_{\rho,0}^\hub(\bX), F_{\rho,0}^\hub(\bG)$, respectively, satisfy
$\|{\widehat\btheta}\|_\infty \stackrel{p}\to 0$ as $n\to\infty$, implying the existence of some $\delta'_n$ so that $\P(\|{\widehat\btheta\|}_\infty > \delta'_n) \to0$. Letting $\Omega_2$ be this set, we conclude
\begin{align}
\left|\E\left[\varphi\left( F_{\rho,0}^\hub(\bXi)\right)-\varphi\left( 
F_{\rho,0}^\hub(\bX)
\right)\right] \right|
&\le
\norm{\varphi}_\infty
\P\left(\Omega_2^c\right)
\to 0,
\end{align}
and similarly for the problems with $\bGamma,\bG$. Now taking $\rho \to\infty$ then $\lambda\to0$ and using Lemma~\ref{lemma:huber} completes the proof.
\paragraph{Independent entries distribution}
 The case where $(\bX,\bG)=(\bX_{\ind},\bG_{\ind})$ is simpler as we can simply take $\bxi_i=\bx_i$, $\bgamma_i = \bg_i$, without conditioning.
 
We similarly apply Theorem~\ref{thm:MS22} to 
$F^\hub_{\rho,l}(\bXi;\lambda), F^\hub_{\rho,l}(\bGamma)$ defined in~\eqref{def:F_aug}
where now instead we take $\widetilde \cC_l =\cC_l$ for $l\in\{1,2\}$,
\begin{equation}
 \widetilde \cC_{0} := \cC_0 \cap \{\norm{\bzeta}_\infty \le\delta'_n \}.
\end{equation}
and $\bzeta_i^\star =\btheta^\star$.

Lemma~\ref{lemma:pointwise-gaussianity_iid} then states that condition~\eqref{eq:pointwise_gaussian_ass} holds for $\cS:=\{\norm{\bzeta}_\infty \le \delta_n\}$ for any sequence $\delta_n \to0$, and Lemma~\ref{lemma:inf_norm_ERM} asserts the appropriate $\ell^\infty$ bound on the minimizers of $F_{\rho,0}^\hub$ necessary to complete the proof as in the case under the random features assumption.

\subsection{Uniform Approximation with the Huber Loss: proof of Lemma~\ref{lemma:huber}
}
\label{section:huber_reduction}
The goal of this section is to establish that for any bounded Lipschitz function $\varphi$, we have

\begin{equation}
\lim_{\lambda\to0}\lim_{\rho \to \infty} \sup_{n \in \N} \left|
\E\left[
\varphi\left(F_{\rho,l}^\hub(\bX; \lambda)\right)
\right] 
-
\E\left[
\varphi\left(  F^\sur_{l}(\bX)\right)
\right] 
 \right| =0,
\end{equation}
(and similarly for the objectives involving $\bG$),
where $F_{\rho,l}^\hub$ was defined in~\eqref{eq:def_hub_obj}.
We will complete the proof for the non-Gaussian features $\bX$ under the the random features assumption of Section~\ref{section:ass_RF}. The proof of the approximation for the Gaussian equivalent feature matrix is similar.
Throughout, we fix all the parameters $l, \kappa, b, s$ and hence suppress their appearance in the notation.
For further notational simplicity, we suppress the appearance of the feature matrix $\bX$ in the argument of $L_\rho, F_\rho$ in some parts of the proof.

Recalling the definition of the huber loss $h_\rho$ in~\eqref{eq:huber_def} and noting that $h_\infty := \lim_{\rho\to\infty} h_\rho$ coincides with the square loss, we see that $F_{\infty,l}^\hub(\bX;\lambda)$ is the original problem $F_l^\sur(\bX)$ but with an objective that is $\ell^2$-regularized. 
Now, we clearly have
\begin{equation}
\lim_{\lambda\to0} \sup_{n \in \N} \left|
\E\left[
\varphi\left(F_{\infty,l}^\hub(\bX; \lambda)\right)
\right] 
-
\E\left[
\varphi\left(  F^\sur_{l}(\bX)\right)
\right] 
 \right| =0
\end{equation}
since the constraint sets $\cC_l$ are contained in an $\ell^2$ ball of radius constant in $n$. So to prove the lemma, it is sufficient to show that for any fixed $\lambda>0$,
\begin{equation}
\lim_{\rho\to\infty} \sup_{n \in \N} \left|
\E\left[
\varphi\left(F_{\rho,l}^\hub(\bX; \lambda)\right)
\right] 
-
\E\left[
\varphi\left(F_{\infty,l}^\hub(\bX; \lambda)\right)
\right] 
 \right| =0.
\end{equation}

The proof of this contains three steps. The first two steps are completely deterministic in nature. The last step involves taking expectation of a function of $\bX$ which exploits certain light tail properties of its entries. 

\paragraph{Step I:} 
First observe that the following inequality holds 
for every $\rho > 0$, and $t \in \R$:
\begin{equation*}
	h_\rho(t) \le h_\infty(t) \le h_\rho(t) + t^2 \mathbf{1}_{|t| \ge \rho}.
\end{equation*}
Introducing the notation $\bu_i(\btheta) = (\kappa - y_i \bx_i^\sT\btheta)_+$, we translate 
the bound on the loss $h_\rho$ to one on the associated objective $L_\rho$ defined in~\eqref{eq:def_hub_obj} as
\begin{equation}
\label{eqn:inequality-one}
	L_\rho(\btheta; \lambda) \le L_\infty(\btheta; \lambda) 
		\le L_\rho(\btheta; \lambda) + \frac{1}{n}\sum_{i=1}^n \bu_i^2(\btheta)  \mathbf{1}_{|\bu_i(\btheta)| \ge \rho}.
\end{equation} 
Note this inequality holds for every $\rho > 0$, and every vector $\btheta$. 
Letting $\btheta_\rho$ denotes the minimizer of 
$\rho \mapsto L_\rho(\btheta; \lambda)$ subject to the constraint $\btheta \in \mathcal{C}_l$, then, 
by taking infimum over $\btheta$ in the inequalities~\eqref{eqn:inequality-one}, we obtain the following bound 
which holds for every $\rho, \btheta$:
\begin{equation}
	F_\rho(\lambda) \le F_\infty(\lambda) \le F_\rho(\lambda) + \frac{1}{n} \sum_{i=1}^n 
		\bu_i^2(\btheta_\rho)  \mathbf{1}_{|\bu_i(\btheta_\rho)| \ge \rho}.
\end{equation} 
To prove~Lemma~\ref{lemma:huber}, it remains to show the limit 
\begin{equation*}
	\lim_{\rho \to \infty} \E \left[\frac{1}{n}\sum_{i=1}^n 
		\bu_i^2(\btheta_\rho)  \mathbf{1}_{|\bu_i(\btheta_\rho)| \ge \rho}\right] 
			= 0.
\end{equation*}
By exchangeability, this is equivalent to showing 
$\lim_{\rho \to \infty} \E[\bu_1^2(\btheta_\rho) \mathbf{1}_{|\bu_i(\btheta_\rho)| \ge \rho}] = 0$.
Meanwhile, by Markov's inequality, it suffices to show that there is uniform control 
\begin{equation}
	\E[\bu_1^4(\btheta_\rho)] \le C
\end{equation} 
where $C < \infty$ is a constant independent of $\rho, n$. 

Bounding the moments of $\bu_1^4(\btheta_\rho)$ tightly in a direct way faces a challenge  
since both the argument $\btheta_\rho$
and $\bu_1$ are both \emph{dependent} on the first sample $(\bx_1, \by_1)$. 
We decouple this dependence via a leave-one-out argument in the next step.

\paragraph{Step II: Deterministic Upper Bound via leave-one-out argument} 
Consider the following perturbed loss function which does not involve the first sample $(\bx_1, \by_1)$: 
\begin{equation*}
	L_{\rho}^{(1)}(\btheta; \lambda) = \frac{1}{n} \sum_{i = 2}^n h_\rho(\bu_i(\btheta)) + \frac{\lambda}{2} \norm{\btheta}_2^2. 
\end{equation*} 
Let us denote $\btheta^{(1)}_\rho$ to be its minimizer subject to the constraint $\btheta \in \mathcal{C}_l$. 
The strong convexity of $\btheta \mapsto L_{\rho}^{(1)}(\btheta; \lambda)$ immediately yields that 
\begin{equation*}
	L_{\rho}^{(1)}(\btheta_\rho; \lambda) \ge L_{\rho}^{(1)}(\btheta_\rho^{(1)}; \lambda) + \frac{\lambda}{2} 
		\norm{\btheta_\rho- \btheta_\rho^{(1)}}_2^2. 
\end{equation*}
Since $\btheta_\rho$ is the minimizer of $L_\rho$, we also have the bound $L_\rho(\btheta_\rho; \lambda) \le L_\rho(\btheta_\rho^{(1)}; \lambda)$.
Combining these bounds, we immediately see that 
\begin{equation*}
	\frac{1}{n} h_\rho (\bu_1(\btheta_\rho^{(1)})) \ge \frac{1}{n} h_\rho (\bu_1(\btheta_\rho)) + \frac{\lambda}{2} 
		\norm{\btheta_\rho- \btheta_\rho^{(1)}}_2^2.
\end{equation*}
The left-hand-side is upper bounded by $\frac{1}{n} h_\infty (\bu_1(\btheta_\rho^{(1)})):= \frac{1}{n}(\bu_1(\btheta_\rho^{(1)}))^2$, 
while the right-hand-side is lower bounded by  $\frac{\lambda}{2} \|\btheta_\rho- \btheta_\rho^{(1)}\|_2^2$.
So 
\begin{equation*}
	\norm{\btheta_\rho- \btheta_\rho^{(1)}}_2 \le \sqrt{\frac{2}{\lambda n}} \bu_1(\btheta_\rho^{(1)}).
\end{equation*} 
As $\btheta \mapsto \bu_\rho^{(1)}(\btheta)$ is Lipschitz, with Lipschitz constant bounded by $\norm{\bx_1}_2$, this allows us to further obtain that 
\begin{equation*}
	\left|\bu_1(\btheta_\rho) - \bu_1(\btheta_\rho^{(1)})\right| \le  \sqrt{\frac{2 }{\lambda n}} 			
		\norm{\bx_1}_2\bu_1(\btheta_\rho^{(1)}).
\end{equation*} 
Rearranging the inequality, we obtain the following bound that holds with probability one
\begin{equation}
\label{eq:upper_bound_urho}
	0 \le \bu_1(\btheta_\rho)\le \left(1 +  \sqrt{\frac{2 }{\lambda n}} 			
		\norm{\bx_1}_2 \right)\bu_1(\btheta_\rho^{(1)}).
\end{equation}

\paragraph{Step III}
Following the bound in~\eqref{eq:upper_bound_urho}
and using Cauchy Schwartz's inequality, we obtain
\begin{equation}
	\E\left[|\bu_\rho^{(1)}(\btheta_\rho)|^4\right] \le \sqrt{\E\left[\left(1 +  \sqrt{2/(\lambda n)} 			
		\norm{\bx_1}_2 \right)^8\right]} \cdot \sqrt{\E\left[\left|\bu_1(\btheta_\rho^{(1)})\right|^8\right]}.
\end{equation} 
We bound each expectation on the right by constants independent of $\rho, n$. Note $\bu_1$ is independent of the 
vector $\btheta_\rho^{(1)}$. Furthermore $\|\btheta_\rho^{(1)}\|_2 \le 2$,
since the constraint set $\mathcal{C}_l$ is contained in $\{\btheta: \norm{\btheta}_2 \le 2\}$. 
By H\"{o}lder's inequality, it suffices to prove the following two properties for the row vector $\bx_1$: 
\begin{enumerate}
\item $\E \left[\norm{\bx_1/\sqrt{n}}_2^8\right]$ is bounded by a constant $C$ independent of $n, \gamma$. 
\item $\sup_{\btheta: \norm{\btheta}_2 = 1} \E[\langle \bx_1, \btheta\rangle^8]$ is bounded by a 
	constant $C$ independent of $n, \gamma$. 
\end{enumerate}

However, both properties are a direct consequence of the subgaussianity of the feature vectors under $\bX =\bX_{\ind}$ or the subgaussianty of conditional on $\bW$ for $\bX =\bX_{\RF}$ (c.f. Lemma~\ref{lem:basic:op:norm:RF}).

\subsection{$\ell^\infty$-norm bounds on the minimizers of $L^\hub$ over $\cC_0$}
Our goal is to prove the following $\ell^\infty$ bounds on the minimizers, needed for the proof of Proposition~\ref{prop:univ_fixed_kappa}.
\begin{lemma}
\label{lemma:inf_norm_ERM}
Let $\widehat\btheta$ be the minimizer of $L^{\hub}_\rho$ over $\cC_0$. Then under the assumptions of Proposition~\ref{prop:univ_fixed_kappa}, we have
\begin{equation}
 \norm{\widehat \btheta}_\infty \stackrel{p}{\to} 0.
\end{equation}
\end{lemma}

In what follows, we will consider the case where the features $(\bX,\bG)$ correspond to each of the distribution classes in Sections~\ref{section:ass_RF} and~\ref{section:ass_ind} separately. Furthermore, in each case, we only provide a proof for $L^\hub_\rho(\btheta;\bX, \lambda)$ since the proof for the minimizer of $L^\hub_\rho(\btheta;\bG, \lambda)$  follows from a similar argument.
In what follows, we suppress the appearance of $\rho$ and $\lambda$ in the notation.

\begin{proof}[Proof of Lemma~\ref{lemma:inf_norm_ERM} for $(\bX,\bG) = (\bX_{\RF},\bG_{\RF})$]

For a vector $\bv$, let $\bv[j;t]$ be the vector obtained by replacing the $j$th element of $\bv$ with $t$.
Fix a column index $j\in[p]$.
Then we can write with this notation
\begin{equation}
L^\hub(\bX, \btheta[j;t]):= \frac1n \sum_{i=1}^n 
\ell\left(\bx_i[j; 0]^\sT \btheta[j;0] + d_{i,j}t; y_i \right)+ \frac{\lambda}{2} \norm{\btheta[j;0]}_2^2 + \frac\lambda2 t^2
\end{equation}
for $\btheta\in\R^p$ and $t\in\R$;
here we defined
$ \ell(s;y) := \ell_\rho\big((\kappa - y s)_+\big).$
For fixed $t \in \R$, let
\begin{equation}
 \widehat \btheta^\up{j;t}:= \argmin_{ \btheta\in \cC_0, \theta_j = t}  L^\hub(\bX,\btheta).
\end{equation}

We remark that $\ell$ is convex since $\ell_\rho$ restricted to $\R_{>0}$ is non-decreasing.
Furthermore, on $\Omega_1$, the constraint set is convex in $\btheta[j;0]$ for a fixed $t$ since $\cC_0$ is convex.
The convexity of $\ell(s;y)$ and strong convexity of the norm squared
at $\widehat \btheta^\up{j;t}$ imply that for all $\btheta\in\cC_0$ we have
\begin{align}
 L^\hub(\bX, \btheta) &\ge
 \frac1n \sum_{i=1}^n
\ell(\widehat\btheta^{\up{j;0}\sT}\bx_{i}[j;0];y_i) 
+ \frac1n\sum_{i=1}^n\ell'(\widehat\btheta^{\up{j;0}\sT}\bx_{i}[j;0];y_i) \left( \btheta -\widehat \btheta^\up{j;0}\right)^\sT \bx_{i}[j;0] +  
\frac\lambda2\norm{\widehat\btheta^\up{j;0}}_2^2  \\
&+ \frac1n\sum_{i=1}^n\ell'(\widehat\btheta^{\up{j;0}\sT} \bx_{i}[j;0];y_i)\theta_{j} d_{i,j}
 + \lambda\left(\btheta- \widehat\btheta^\up{j;0}\right)^\sT\widehat\btheta^\up{j;0}
 + \frac\lambda2 \norm{\btheta[j;0] - \widehat\btheta^\up{j;0}}_2^2
 +\frac\lambda2 \theta_{j}^2
 \\
&= L^\hub(\bX, \widehat \btheta^\up{j;0}) + \widehat \Delta_j \theta_{j} + \frac\lambda2 \theta_{j}^2 
+ \frac{\lambda}2 \norm{\btheta[j;0] - \widehat\btheta^\up{j;0}}_2^2 + \widehat h_j(\btheta).
\label{eq:lower_bound_loo}
\end{align}
where we defined  
\begin{align}
\widehat \Delta_j&:= \frac1n\sum_{i=1}^n \ell'(\widehat\btheta^\up{j;0}\bx_{i}[j;0];y_i)d_{i,j} \\
\widehat h_j(\btheta) &:= \frac1n \sum_{i=1}^n \ell'(\widehat\btheta^{\up{j;0}\sT}\bx_{i}[j;0];y_i)\bx_{i}[j;0]^\sT(\btheta - \widehat\btheta^\up{j;0})  
+ \lambda\widehat\btheta^{\up{j;0}\sT}(\btheta -\widehat\btheta^{\up{j;0}}).
\end{align}

Now by the KKT conditions for $\widehat\btheta^\up{j;0}$ along with complementary slackness, $h(\btheta) \ge 0$ 
for all $\btheta \in\cC_0$ such that $\btheta_j = 0$.
However, if $\btheta\in\cC_0$ then $\btheta[j;0] \in\cC_0$ as well. To see this, note that $\cC_0$ is orthosymmetric: if $\btheta \in\cC_0$, then $\btheta[j; -\theta_{j}] \in\cC_0$ directly from the definition. 
So convexity of $\cC_0$ then implies that
$\btheta[j;0]\in\cC_0$.
Hence, for all $\btheta\in\cC_0$, we have
\begin{equation}
 \widehat h_j(\btheta) = \widehat h_j(\btheta[j;0]) \ge 0
\end{equation}
where the first equality follows since $\widehat h_j$ is independent of the $j$th coordinate.
Combing this with the lower bound in~\eqref{eq:lower_bound_loo}, we obtain
\begin{equation}
 L^\hub(\bX, \widehat\btheta^\up{j;0}) \ge  \min_{\btheta\in\cC_0} L^\hub(\bX, \btheta) \ge 
 L^\hub(\bX, \widehat\btheta^\up{j;0 }) + \widehat \Delta_j  \widehat \theta_{j} + \frac\lambda2 \widehat\theta_{j}^2,
\end{equation}
where $\widehat\btheta$ is the minimizer over $\cC_0$ of $L^\hub(\bX,\btheta).$
Rearranging gives
\begin{equation}
\label{eq:inf_bound_whp}
 | \widehat \theta_j | \le \frac{2\widehat \Delta_j}{\lambda}.
\end{equation}
Let us show that $\widehat \Delta_j$ is small.
Let  $\Omega_3 := \left\{ \norm{\bZ}_\op\le \frac{2}{\sqrt{n}} \right\}$
and note that we have
$$\widehat \Delta_j= \frac1n\sum_{i=1}^n \ell'(\widehat\btheta^\up{j;0}\bx_{i}[j;0];y_i)d_{i,j} = \frac1n \ba_j^\sT \widetilde\bx_j$$
for $\ba_j \in \R^n$, and $\widetilde\bx_j$ the $j$th column of $\bX$. 
For $\bW$ random and conditional on $\bZ\in\Omega_3$, $\ba_j$ is independent of $\widetilde \bx_j$ for all $j$ since $y_i$ is only a function of the $j$th 
weight $\bw_j$. Then since $\norm{\ba_j}_\infty \le 2 \rho$, we have that each $\widehat \Delta_j$ is $2\rho/\sqrt{n}$ subgaussian. 
So

\begin{equation}
 \P\left(\left\{\max_{j\le p } \widehat \Delta_j \ge 2\rho \frac{\log{p}}{\sqrt{p}}\right\} \right) \le C_1 \exp\left\{-c (\log p)^2\right\} + \P(\Omega_3^c)  \le C_2 \exp\{-c (\log p)^2\}
\end{equation}
for some $C_1,C_2 >0$.
Hence, by Eq.~\eqref{eq:inf_bound_whp}, and using that $\Omega_1$ has exponentially large probability,
\begin{equation}
\P\left(\big\|{\widehat \btheta}\big\|_\infty \ge \frac{4 \rho}{\lambda} \frac{\log p}{\sqrt{p}}  \right)
\le C_2 \exp\{-c (\log p)^2\}.
\end{equation}
Using that $\norm{\bX}_\op \le C_2/\sqrt{n}$ with probability at least $1- C_3 e^{-c_2 n}$
we conclude that 
\begin{equation}
 \big\|\widehat \btheta\big\|_\infty \stackrel{p}{\to} 0 
\end{equation}

\end{proof}

\begin{proof}[Proof of Lemma~\ref{lemma:inf_norm_ERM} for $(\bX,\bG) = (\bX_{\ind},\bG_{\ind})$]
We will establish the lemma in the case where $\bSigma =\bI$. Generalizing other diagonal $\bSigma$ with bounded spectrum as in Assumption~${\sf (B1)}$ is straightforward. 

For $j\in[p]$, let $y_i^\up{j}$ be the response variable obtained by setting the $j$th element of $\bx_i$ to 0, i.e.,
\begin{equation}
 \P\left(y_i^\up{j} = +1 | \bx_i \right) = f(\bx_i[j;0]^\sT \btheta^\star ).
\end{equation}
Now let
\begin{equation}
\widehat \btheta^\up{j;t} := \argmin_{\btheta \in \cC_0, \theta_j = 0} \frac1n \sum_{i=1}^n \ell(\bx_i^\sT \btheta; y_i^\up{j}) +\frac\lambda2 \norm{\btheta}^2.
\end{equation}
From the definition, it is clear that $\widehat\btheta^\up{j;t}$ is independent of $x_{ij}$ for all $i$.
Now using an argument analogous to the one in the previous proof, we can show that
for any $\btheta$ such that $\btheta \in\cC_0$,
\begin{equation}
 L^{\hub,\up{j}}(\bX, \btheta) \ge L^{\hub,\up{j}}(\bX;\widehat\btheta^\up{j;0}) + \widehat \Delta_j \theta_j + \frac{\lambda}{2} \theta_j^2 + \widehat h_j(\btheta)
\end{equation}
where $L^{\hub,\up{j}}$ is the objective with $y_i^\up{j}$ instead of $y_i$ and 
\begin{align}
\widehat \Delta_j&:= \frac1n\sum_{i=1}^n \ell'(\widehat\btheta^\up{j;0}\bx_{i}[j;0];y_i^\up{j})d_{i,j} \\
\widehat h_j(\btheta) &:= \frac1n \sum_{i=1}^n \ell'(\widehat\btheta^{\up{j;0}\sT}\bx_{i}[j;0];y_i^\up{j})\bx_{i}[j;0]^\sT(\btheta - \widehat\btheta^\up{j;0})
+ \lambda\widehat\btheta^{\up{j;0}\sT}(\btheta -\widehat\btheta^{\up{j;0}}).
\end{align}

Now we note that since $|\ell(s;t) - \ell(s;t')|\le C |s| |t-t'|$, we have for any $\btheta $ with bounded norm,
\begin{equation}
 |L^{\hub,\up{j}}(\bX;\btheta) -L^\hub(\bX;\btheta)|  \le  C \frac{\norm{\bX}_\op}{\sqrt{n}} \frac{\norm{\by - \by^\up{j}}_2}{\sqrt{n}} = :C\frac{\norm{\bX}_\op}{\sqrt{n}} \widehat\epsilon_j^{1/2}.
\end{equation}
So one obtains for $\widehat \btheta$, the minimizer of $L^\hub(\bX;\btheta)$, 
\begin{equation}
 2 C\frac{\norm{\bX}_\op}{\sqrt{n}} \widehat \epsilon_j^{1/2} \ge \widehat \Delta_j \widehat\theta_j +\frac{\lambda}{2} \widehat\theta_j^2  
\end{equation}
where we used that $\widehat h_j(\widehat \btheta) \ge 0$ by the KKT conditions at $\widehat \btheta^\up{j;0}$. This implies
\begin{equation}
 |\theta_j| \le C(\lambda) \left(| \widehat \Delta_j| +   \frac{\norm{\bX}_\op^{1/2}}{n^{1/4}}|\widehat \epsilon_j|^{1/4} \right)\le 
 C(\lambda) \left(|\widehat \Delta_j| +  \frac{\norm{\bX}_\op^{1/2}}{n^{1/4}}\left(|\widehat \epsilon_j - \mu_j|^{1/4} + |\mu_j|^{1/4}\right)\right)
\end{equation}
where $\mu_j = \E[\widehat \epsilon_j]$.

First note that
\begin{equation}
 \mu_j = \frac1n\sum_{i=1}^n \E[(y_i - y_i^\up{j})^2]\le   C \norm{f}_\Lip \theta_j^{\star}.
\end{equation}
So $\max_j |\mu_j| \le C \norm{f}_\Lip \norm{\btheta^\star}_\infty$. 
Second, 
\begin{equation}
 |\widehat \epsilon_j - \mu_j | = \left|\frac1n \sum_{i=1}^n \chi_i^\up{j} - \E[\chi_i^\up{j}]\right|
\end{equation}

for some $\chi_i^\up{j}$ i.i.d. random variables, for different $i$, bounded by $4$.  So Hoeffding gives
\begin{equation}
 \P\left(|\widehat \epsilon_j - \mu_j| \ge t \right) \le \exp\left\{ -C n t^2\right\}
\end{equation}
for all $t>0$, implying that $|\widehat \epsilon_j - \mu_j|$ is $O(1/\sqrt{n})$ subgaussian for each $j$. Hence
$\max_{j}|\widehat \epsilon_j  - \mu_j| = o_\P(1)$. Finally, $\widehat \Delta_j$ is also $O(1/\sqrt{n})$ subgaussian by the same logic applied to the random features
case. Hence, $\norm{\widehat \btheta}_\infty = o_\P(1)$ as long as $\norm{\btheta^\star}_\infty = o(1)$.

\end{proof}
\def\sRF{\sf RF}
\def\bbW{\boldsymbol{\widetilde{W}}}
\def\bbw{\boldsymbol{\widetilde{w}}}
\def\bbSigma{{\boldsymbol{\widetilde{\Sigma}}}}
\def\bbLambda{{\boldsymbol{\widetilde{\Lambda}}}}
\def\Unif{\mathsf{Unif}}

\section{Delocalization of max-margin estimators}
\label{sec:delocalization:MM}
In this section, we prove Proposition \ref{proposition:delocalization-for-general-max-margin}. We prove Proposition \ref{proposition:delocalization-for-general-max-margin} first for the Gaussian features $\bD=\bG$ and then leverage the universality of the separable loss proven in Proposition \ref{prop:univ_fixed_kappa} to show the same for the non-Gaussian features $\bD=\bX$. For the Gaussian features $\bD=\bG$, the proof is based Gordon's inequality \cite{Gordon88, thrampoulidis2015, montanari2019generalization}. To this end, assume that $\bg\sim \normal(0,\bSigma)$, where the covariance matrix $\bSigma$ satisfies the following. 
    \begin{equation}\label{eq:assumption:Sigma}
    \lambda_{\min}(\bSigma), \lambda_{\max}(\bSigma)\in [c,C]\,, \quad \sup_{\la \geq 0}\big\|\big(\bSigma+\la \bI\big)^{-1}\bSigma\btheta^\star\big\|_{\infty}\to 0\,,\quad \sup_{\la \geq 0}\frac{1}{\sqrt{n}}\big\|\big(\bSigma+\la \bI\big)^{-1}\bSigma^{1/2} \bh \big\|_{\infty}\stackrel{p}{\to} 0\,,
    \end{equation}
    where $\bh \sim \normal(0,\bI_p)$. Indeed, for the independent features model under Assumption \ref{assumption:ind}, the covariance matrix $\bSigma=\bSigma_{\ind}$ trivially satisfies the condition \eqref{eq:assumption:Sigma} since $\bSigma_{\ind}$ is a diagonal matrix.

    For the random features model, recall that $\bSigma_{\sRF}\equiv \mu_1^2\bW\bW^{\sT}+\mu_2^2\bI_p$, where the rows $(\bw_i)_{i\leq p}$ of $\bW$ are $\bw_i\stackrel{i.i.d.}{\sim}\mathsf{Unif}(\mathbb{S}^{d-1}(1))$. Since the first condition in \eqref{eq:assumption:Sigma} holds w.h.p. by Theorem \ref{theorem:matrix-concentration-covariance-operator}, we focus on establishing the other $2$ conditions for $\bSigma_{\sRF}$. Recalling that the corresponding $\btheta^\star$ is given in \eqref{eq:def:theta:star}, we establish the following lemma in Section \ref{ss:cov:RF}.
\begin{lemma}\label{lem:cov:RF:delocalized}
The matrix $\bSigma_{\sRF}=\mu_1^2\bW\bW^{\sT}+\mu_2^2\bI_p$ satisfies the condition \eqref{eq:assumption:Sigma} w.h.p.. More precisely, for $\bh \sim \normal(0,\bI_p)$ independent of $\bW$, we have almost surely that
\[
\begin{split}
&\sup_{\lambda \geq \la_0} \Big(\big\|(\bW\bW^{\sT}+\la_0 \bI)^{-1}\bW \beta^\star\big\|_{2}\Big)^{-1}\cdot \big\|(\bW\bW^{\sT}+\la \bI)^{-1}\bW \beta^\star\big\|_{\infty}{}\longrightarrow 0\,,\\
&\sup_{\lambda \geq \la_0} \frac{1}{\sqrt{n}}\Big\|\big(\bW\bW^{\sT}+\la \bI\big)^{-1}(\bW\bW^{\sT}+\la_0\bI)^{1/2}\bh\Big\|_{\infty}\longrightarrow 0\,,
\end{split}
\]
where $\lambda_0 := \mu_2^2/\mu_1^2 >0$ is a constant that does not depend on $n,p$.
\end{lemma}
Next, we show that under the condition \eqref{eq:assumption:Sigma}, the max-margin estimator under the Gaussian design is delocalized. Note that in the separable regime, the max-margin estimator satisfies $\big\|\bthetaMM^{\bD}\big\|_2=1$ w.h.p.. Thus,  $\big\|\bthetaMM^{\bD}\big\|_{\infty}\leq \big\|\bthetaMM^{\bD}\big\|_{4}\leq \big\|\bthetaMM^{\bD}\big\|_{\infty}^{1/2}$ holds, so $\big\|\bthetaMM^{\bD}\big\|_{\infty}\pto 0$ is equivalent to $\big\|\bthetaMM^{\bD}\big\|_4\pto 0$.  To this end, let $\widehat H_3(\bD)$ be defined by
\begin{equation}\label{eq:def:H0}
    \widehat H_3(\bD) := \min_{\substack{\norm{\btheta}_2\leq 2\\ y_i\langle \bd_i, \btheta\rangle \ge \kappa^\star, \forall i \leq n}} \norm{\btheta}_2^2 -s_0 \norm{\btheta}_4^{4},\qquad s_0\equiv \frac{1}{12}
\end{equation}
which is an analogue of $\widehat H_{\ell}(\bD,s,\kappa), \ell\in \{1,2\},$ in \eqref{eq:def:H}. The choice of $s_0=\frac{1}{12}$ and the additional $2-$norm constraint on $\norm{\btheta}\leq 2$ above is to guarantee that i) the above optimization is convex and ii) $\hbtheta^{\bD}_0\equiv \frac{\kappa^\star}{\widehat\kappa_n(\bD)}\bthetaMM^{\bD}$ belongs to the constraint set with high probability. Moreover, define $\widehat F_{3}(\bD,\eps)$ for $\eps>0$ by
\begin{equation}\label{eq:def:F:eps}
\widehat F_{3}(\bD,\eps):=\min_{\substack{\norm{\btheta}_2^2 - s_0 \norm{\btheta}_4^{4}\leq 1-\eps \\ \norm{\btheta}_2\leq 2}} \frac{1}{n}\sum_{i=1}^{n} \left(\kappa^\star-y_i\langle \bd_i, \btheta\rangle \right)_{+}^2.
\end{equation}
$\widehat F_{3}(\bD,\eps)$ and $\widehat H_3(\bD)$ are related by the equivalence
\begin{equation}\label{eq:another:equiv}
    \widehat F_{3}(\bD,\eps)=0 \iff \widehat H_3(\bD)\leq 1-\eps
\end{equation}
Then, the key to proving $\big\|\bthetaMM^{\bG}\big\|\pto 0$ is the following lemma, whose proof is based on Gordon's inequality  as seen in Section \ref{subsec:lem:4norm:separable:loss:lower}.
\begin{lemma}\label{lem:4norm:separable:loss:lower}
Suppose that the rows of $\bG\in \R^{n\times p}$ are given by $(\bg_i)_{i\leq n}\stackrel{i.i.d.}{\sim}\normal(0,\bSigma)$, and the labels $(y_i)_{i\leq n}\in \{\pm 1\}$ are generated according to $\P(y_i=+1\given \bg_i)=f(\bg_i^{\sT}\btheta^\star)$. where $(\bSigma, \btheta^\star)$ satisfies the condition \eqref{eq:assumption:Sigma} and $f$ is continous. Then, for any $\eps>0$, there exists $\delta(\eps)>0$ such that $\widehat F_3(\bG,\eps)\geq \delta(\eps)$ w.h.p..
\end{lemma}
We now show that Proposition \ref{prop:univ_fixed_kappa} and Lemma \ref{lem:4norm:separable:loss:lower} implies Proposition \ref{proposition:delocalization-for-general-max-margin}.
\paragraph{Proof of Proposition \ref{proposition:delocalization-for-general-max-margin}}
We first show that $\widehat H_3(\bD) \stackrel{p}{\to} 1$ holds for both $\bD=\bX$ and $\bD=\bG$. First, by definition of $\widehat H_3(\bD)$ in \eqref{eq:def:H0}, we have
\begin{equation*}
    \widehat H_3(\bD)\leq \min_{y_i\langle \bd_i, \btheta\rangle \ge \kappa^\star, \forall i \leq n} \norm{\btheta}_2^2 \pto 1,
\end{equation*}
where the convergence in probability is due to Lemma \ref{lemma:MM_l2_limit}. To obtain a matching lower bound, consider arbitrary $\eps>0$. Note that Lemma \ref{lem:cov:RF:delocalized} and Lemma \ref{lem:4norm:separable:loss:lower} show that $\widehat F_3(\bG,\eps)\geq \delta(\eps)$ holds w.h.p., for both $\bG\in \{\bG_{\RF},\bG_{\ind}\}$. By Proposition \ref{prop:univ_fixed_kappa}, the distribution of $\widehat F_3(\bD,\eps)$ is universal for $\bD=\bX$ and $\bD=\bG$, thus it follows that $\widehat F_3(\bD,\eps)\geq \frac{\delta(\eps)}{2}$ holds w.h.p. for any $\bD$. As a consequence,
\begin{equation*}
\begin{split}
    \lim_{n\to\infty} \P\left(\widehat H_3(\bD)\leq 1-\eps\right)
    &\stackrel{(a)}{=}\lim_{n\to\infty}\P\left( \widehat F_{3}(\bD,\eps)=0 \right)\\
    &\leq \lim_{n\to\infty}\P\left( \widehat F_{3}(\bD,\eps)<
    \frac{\delta(\eps)}{2}\right)\\
    &=0,
\end{split}
\end{equation*}
where $(a)$ is by the equivalence \eqref{eq:another:equiv}. Therefore, $\widehat H_3(\bD) \stackrel{p}{\to} 1$ holds for both $\bD=\bX$ and $\bD=\bG$.

Now, let $\hbtheta_0^{\bD}\equiv \frac{\kappa^\star}{\widehat \kappa_n(\bD)}\bthetaMM^{\bD}$. Since $\widehat \kappa_n(\bD)\pto \kappa^\star$ holds by Theorem \ref{theorem:universality-of-the-margin}, it suffices to show that $\big\|\hbtheta_0^{\bD}\big\|_4\pto 0$. Note that $\big\|\hbtheta_0^{\bD}\big\|_2\leq \frac{\kappa^\star}{\widehat \kappa_n(\bD)}\pto 1$, so $\big\|\hbtheta_0^{\bD}\big\|_2\leq 2$ holds w.h.p.. Further, $\min_{i\leq n} y_i \langle \bg_i, \hbtheta_0^{\bD}\rangle =\kappa^\star$ holds by construction of $\hbtheta_0^{\bD}$. Thus, $\hbtheta_0^{\bD}$ is in the constraint set of the optimization problem \eqref{eq:def:H0} w.h.p.. Therefore, w.h.p. we have that
\begin{equation*}
    \norm{\hbtheta_{0}^{\bD}}_4=\frac{\norm{\hbtheta_0^{\bD}}_2^2-s_0\norm{\hbtheta_0^{\bD}}_4^4-\norm{\hbtheta_0^{\bD}}_2^2}{-s_0}\leq\frac{\widehat H_3(\bD)-\norm{\hbtheta_0^{\bD}}_2^2}{-s_0}.
\end{equation*}
Since we have shown that $\widehat H_3(\bD) \stackrel{p}{\to} 1$ holds for both $\bD=\bX$ and $\bD=\bG$ and $\norm{\hbtheta_0^{\bD}}_2^2\pto 1$ holds by Lemma \ref{lemma:MM_l2_limit}, the right hand side of the inequality above converges in probability to $0$. This concludes the proof.

\subsection{Proof of Lemma \ref{lem:cov:RF:delocalized}}
\label{ss:cov:RF}
In this subsection, we prove Lemma \ref{lem:cov:RF:delocalized}. Note that the $\lambda_{\min}(\bSigma_{\sRF}), \lambda_{\max}(\bSigma_{\sRF})\in [c,C]$ w.h.p. by Theorem \ref{theorem:matrix-concentration-covariance-operator}, so the first condition of \eqref{eq:assumption:Sigma} is satisfied. Moreover, observe the last condition of \eqref{eq:assumption:Sigma} is implied by the first condition: since $(\bSigma+\la\bI)^{-1}\bSigma^{1/2}\bh\sim \normal(0,(\bSigma+\la\bI)^{-1}\bSigma(\bSigma+\la\bI)^{-1})$, we have with probability at least $1-e^{-cp}$ that
\begin{equation}\label{eq:third:assumption:Sigma}
\begin{split}
\frac{1}{\sqrt{n}}\big\|(\bSigma+\la\bI)^{-1}\bSigma^{1/2}\bh\big\|_{\infty}
&\leq \frac{C\sqrt{\log p}}{\sqrt{n}}\max_{i\leq p}\Big((\bSigma+\la\bI)^{-1}\bSigma(\bSigma+\la\bI)^{-1}\Big)^{1/2}_{ii}\\
&\leq \frac{C\sqrt{\log p}}{\sqrt{n}}\Big\|(\bSigma+\la\bI)^{-1}\bSigma(\bSigma+\la\bI)^{-1}\Big\|_{\op}^{1/2}\leq \frac{C\sqrt{\log p}}{\sqrt{n}}\cdot\frac{\lambda_{\max}(\bSigma)^{1/2}}{\lambda_{\min}(\bSigma)}\,.
\end{split}
\end{equation}
Therefore, we focus on establishing the second condition of \eqref{eq:assumption:Sigma} for $\bSigma_{\sRF}$.

To establish the second condition of \eqref{eq:assumption:Sigma}, we consider the random matrix $\bbW\in \R^{n\times p}$ with rows $(\bbw_i)_{i\leq n} \stackrel{i.i.d.}{\sim}\normal(0,\bI_d/d)$. Recalling that the rows $(\bw)_{i\leq p}$ of $\bW$ are distributed $\bw_i\stackrel{i.i.d.}{\sim}\Unif(\SS^{d-1}(1))$, the matrices $\bbW$ and $\bW$ are related by
\begin{equation}\label{eq:dist:W:bbW}
    (\bw_i)_{i\leq p}\stackrel{d}{=}\Big(\frac{\bbw_i}{\|\bbw_i\|_2}\Big)_{i\leq p}\,.
\end{equation}
We will first the analog of Lemma \ref{lem:cov:RF:delocalized} for $\bbW$ and then use the distributional relation \eqref{eq:dist:W:bbW} to argue the same for $\bW$. The advantage of considering $\bbW$ first is that $\bbW$ satisfies a rotational invariance property (see Lemma \ref{lem:W:haar} below). In particular, we prove the following lemma in Section \ref{sss:RF:tilde}.
\begin{lemma}\label{lem:cov:RF:tilde:delocalized}
Denote the matrix $\bbSigma_{\sRF}=\mu_1^2\bbW\bbW^{\sT}+\mu_2^2\bI_p$. Then, for $\bh \sim \normal(0,\bI_p)$ independent of $\bbW$, we have almost surely that
\[
\begin{split}
\sup_{\lambda \geq \la_0} \Big(\big\|(\bbW\bbW^{\sT}+\la_0 \bI)^{-1}\bbW \beta^\star\big\|_{2}\Big)^{-1}\cdot \big\|(\bbW\bbW^{\sT}+\la \bI)^{-1}\bbW \beta^\star\big\|_{\infty}{}\longrightarrow 0\,,
\end{split}
\]
where $\lambda_0 := \mu_2^2/\mu_1^2 >0$ is a constant that does not depend on $n,p$.
\end{lemma}
We first prove that Lemma~\ref{lem:cov:RF:tilde:delocalized} implies Lemma~\ref{lem:cov:RF:delocalized}.
\paragraph{Proof of Lemma \ref{lem:cov:RF:delocalized}}
We have $\lambda_{\min}(\bSigma_{\sRF}), \lambda_{\max}(\bSigma_{\sRF})\in [c,C]$ w.h.p. by Theorem \ref{theorem:matrix-concentration-covariance-operator}, so the first condition of \eqref{eq:assumption:Sigma} is satisfied. Further, as seen in \eqref{eq:third:assumption:Sigma}, the last condition of \eqref{eq:assumption:Sigma} is implied by the first condition, thus we focus on proving the second condition in \eqref{eq:assumption:Sigma}.

First, note that on the w.h.p. event $\big\{\|\bW\|_{\op}^2\leq C\}$, we have
\begin{equation}\label{eq:cov:RF:tech}
    \big\|(\bW \bW^{\sT}+\la_0 \bI)^{-1}\bW \bbeta^\star\big\|_2\geq (\la_0+C)^{-1}\big\|\bW\bbeta^\star\big\|_2=(\la_0+C)^{-1}\Big(\sum_{i=1}^{p}\big(\langle \bw_i, \bbeta^\star\rangle\big)^2\Big)^{1/2}\,.
\end{equation}
Since $\sum_{i=1}^{p}\big(\langle \bw_i, \bbeta^\star\rangle\big)^2$ concentrates around its expectation and $\bw_i$ is isotropic, we have that $\sum_{i=1}^{p}\big(\langle \bw_i, \bbeta^\star\rangle\big)^2\geq \Omega(1)$ w.h.p.. Hence, it follows that $\big\|(\bW\bW^{\sT}+\la_0\bI)^{-1}\bW \bbeta^\star\big\|_2\geq \Omega(1)$ w.h.p.. 

Thus, for the rest of the proof, we argue that $\sup_{\la\geq \la_0}\big\|(\bW\bW^{\sT}+\la\bI)^{-1}\bW\bbeta^\star\big\|_{\infty}\to 0$ holds almost surely. Note that by the distributional relation \eqref{eq:dist:W:bbW}, we can couple $\bW$ and $\bbW$ by
\begin{equation*}
    \bW =\bbLambda^{-1}\bbW\,,\quad\textnormal{where}\quad \bbLambda :=\diag\Big(\big\|\bbw_i\big\|_2\Big)_{i\leq p}\,.
\end{equation*}
Then, for each fixed $\lambda\geq \lambda_0$, we can bound $\big\|(\bW\bW^{\sT}+\la\bI)^{-1}\bW\bbeta^\star\big\|_{\infty}$ by
\begin{equation}\label{eq:bound:infinity:1}
\begin{split}
\big\|(\bW\bW^{\sT}+\la\bI)^{-1}\bW\bbeta^\star\big\|_{\infty}
&=\big\|\bbLambda(\bbW\bbW^{\sT}+\la\bbLambda)^{-1}\bbW\bbeta^\star\big\|_{\infty}\\
&\leq \max_{i\leq p}\big\{\|\bbw_i\|_2\big\}\cdot \big\|(\bbW\bbW^{\sT}+\la\bbLambda)^{-1}\bbW\bbeta^\star\big\|_{\infty}\,.
\end{split}
\end{equation}
Since $(\bbw_i)_{i\leq p}\stackrel{i.i.d.}{\sim}\normal(0,\bI_d/d)$ holds, $\sqrt{d}\big(\|\bbw_i\big\|_2-1)$ is $O(1)$-subgaussian (see e.g. \cite[Theorem 3.1.1]{vershynin2018high}). Thus, with probability at least $1-e^{-cp}$, we have
\begin{equation}\label{eq:concentration:2:norm}
\max_{i\leq p}\Big\{\big|\|\bbw_i\|_2-1\big|\Big\}\leq 1+C\frac{\sqrt{\log p}}{\sqrt{d}}\,,
\end{equation}
where $C>0$ is a universal constant. In particular, by Borel-Cantelli lemma, $\max_{i\leq p}\big\{\|\bbw_i\|_2\big\}\to 1$ almost surely. Moreover, using the identity $(\bA+\bB)^{-1}=\bA^{-1}-(\bA+\bB)^{-1}\bB\bA^{-1}$ for $\bA=\bbW\bbW^{\sT}+\la \bI$ and $\bB=\la(\bbLambda-\bI)$, we can bound
\begin{equation}\label{eq:bound:infinity:2}
\begin{split}
&\big\|(\bbW\bbW^{\sT}+\la\bbLambda)^{-1}\bbW\bbeta^\star\big\|_{\infty}\\&\leq \big\|(\bbW\bbW^{\sT}+\la\bI)^{-1}\bbW\bbeta^\star\big\|_{\infty}+\lambda \big\|(\bbW\bbW^{\sT}+\la\bbLambda)^{-1}(\bbLambda-\bI)(\bbW\bbW^{\sT}+\la\bI)^{-1}\bbW\bbeta^\star\big\|_{\infty}\,.
\end{split}
\end{equation}
Using the crude bound $\norm{\ba}_{\infty}\leq \norm{\ba}_2$, the last piece can be bounded by
\begin{equation*}
\begin{split}
&\lambda \big\|(\bbW\bbW^{\sT}+\la\bbLambda)^{-1}(\bbLambda-\bI)(\bbW\bbW^{\sT}+\la\bI)^{-1}\bbW\bbeta^\star\big\|_{2}\\
&\leq \lambda \big\|(\bbW\bbW^{\sT}+\la\bbLambda)^{-1}\big\|_{\op}\big\|(\bbW\bbW^{\sT}+\la\bI)^{-1}\big\|_{\op}\big\|\bbW\big\|_{\op}\cdot \max_{i\leq p}\Big\{\big|\|\bbw_i\|_2-1\big|\Big\}\,.
\end{split}
\end{equation*}
Thus, on the w.h.p. event where $\big\{\big\|\bbW\big\|_{\op}\leq C\big\}$ and \eqref{eq:concentration:2:norm} hold, we have
\begin{equation*}
\lambda \big\|(\bbW\bbW^{\sT}+\la\bbLambda)^{-1}(\bbLambda-\bI)(\bbW\bbW^{\sT}+\la\bI)^{-1}\bbW\bbeta^\star\big\|_{2}\leq \lambda^{-1}C^\prime\cdot\frac{\sqrt{\log p}}{\sqrt{d}}\,.
\end{equation*}
Combining with \eqref{eq:bound:infinity:1} and \eqref{eq:bound:infinity:2}, we have w.h.p. that 
\begin{equation*}
\sup_{\la\geq \la_0}\big\|(\bW\bW^{\sT}+\la\bI)^{-1}\bW\bbeta^\star\big\|_{\infty}\leq C\bigg(\sup_{\la\geq \la_0}\big\|(\bbW\bbW^{\sT}+\la\bI)^{-1}\bbW\bbeta^\star\big\|_{\infty}+\lambda_0^{-1} \frac{\sqrt{\log p}}{\sqrt{d}}\bigg)\,.
\end{equation*}
Combining with Lemma \ref{lem:cov:RF:tilde:delocalized} concludes the proof of Lemma \ref{lem:cov:RF:delocalized}.
\subsubsection{Proof of Lemma~\ref{lem:cov:RF:tilde:delocalized}}
\label{sss:RF:tilde}
We prove Lemma~\ref{lem:cov:RF:tilde:delocalized} by a covering argument. First, we show in Lemma \ref{lem:ptwise:bound} that for fixed $\lambda\geq \lambda_0$, $\big\|(\bbW\bbW^{\sT}+\la \bI)^{-1}\bbW\beta^\star\big\|_{\infty}=O(p^{-1/4})$ with $1-e^{-c\sqrt{p}}$. Then, we take care of large enough $\lambda$ with $\lambda\geq p^{1/4}$ separately in Lemma \ref{lem:large:lambda}, and then cover the interval $[\lambda_0,p^{1/4}]$ by a $p^{-1/4}$ grid. The following symmetry property of $\bbW$ is crucial to argue $\big\|(\bbW\bbW^{\sT}+\la \bI)^{-1}\bbW\beta^\star\big\|_{\infty}=O(p^{-1/4})$ holds pointwise.
\begin{lemma}\label{lem:W:haar}
Let $\bbW\in \R^{p\times d}$ be the matrix with entries $\widetilde{w}_{i,j}\stackrel{i.i.d}{\sim}\normal(0,1/d)$. Let $\bO\in \R^{p\times p}$ distributed according to Haar measure on the orthgonal group $O(p)$ independent of $\bbW$. Then, we have for any $\beta\in \R^p,$ and constants $\alpha_1,\alpha_2>0$,
\begin{equation*}
\boldsymbol v \stackrel{d}{=}\bO \boldsymbol v\quad\textnormal{where}\quad\boldsymbol v:=
    \left(\alpha_1\bbW\bbW^{\sT}+\alpha_2 \bI_p\right)^{-1}\bbW \beta.
\end{equation*}
\end{lemma}
\begin{proof}
The assertion immediately follows from the fact $\bO \bbW\stackrel{d}{=}\bbW$.
\end{proof}
As a consequence of Lemma \ref{lem:W:haar}, we first show the following pointwise estimate.
\begin{lemma}\label{lem:ptwise:bound}
For each fixed $\lambda \geq \lambda_0>0$, we have with probability at least $1-e^{-c\sqrt{p}}$ that
\[
\big\|(\bbW\bbW^{\sT}+\la \bI)^{-1}\bbW\beta^\star\big\|_{\infty}\leq  Cp^{-1/4}\,.
\]
\end{lemma}
\begin{proof}
By Lemma \ref{lem:W:haar}, we have that 
\begin{equation}\label{eq:goal:1:lem:xi:1}
\begin{split}
    \overline{\bv} := \frac{\bv}{\norm{\bv}_2}\sim \Unif(\SS^{p-1}(1))\,,\quad\textnormal{where}\quad \bv:=(\bbW\bbW^{\sT}+\la \bI)^{-1}\bbW\beta^\star\,.
\end{split}
\end{equation}
Thus, $\overline{\bv}\stackrel{d}{=}\frac{\bg}{\norm{\bg}_2}$, where $\bg\sim \normal(0,\bI_p)$. Note that $\norm{\bg}_2\geq \sqrt{p}/2$ with probability at least $1-e^{-cp}$ (see e.g. Theorem 3.1.1 in \cite{vershynin2018high}). Further, $\norm{\bg}_{\infty}\leq p^{1/4}$ with probability at least $1-e^{-c\sqrt{p}}$ by a union bound. Hence,
\begin{equation}\label{eq:v:vprime:infty:norm}
\P\Big(\norm{\overline{\bv}}_{\infty}\geq p^{-1/4}\Big)\leq e^{-c\sqrt{p}}\,.
\end{equation}
Furthermore, we have that
\begin{equation*}
\norm{\bv}_2\leq \la^{-1}\norm{\bbW\beta^\star}\leq \la_0^{-1}\cdot\norm{\bbW}_{\op}\,.
\end{equation*}
By Theorem \ref{theorem:matrix-concentration-covariance-operator}, $\norm{\bbW}_{\op}\leq C$ holds with probability at least $1-e^{-cp}$. Thus, $\norm{\bv}_2\leq C^\prime$ holds with probability at least $1-e^{-cp}$. Therefore, combining with \eqref{eq:v:vprime:infty:norm}, we have that with probability at least $1-e^{-c\sqrt{p}}$ that
\[
\norm{\bv}_{\infty}=\norm{\bv}_2\cdot \norm{\overline{\bv}}_{\infty}\leq Cp^{-1/4}\,,
\]
which concludes the proof.
\end{proof}
By a crude estimate, we next deal with the case where $\la\geq p^{1/4}$.
\begin{lemma}\label{lem:large:lambda}
    With probability at least $1-e^{-cp}$, we have
    \[
    \sup_{\la\geq p^{1/4}}\big\|(\bbW\bbW^{\sT}+\la \bI)^{-1}\bbW\beta^\star\big\|_{\infty}\leq  Cp^{-1/4}\,.
    \]
\end{lemma}
\begin{proof}
    Using a crude bound $\norm{\ba}_{\infty}\leq \norm{\ba}_2$ for $\ba\in \R^p$, we have
    \[
    \big\|(\bbW\bbW^{\sT}+\la \bI)^{-1}\bbW\beta^\star\big\|_{\infty}\leq \big\|(\bbW\bbW^{\sT}+\la \bI)^{-1}\bbW\beta^\star\big\|_{2}\leq \la^{-1}\norm{\bbW}_{\op}.
    \]
    Since $\norm{\bbW}_{\op}\leq C$ with probability at least $1-e^{-cp}$ by Theorem \ref{theorem:matrix-concentration-covariance-operator}, we have that with probability at least $1-e^{-cp}$ that
    \[
    \sup_{\la\geq p^{1/4}}\big\|(\bbW\bbW^{\sT}+\la \bI)^{-1}\bbW\beta^\star\big\|_{\infty}\leq C p^{-1/4}\,,
    \]
   which concludes the proof.
\end{proof}
Having Lemmas \ref{lem:ptwise:bound} and \ref{lem:large:lambda} in hand, we show Lemma \ref{lem:cov:RF:tilde:delocalized} by a covering argument.
\paragraph{Proof of Lemma \ref{lem:cov:RF:tilde:delocalized}} Arguing the same as in \eqref{eq:cov:RF:tech}, we have that $\big\|(\bbW\bbW^{\sT}+\la \bI)^{-1}\bbW\beta^\star\big\|_{2}\ge \Omega(1)$ w.h.p., thus we aim to argue that $\sup_{\lambda\geq \lambda_0} \big\|(\bbW\bbW^{\sT}+\la \bI)^{-1}\bbW\beta^\star\big\|_{\infty}\to 0$ holds almost surely.

Denote the random function  $f(\la;\bbW):= \big\|(\bbW\bbW^{\sT}+\la \bI)^{-1}\bbW\beta^\star\big\|_{\infty}$. We first show that on the event $\big\{\big\|\bbW\big\|_{\op}\leq C\big\}$, which happens with probability at least $1-e^{-cp}$, the function $\lambda\to f(\lambda;\bbW)$ for $\la \geq \la_0$ is $O(1)$-Lipschitz. To this end, consider $\la\geq \la_0$ and $\delta >0$. Then, by a triangular inequality,
\[
\begin{split}
\big|f(\la+\delta; \bbW)-f(\la;\bbW)\big|
&\leq \Big\|\big(\bbW\bbW^{\sT}+(\la+\delta) \bI\big)^{-1}\bbW\beta^\star-(\bbW\bbW^{\sT}+\la \bI)^{-1}\bbW\beta^\star\Big\| _{\infty}\\
&=\delta \cdot \Big\|\big(\bbW\bbW^{\sT}+(\la+\delta) \bI\big)^{-1}(\bbW\bbW^{\sT}+\la \bI)^{-1}\bbW\beta^\star\Big\|_{\infty}\,.
\end{split}
\]
Using the crude bound $\norm{\ba}_{\infty}\leq \norm{\ba}_{2}$, we further bound
\[
\begin{split}
\big|f(\la+\delta; \bbW)-f(\la;\bbW)\big|
&\leq \delta \cdot \Big\|\big(\bbW\bbW^{\sT}+(\la+\delta) \bI\big)^{-1}(\bbW\bbW^{\sT}+\la \bI)^{-1}\bbW\beta^\star\Big\|_{2}\\
&\leq \delta \cdot \frac{\norm{\bbW}_{\op}}{\la(\la+\delta)}\\
&\leq \delta\cdot \frac{C}{\la_0^2}\,.
\end{split}
\]
Hence, on the w.h.p. event $\{\norm{\bbW}_{\op}\leq C\}$, the function $\la \to f(\la;\bbW)$ for $\la\geq \la_0$ is $C/\la_0^2$-Lipschitz.

Now, denote the set 
\begin{equation}\label{eq:def:Lambda}
\Lambda:=\big\{\lambda_0+p^{-1/4}i: i \in [0,\sqrt{p}]\cap \mathbb{Z}\big\}\cup[p^{1/4},\infty)\,.
\end{equation}
Then, by Lemma \ref{lem:ptwise:bound}, Lemma \ref{lem:large:lambda} and a union bound, we have that
\begin{equation}\label{eq:covering:tech:1}
\P\bigg(\sup_{\la\in \Lambda} f(\la;\bbW)\geq Cp^{1/4}\bigg)\leq (\sqrt{p}+1)e^{-c\sqrt{p}}+e^{-cp}\leq e^{-c^\prime \sqrt{p}}\,.
\end{equation}
Furthermore, for $\la \notin \Lambda_1$ and $\la\geq \la_0$, there exists $\la^\prime$ such that $|\la-\la^\prime|\leq p^{-1/4}$. Since $\la \to f(\la;\bbW)$ for $\la\geq \la_0$ is $C^\prime$-Lipshitz with probability at least $1-e^{-cp}$, this implies that
\[
\P\bigg(\sup_{\la\geq \la_0} f(\la;\bbW)\geq \sup_{\la\in \Lambda} f(\la;\bbW)+C^\prime p^{-1/4}\bigg)\leq e^{-cp}\,.
\]
Combining with \eqref{eq:covering:tech:1}, we have that
\[
\P\bigg(\sup_{\la\geq \la_0} f(\la;\bbW)\geq C^{\prime\prime} p^{1/4}\bigg)\leq e^{-c^{\prime\prime} \sqrt{p}}\,,
\]
which concludes the proof of Lemma \ref{lem:cov:RF:tilde:delocalized}.

\subsection{Proof of Lemma \ref{lem:4norm:separable:loss:lower}}
\label{subsec:lem:4norm:separable:loss:lower}
The proof of Lemma \ref{lem:4norm:separable:loss:lower} is based on Gordon's inequality \cite{Gordon88, thrampoulidis2015, montanari2019generalization}. Throughout, we fix $s_0\equiv \frac{1}{12}$ so that $\btheta \to \norm{\btheta}_2^2-s_0\norm{\btheta}^4_4$ is strictly convex for $\norm{\btheta}_2\leq 2$. We abbreviate $\widehat{F}(\bG,\eps)\equiv \widehat{F}_{3}(\bG, \eps)$ for simplicity.

By \cite[Lemma 6.1 and Lemma 6.2]{montanari2019generalization}, which are consequences of Gordon's inequality and Gaussian concentration, we have for every $t>0$ that
\begin{equation}\label{eq:gordon}
    \P\left(\widehat F(\bG,\eps)^{\frac{1}{2}}\leq t\right) \leq 2 \P\left(\widehat G(\bg,\eps)\leq 2t\right)+ o_n(1),
\end{equation}
where $\widehat G(\bg,\eps)$ is the surrogate Gordon's optimization problem corresponding to $\widehat F(\bG,\eps)$: for $\bg\sim \normal(\bzero, \id_p)$, define
\begin{equation}\label{eq:gordon:opt}
    \widehat G(\bg,\eps)\equiv \min_{\substack{\norm{\btheta}_2^2 - s_0 \norm{\btheta}_4^{4}\leq 1-\eps \\ \norm{\btheta}_2\leq 2}} F_{\kappa^\star} \left(\langle \btheta, \bSigma^{1/2}\bw\rangle, 
			\norm{\pw\bSigma^{1/2} \btheta}_2 \right) 
				+ \frac{1}{\sqrt{n}} \bg^{\sT} \pw \bSigma^{1/2}\btheta.
\end{equation}
Here, we denoted 
\begin{equation}\label{eq:def:rho:w}
    \bw = \frac{\bSigma^{1/2} \btheta^\star}{\rho}\,,\quad\textnormal{where}\quad \rho:=\norm{\bSigma^{1/2}\btheta^\star}_2\,,\quad\textnormal{and}\quad\pw=\id -\bw \bw^{\sT}\,.
\end{equation}
In particular, $\pw$ is the projection matrix on to the orthogonal space of $\bw$. The definition of $F_{\kappa}(c_1,c_2)$ is given in \eqref{eqn:def-F-kappa}. It can be shown that $(c_1,c_2)\to F_{\kappa}(c_1,c_2)$ is convex (see \cite[Lemma 6.3]{montanari2019generalization}), thus the optimization problem $\widehat G(\bg,\eps)$ in \eqref{eq:gordon:opt} is convex.
By analyzing it's KKT condition, we establish the following lemma in Section \ref{subsec:proof:lem:minimizer:gordon:deloc}
\begin{lemma}\label{lem:minimizer:gordon:deloc}
Suppose $(\bSigma,\btheta^\star)$ satisfies \eqref{eq:assumption:Sigma} and let $\hbtheta_{\eps}=(\widehat{\theta}_{\eps,i})_{i\leq p}$ be the minimizer for the optimization problem in \eqref{eq:gordon:opt}. Then we have $\norm{\hbtheta_{\eps}}_4\stackrel{p}{\to} 0$.
\end{lemma}
We now show that Lemma \ref{lem:minimizer:gordon:deloc} implies Lemma \ref{lem:4norm:separable:loss:lower}.
\paragraph{Proof of Lemma \ref{lem:4norm:separable:loss:lower}}
Since $\eps \to \widehat F_3(\bG,\eps)$ is monotonically increasing, it suffices to establish Lemma \ref{lem:4norm:separable:loss:lower} for small enough $\eps$. To this end, we fix $\eps\in (0,1)$. Let $\widetilde{G}_{\eps}\equiv  \widetilde{G}_{\eps}(\bg)$ be
\begin{equation*}
\begin{split}
    \widetilde{G}_{\eps}
    &\equiv \min_{ \norm{\btheta}_2^2 \leq 1-\frac{\eps}{2}} F_{\kappa^\star} \left(\langle \btheta, \bSigma^{1/2}\bw\rangle, 
			\norm{\pw\bSigma^{1/2} \btheta}_2 \right) 
				+ \frac{1}{\sqrt{n}} \bg^{\sT} \pw \bSigma^{1/2}\btheta\\
	&\stackrel{(a)}{=}\left(1-\frac{\eps}{2}\right)^{1/2} \cdot \min_{ \norm{\btheta}_2\leq 1} F_{\kappa^\star(\eps)} \left(\langle \btheta, \bSigma^{1/2}\bw\rangle, 
			\norm{\pw\bSigma^{1/2} \btheta}_2 \right) 
				+ \frac{1}{\sqrt{n}} \bg^{\sT} \pw \bSigma^{1/2}\btheta,
\end{split}
\end{equation*}
where $\kappa^\star(\eps)\equiv \kappa^\star\left(1-\frac{\eps}{2}\right)^{-1/2}>\kappa^\star$ and $(a)$ holds by the homogeneity of the map $(\kappa,c_1,c_2)\to F_{\kappa}(c_1,c_2)$. A crucial observation is that on the event where the minimizer of \eqref{eq:gordon:opt} $\hbtheta_{\eps}$ satisfies $\norm{\hbtheta_{\eps}}_4^4\leq \frac{\eps}{2s_0}$, which holds w.h.p. by Lemma \ref{lem:minimizer:gordon:deloc}, we have that $\widehat{G}(\bg,\eps)\geq \widetilde{G}_{\eps}$.

Now, take $\eta\equiv \eta(\eps)$ small enough so that $\kappa^\star(\eps)>\kappa^\star + \eta$. By Lemma \ref{lem:partial:kappa:lower} in Section \ref{subsec:properties:F:kappa}, we have that $\min _{\kappa>0, c_1\in \R, c_2\geq 0} \partial_{\kappa}F_{\kappa}(c_1,c_2)\geq C>0$ for some constant $C\equiv C(\rho^\star,f^\star)$. Moreover, \cite[Proposition 6.4]{montanari2019generalization} shows that
\begin{equation*}
   \min_{\norm{\btheta}_2\leq 1} F_{\kappa^\star} \left(\langle \btheta, \bSigma^{1/2}\bw\rangle, 
			\norm{\pw\bSigma^{1/2} \btheta}_2 \right) 
				+ \frac{1}{\sqrt{n}} \bg^{\sT} \pw \bSigma^{1/2}\btheta\pto 0.
\end{equation*}
Hence, we have with high probability that
\begin{equation}\label{eq:lower:tilde:G}
    \widetilde{G}_{\eps} \geq \left(1-\frac{\eps}{2}\right)^{1/2} \frac{C\eta}{2}\geq \frac{C\eta}{4}.
\end{equation}
Therefore, if we let $\delta(\eps):= \left(\frac{C\eta}{8}\right)^{2}$, then \eqref{eq:gordon} shows that
\begin{equation*}
\begin{split}
    \P\left(\widehat{F}(\bG,\eps)\leq \delta(\eps)\right)
    &\leq 2\P \left(\widehat{G}(\bg,\eps)\leq \frac{C\eta}{4}\right)+o(1)\\
    &\leq 2\P \left(\widetilde{G}_{\eps}\leq \frac{C\eta}{4}\right)+2\P\left(\norm{\hbtheta_{\eps}}_4^4\leq \frac{\eps}{2s_0}\right)+o(1)\\
    &\stackrel{(a)}{=}o(1),
\end{split}
\end{equation*}
where $(a)$ is by \eqref{eq:lower:tilde:G} and Lemma \ref{lem:minimizer:gordon:deloc}. This concludes the proof.

\subsubsection{Proof of Lemma \ref{lem:minimizer:gordon:deloc}}
\label{subsec:proof:lem:minimizer:gordon:deloc}
First, note that $\norm{\hat{\btheta}_\eps}_2\leq 2$ holds by primal feasibility. Further, $\norm{\hat{\btheta}_{\eps}}_4\leq \norm{\hat{\btheta}_{\eps}}_{\infty}^{1/2} \norm{\hat{\btheta}_\eps}_2^{1/2}$ holds, thus it suffices to show that $\norm{\hat{\btheta}_{\eps}}_{\infty}\pto 0$ to prove Lemma \ref{lem:minimizer:gordon:deloc}. Throughout this subsection, we assume the condition \eqref{eq:assumption:Sigma}.

Denote $c_1\equiv \langle \hbtheta_\eps, \bSigma^{1/2} \bw \rangle$  and $c_2\equiv\norm{\pw\bSigma^{1/2} \hbtheta_{\eps}}_2$. Note that Lemma \ref{lem:minimizer:gordon:deloc} is immediate when $c_2=0$: suppose $c_2=0$, which implies that $\bSigma^{1/2} \hbtheta_\eps$ is parallel to $\bw$. Thus $\hbtheta_{\eps}=\alpha \btheta^\star $for some $\alpha\in \R$. Since $\norm{\hbtheta_{\eps}}_2\leq 2$ holds by feasibility of the optimization problem \eqref{eq:gordon:opt} and $\norm{\btheta^\star}_2=1$ holds, we have $|\alpha|\leq 2$. Hence, 
\[\big\|\hbtheta_{\eps}\big\|_{\infty}=|\alpha| \norm{\btheta^\star}_{\infty}\leq 2\norm{\btheta^\star}_{\infty}\to 0\,,
\]
where the last convergence is due to our assumption \eqref{eq:assumption:Sigma}.

To this end, we assume $c_2>0$ and show that $\big\|\hat{\btheta}_{\eps}\big\|_{\infty}\pto 0$. Then,the KKT condition of the convex optimization problem in \eqref{eq:gordon:opt} reads
\begin{equation}\label{eq:KKT}
\begin{split}
    &\partial_1 F_{\kappa^\star}(c_1,c_2)\bSigma^{1/2}\bw+\frac{\partial_2 F_{\kappa^\star}(c_1,c_2)}{c_2} \bSigma^{1/2}\pw \bSigma^{1/2}\hbtheta_{\eps}+\frac{1}{\sqrt{n}}\bSigma^{1/2}\pw \bg \\
    &\qquad\qquad\qquad\qquad\qquad\qquad\qquad+s_1\hbtheta_{\eps}+s_2\left(\tbe-2s_0\cdot \diag\left(\big(\widehat{\theta}_{\eps,i}\big)^2\right)_{i\leq p}\tbe\right)=0\\
    & s_1\left(\norm{\tbe}_2-2\right)=0, \quad s_2\left(\norm{\tbe}_2^2 - s_0 \norm{\tbe}_4^{4}-(1-\eps)\right)=0, ~s_1,s_2\geq 0,\\
    &\norm{\tbe}_2\leq 2,\quad\quad\quad\quad\norm{\tbe}_2^2 - s_0 \norm{\tbe}_4^{4}\leq 1-\eps,
\end{split}
\end{equation}
where $\partial_i F_{\kappa}(c_1,c_2)\equiv \frac{\partial}{\partial c_i} F_{\kappa}(c_1,c_2), i=1,2$. Rearranging the first equation above gives
\begin{equation}\label{eq:tilde:theta}
\begin{split}
    &\left(\frac{\partial_2 F_{\kappa^\star}(c_1,c_2)}{c_2} \bSigma + s_1 \id+s_2\id- 2s_0s_2\cdot \diag\left(\big(\widehat{\theta}_{\eps,i}\big)^2\right)_{i\leq p}\right)\tbe\\
    &= -\left(\partial_1F_{\kappa^\star}(c_1,c_2)-c_1\frac{\partial_2 F_{\kappa^\star}(c_1,c_2)}{c_2}-\frac{1}{\sqrt{n}}\bg^{\sT}\bw \right)\bSigma^{1/2}\bw-\frac{1}{\sqrt{n}}\bSigma^{1/2}\bg.
\end{split}
\end{equation}
To simplify notations, denote
\begin{equation}\label{eq:def:A:D:zeta}
\begin{split}
    \bA:&=\frac{\partial_2 F_{\kappa^\star}(c_1,c_2)}{c_2} \bSigma + s_1 \id+s_2\id\\
    \bM:&=- 2s_0s_2\cdot \diag\left(\big(\widehat{\theta}_{\eps,i}\big)^2\right)_{i\leq p}\\
    \zeta:&=-\rho^{-1}\left(\partial_1F_{\kappa^\star}(c_1,c_2)-c_1\frac{\partial_2 F_{\kappa^\star}(c_1,c_2)}{c_2}-\frac{1}{\sqrt{n}}\bg^{\sT}\bw \right).
\end{split}
\end{equation}
We remark that by \cite[Lemma 6.3]{montanari2019generalization}, we have that $\partial_2F_{\kappa^\star}(c_1,c_2)$ is positive, thus $\bA$ and $\bA+\bM$ are invertible. Thus, \eqref{eq:tilde:theta} reads
\begin{equation}\label{eq:express:theta:eps}
\begin{split}
    \tbe 
    &= \zeta \left(\bA+\bM\right)^{-1}\bSigma^{1/2}\bw-\frac{1}{\sqrt{n}}\left(\bA+\bM\right)^{-1}\bSigma^{1/2}\bg\\
    &=\zeta\bA^{-1}\bSigma^{1/2}\bw-\zeta\left(\bA+\bM\right)^{-1}\bM\bA^{-1}\bSigma^{1/2}\bw\\
    &\quad\quad\quad\quad\quad\quad-\frac{1}{\sqrt{n}}\bA^{-1}\bSigma^{1/2}\bg+\frac{1}{\sqrt{n}}\left(\bA+\bM\right)^{-1}\bM\bA^{-1}\bSigma^{1/2}\bg,
\end{split}
\end{equation}
where the last identity is due to $\left(\bA+\bM\right)^{-1}=\bA^{-1}-\left(\bA+\bM\right)^{-1}\bM\bA^{-1}$. By a triangular inequality, it follows that
\begin{equation}\label{eq:bound:theta:eps}
\begin{split}
\norm{\tbe}_{\infty} &\leq \Xi_1+\Xi_2\quad\textnormal{where}\\
\Xi_1&:= |\zeta|\cdot \norm{\bA^{-1}\bSigma\btheta^\star}_{\infty}+\frac{1}{\sqrt{n}}\norm{\bA^{-1}\bSigma^{1/2}\bg}_{\infty}\\
\Xi_2&:=|\zeta|\cdot\norm{\left(\bA+\bM\right)^{-1}\bM\bA^{-1}\bSigma\btheta^\star}_{\infty}+\frac{1}{\sqrt{n}}\norm{\left(\bA+\bM\right)^{-1}\bM\bA^{-1}\bSigma^{1/2}\bg}_{\infty}.
\end{split}
\end{equation}
To this end, we now aim to bound $\Xi_1$ and $\Xi_2$. First, we show that $\zeta$ defined in \eqref{eq:def:A:D:zeta} is bounded w.h.p..
\begin{lemma}\label{lem:zeta:bounded}
There exists $C>0$, which does not depend on $n$, such that $|\zeta|\leq C$ with high probability.
\end{lemma}
\begin{proof}
Note that $\rho$ is bounded away from zero since
\begin{equation}\label{eq:lower:rho}
    \rho\equiv \norm{\bSigma^{1/2}\btheta^\star}_2\ge \lambda_{\min}(\bSigma)^{1/2}\norm{\btheta^\star}_2= \lambda_{\min}(\bSigma)^{1/2}>0.
\end{equation}
Thus, we have the bound
\begin{equation}\label{eq:zeta:bound:1}
    |\zeta|\leq \lambda_{\min}(\bSigma)^{-1/2}\left(\left|\partial_1F_{\kappa^\star}(c_1,c_2)-c_1\frac{\partial_2 F_{\kappa^\star}(c_1,c_2)}{c_2}\right|+\frac{\left|\bg^{\sT}\bw\right|}{\sqrt{n}}\right).
\end{equation}
Now, recall that $c_1\equiv \langle \hbtheta_\eps, \bSigma^{1/2} \bw \rangle$  and $c_2\equiv\norm{\pw\bSigma^{1/2} \hbtheta_{\eps}}_2$. Since there exists $\osla>0$ such that $\lambda_{\max}(\bSigma)\leq \osla$, we have that $|c_1|,|c_2|\leq \osla^{1/2}$ w.h.p.. Moreover, Lemma \ref{lem:partial:1:2:estimate} in Section \ref{subsec:properties:F:kappa} shows that there exists a constant $C\equiv C(\osla, \rho^\star, f^\star)$ such that
\begin{equation}\label{eq:zeta:bound:2}
    \max_{|\tc_1|\vee \tc_2\leq \osla^{1/2},\tc_2\neq 0}\left|\partial_1F_{\kappa^\star}(\tc_1,\tc_2)-\tc_1\frac{\partial_2 F_{\kappa^\star}(\tc_1,\tc_2)}{\tc_2}\right|\leq C.
\end{equation}
Meanwhile, note that $\frac{\left|\bg^{\sT}\bw\right|}{\sqrt{n}}\pto 0$ holds since $\bg^{\sT}\bw\sim \normal(0,1)$. Therefore, \eqref{eq:zeta:bound:1} and \eqref{eq:zeta:bound:2} conclude the proof.
\end{proof}

\begin{lemma}\label{lem:xi:1}
We have $\Xi_1\pto 0$.
\end{lemma}
\begin{proof}
Note that if we denote $\la_{\star}:= c_2(s_1+s_2)/\partial_2 F_{\kappa^\star}(c_1,c_2)\geq 0$, then
\begin{equation*}
\begin{split}
\Xi_1
&=\bigg(\frac{\partial_2 F_{\kappa^\star}(c_1,c_2)}{c_2}\bigg)^{-1}\bigg(|\zeta|\cdot\norm{\big(\bSigma+\la_{\star} \bI\big)^{-1}\bSigma \btheta^\star}_{\infty}+\frac{1}{\sqrt{n}}\norm{\big(\bSigma+\la_{\star}\bI\big)^{-1}\bSigma^{1/2} \bg}_{\infty}\bigg)\\
&\leq \bigg(\frac{\partial_2 F_{\kappa^\star}(c_1,c_2)}{c_2}\bigg)^{-1}\cdot \sup_{\la\geq 0} \bigg\{|\zeta|\cdot\norm{\big(\bSigma+\la \bI\big)^{-1}\bSigma \btheta^\star}_{\infty}+\frac{1}{\sqrt{n}}\norm{\big(\bSigma+\la \bI\big)^{-1}\bSigma^{1/2} \bg}_{\infty}\bigg\}\,.
\end{split}
\end{equation*}
Lemma \ref{lem:partial:1:2:estimate} below shows that there exists a constant $c>0$, which does not depend on $n,p$ such that $\frac{\partial_2 F_{\kappa^\star}(c_1,c_2)}{c_2}\geq c$. Moreover, since $|\zeta|\leq C $ holds w.h.p. by Lemma \ref{lem:zeta:bounded}, the right hand side of the display above converges in probability to $0$ by our assumption \eqref{eq:assumption:Sigma}.
\end{proof}
Finally, we show by crude estimates that $\Xi_2\pto 0$.
\begin{lemma}\label{lem:xi:2}
We have $\Xi_2\pto 0$.
\end{lemma}
\begin{proof}
Observe that when $s_2=0$, then $\bM=0$, thus we assume without loss of generality that $s_2\neq 0$.

We start with the first term of $\Xi_2$: by the crude bound $\norm{\ba}_{\infty} \leq \norm{\ba}_2$ for $\ba\in \R^p$, we have 
\begin{equation}\label{eq:first:piece}
\begin{split}
    |\zeta|\cdot \norm{\left(\bA+\bM\right)^{-1}\bM\bA^{-1}\bSigma^{1/2}\bw}_{\infty}
    &\leq |\zeta|\cdot\norm{\left(\bA+\bM\right)^{-1}\bM\bA^{-1}\bSigma^{1/2}\bw}_2\\
    &\leq |\zeta|\cdot\norm{\left(\bA+\bM\right)^{-1}}_{\op}\norm{\bM\bA^{-1}\bSigma^{1/2}\bw}_2.
\end{split}
\end{equation}
Note that from the definition of $\bA$ and $\bM$ in \eqref{eq:def:A:D:zeta}, 
\begin{equation}\label{eq:AplusD:lower}
\begin{split}
    \bA+\bM
    &=\frac{\partial_2 F_{\kappa^\star}(c_1,c_2)}{c_2} \bSigma + s_1 \id+s_2\left(\id- 2s_0\cdot \diag\left(\big(\widehat{\theta}_{\eps,i}\big)^2\right)_{i\leq p}\right)\\
    &\succeq s_2\left(\id- 2s_0\cdot \diag\left(\big(\widehat{\theta}_{\eps,i}\big)^2\right)_{i\leq p}\right)\\
    &\stackrel{(a)}{\succeq} \frac{2s_2}{3}\bI,
\end{split}
\end{equation}
where $(a)$ holds because for every $i\leq p$, we have
\begin{equation*}
    1-2s_0\widehat{\theta}_{\eps,i}^2=1-\frac{\widehat{\theta}_{\eps,i}^2}{6}\geq  1-\frac{\norm{\widehat{\btheta}_{\eps}}_2^2}{6}\geq \frac{2}{3}.
\end{equation*}
Moreover, since $\bM$ is a diagonal matrix with entries $-\frac{s_2\widehat{\theta}_{\eps,i}^2}{6}, i\leq p$, we can bound $\norm{\bM\bA^{-1}\bSigma^{1/2}\bw}_2$ by
\begin{equation*}
    \norm{\bM\bA^{-1}\bSigma^{1/2}\bw}_2\leq \norm{\bM}_{F} \norm{\bA^{-1}\bSigma^{1/2}\bw}_{\infty}=\frac{s_2}{6}\norm{\widehat{\btheta}_{\eps}}_4^2 \norm{\bA^{-1}\bSigma^{1/2}\bw}_{\infty},
\end{equation*}
where $\norm{\bM}_{F}$ denotes the frobenius norm of $\bM$. By the crude bound $\big\|\widehat{\btheta}_{\eps}\big\|_4^2\leq \big\|\widehat{\btheta}_{\eps}\big\|_2^2\leq 2$, we can further bound
\begin{equation}\label{eq:D:A:inverse:upper}
     \norm{\bM\bA^{-1}\bSigma^{1/2}\bw}_2\leq \frac{s_2}{3}\norm{\bA^{-1}\bSigma^{1/2}\bw}_{\infty}.
\end{equation}
By plugging in \eqref{eq:AplusD:lower} and \eqref{eq:D:A:inverse:upper} into \eqref{eq:first:piece}, we have that
\begin{equation*}
     |\zeta|\cdot \norm{\left(\bA+\bM\right)^{-1}\bM\bA^{-1}\bSigma^{1/2}\bw}_4\leq \frac{|\zeta|}{2}\norm{\bA^{-1}\bSigma^{1/2}\bw}_{\infty}.
\end{equation*}
Next, we consider the second term in $\Xi_2$. Repeating the same argument as illustrated in \eqref{eq:first:piece}, \eqref{eq:AplusD:lower}, and \eqref{eq:D:A:inverse:upper}, we have that
\begin{equation*}
    \frac{\norm{\left(\bA+\bM\right)^{-1}\bM\bA^{-1}\bSigma^{1/2}\bg}_4}{\sqrt{n}}\leq \frac{1}{2\sqrt{n}}\norm{\bA^{-1}\bSigma^{1/2}\bg}_4.
\end{equation*}
Therefore, it follows that $\Xi_2\leq \Xi_1/2$, and since $\Xi_1\pto 0$ by Lemma \ref{lem:xi:1}, the conclusion follows.
\end{proof}
Finally, the proof of Lemma \ref{lem:minimizer:gordon:deloc} is immediate from Lemma \ref{lem:xi:1} and Lemma \ref{lem:xi:2}.
\paragraph{Proof of Lemma \ref{lem:minimizer:gordon:deloc}}
As illustrated at the start of this subsection, it suffices to consider the case where $c_2>0$ and show that $\big\|\hbtheta_\eps\big\|_{\infty}\pto 0$. In this case, we have $\big\|\hbtheta_\eps\big\|_{\infty}\leq \Xi_1+\Xi_2$ by equation \eqref{eq:bound:theta:eps}. Meanwhile, $\Xi_1$ and $\Xi_2$ converges in probability to $0$ by Lemmas \ref{lem:xi:1} and \ref{lem:xi:2}, which concludes the proof.

\subsection{Properties of $F_{\kappa}$}
\label{subsec:properties:F:kappa}
Recall the function $F_{\kappa}(c_1,c_2)$ defined in \eqref{eqn:def-F-kappa}. In this subsection, we prove new properties of $F_{\kappa}(\cdot,\cdot)$ which were used in Section \ref{subsec:lem:4norm:separable:loss:lower}.
\begin{lemma}\label{lem:partial:kappa:lower}
There exists a constant $C>0$, which only depends on $f^\star, \rho^\star$, such that
\begin{equation}\label{eq:lem:partial:kappa:lower}
    \min_{\kappa>0,c_1\in \R, c_2\geq 0} \partial_{\kappa} F_{\kappa}(c_1,c_2)\geq C>0.
\end{equation}
\end{lemma}
\begin{proof}
The function $F_{\kappa}(c_1,c_2)$ defined in \eqref{eqn:def-F-kappa} depends only on the parameters $\kappa,c_1,c_2,f^\star,\rho^\star$, so it suffices prove that $\min_{\kappa>0,c_1\in \R, c_2\geq 0} \partial_{\kappa} F_{\kappa}(c_1,c_2)>0$. By a direct calculation, we have
\begin{equation}\label{eq:express:partial:kappa}
    \partial_{\kappa}F_{\kappa}(c_1,c_2)=\frac{\E\left[\left(\kappa-c_1 YG-c_2 Z\right)_{+}\right]}{\left(\E\left[\left(\kappa-c_1 YG-c_2 Z\right)_{+}^2\right]\right)^{1/2}}=\partial_{\kappa}F_{t\kappa}(tc_1,tc_2),
\end{equation}
where $t>0$ is arbitrary. Thus, without loss of generality, we can fix $\kappa=1$ in proving \eqref{eq:lem:partial:kappa:lower}. Note that for any fixed $c_1\in \R, c_2\geq 0$, $\partial_{\kappa}F_1(c_1,c_2)$ is strictly positive by the expression above since for any $c_1,c_2$, $1-c_1YG-c_2 Z$ is positive with positive probability.

Now, suppose by contradiction that $\min_{c_1\in \R, c_2\geq0}\partial_{\kappa}F_1(c_1,c_2)=0$. Since $(\kappa,c_1,c_2) \to \partial_{\kappa}F_{\kappa}(c_1,c_2)$ is continuous by the expression \eqref{eq:express:partial:kappa}, this implies that there exists a sequence $\left\{(c_{1,m}, c_{2,m})\right\}_{m\geq 1}$ such that 
\begin{equation}\label{eq:contradict}
    \lim_{m\to\infty}\partial_{\kappa}F_1(c_{1,m},c_{2,m})=0\quad\textnormal{and}\quad \lim_{m\to\infty} c_{1,m}^2+c_{2,m}^2=\infty.
\end{equation}
Denote $d_{i,m}=\frac{c_{i,m}}{\sqrt{c_{1,m}^2+c_{2,m}^2}}$ for $i=1,2$. Since $d_{1,m}^2+d_{2,m}^2=1$, there exists a subsequence $\{m_{\ell}\}_{\ell\geq 1}$ such that $\left\{(d_{1,m_{\ell}},d_{1,m_{\ell}})\right\}_{\ell \geq 1}$ converges to $(d_{1,\infty}, d_{2,\infty})$ for some $d_{1,\infty}^2+d_{2,\infty}^2=1$. By continuity of $(\kappa,c_1,c_2)\to \partial_{\kappa}F_{\kappa}(c_1,c_2)$ for $(\kappa,c_1,c_2)\neq (0,0,0)$, it follows that
\begin{equation}\label{eq:contradict:2}
    0\stackrel{(a)}{=}\lim_{\ell\to\infty}\partial_{\kappa}F_1(c_{1,m_{\ell}},c_{2,m_{\ell}})\stackrel{(b)}{=}\lim_{\ell\to\infty}\partial_{\kappa}F_{(c_{1,m_{\ell}}^2+c_{2,m_{\ell}}^2)^{-1/2}}(d_{1,m_{\ell}},d_{2,m_{\ell}})\stackrel{(a)}{=}\partial_{\kappa}F_{0}(d_{1,\infty},d_{2,\infty}),
\end{equation}
where $(a)$ follows from \eqref{eq:contradict} and $(b)$ follows from \eqref{eq:express:partial:kappa}. \eqref{eq:contradict:2} is a contradiction by the fact $\partial_{\kappa}F_0(d_{1,\infty},d_{2,\infty})>0$, which concludes the proof.
\end{proof}
\begin{lemma}\label{lem:partial:1:2:estimate}
We have that
\begin{equation}\label{eq:lem:partial:1:2:estimate:1}
    \sup_{\kappa>0, c_1\in \R,c_2>0}\left|\partial_1F_{\kappa}(c_1,c_2)\right|\leq 1.
\end{equation}
Moreover, for any $\lambda_{\max}>0$, there exist constants $c, C>0$, which depend only on $f^\star, \rho^\star,$ and $\lambda_{\max}>0$, such that
\begin{equation}\label{eq:lem:partial:1:2:estimate:2}
    c\leq \inf_{|c_1|, c_2\leq \lambda_{\max}, c_2\neq 0}\frac{\partial_2 F_{\kappa^\star}(c_1,c_2)}{c_2}\leq \sup_{|c_1|, c_2\leq \lambda_{\max}, c_2\neq 0}\frac{\partial_2 F_{\kappa^\star}(c_1,c_2)}{c_2}\leq C.
\end{equation}
\end{lemma}
\begin{proof}
Since $(c_1,c_2)\to F_{\kappa}(c_1,c_2)$ is convex by \cite[Lemma 6.3]{montanari2019generalization}, we have that $c_1\to \partial_1F_{\kappa}(c_1,c_2)$ is increasing for fixed $\kappa>0$ and $c_2>0$. Meanwhile, we have by L'H\^opital's rule that for any $\kappa, c_2>0$,
\begin{equation*}
    \lim_{c_1\to\infty }\partial_1F_{\kappa}(c_1,c_2)=\lim_{c_1\to\infty }\frac{F_{\kappa}(c_1,c_2)}{c_1}=\left(\E\left[(-YG)_{+}^2\right]\right)^{1/2}\leq \left(\E G^2\right)^{1/2}=1. 
\end{equation*}
Analogously,
\begin{equation*}
    \lim_{c_1\to-\infty }\partial_1F_{\kappa}(c_1,c_2)=\lim_{c_1\to-\infty }\frac{F_{\kappa}(c_1,c_2)}{c_1}=-\left(\E\left[(YG)_{+}^2\right]\right)^{1/2}\geq -\left(\E G^2\right)^{1/2}=-1.
\end{equation*}
Therefore, our first assertion \eqref{eq:lem:partial:1:2:estimate:1} holds.

We now prove the second assertion. First, note that $\partial_2F_{\kappa}(c_1,0)=0$ for any $\kappa>0, c_1\in \R$: by a direct calculation,
\begin{equation}\label{eq:partial:2:zero}
    \partial_2 F_{\kappa}(c_1,0)=\frac{\E\big[(\kappa-c_1 YG)_{+}(-Z)\big]}{\sqrt{\E\big[(\kappa-c_1 YG)_{+}^2\big]}}=0.
\end{equation}
Define $G(c_1,c_2)\equiv G_{\kappa^\star}(c_1,c_2):\R\times \R_{\geq 0} \to \R$ by
\begin{equation*}
    G(c_1,c_2):=
    \begin{cases}
    \frac{\partial_2 F_{\kappa^\star}(c_1,c_2)}{c_2} &\textnormal{if}\quad c_2>0,\\
    \partial_{22}F_{\kappa^\star}(c_1,0)&\textnormal{if}\quad c_2=0.\\
    \end{cases}
\end{equation*}
\cite[Lemma 6.3]{montanari2019generalization} shows that $F_{\kappa}(\cdot,\cdot)$ is strictly convex and strictly increasing in the second coordinate, thus $G(c_1,c_2)$ is positive for every $c_1\in \R, c_2\geq 0$. Moreover, we have $\partial_2F_{\kappa^\star}(c_1,0)=0$ by \eqref{eq:partial:2:zero}, thus $G(c_1,c_2)$ is continous throughout the domain $\R\times \R_{\geq 0}$. Thus, if we let
\begin{equation*}
    c:=\min_{|c_1|,c_2\in [0,\lambda_{\max}]} G(c_1,c_2)\quad \textnormal{and}\quad C:=\max_{|c_1|,c_2\in [0,\lambda_{\max}]} G(c_1,c_2),
\end{equation*}
then $c,C$ are well-defined in $\R_{+}$ which depend only on $f^\star, \rho^\star,$ and $\lambda_{\max}$, and satisfies our second assertion \eqref{eq:lem:partial:1:2:estimate:2}. This concludes the proof.
\end{proof}

\section{Universality of the test error}
Here, we complete the proof of Theorem~\ref{thm:test} as outlined in Section~\ref{subsec:proof:test:error} by establishing 
Propositions~\ref{proposition:asymptotic-equivalence-for-regular-estimators}-\ref{proposition:delocalization-for-general-max-margin}.

\subsection{Proof of Proposition~\ref{proposition:asymptotic-equivalence-for-regular-estimators}}
The result for both the random features model and the features with independent entries
are proven similarly. However, the random features requires additional conditioning on the weight matrix $\bW_\RF$. Here
we provide a proof under this model.

We proceed by directly approximating indicator functions using 
Lipschitz functions and applying Lemma~\ref{lemma:pointwise-gaussianity}.
First, note that it is sufficient to establish the following uniform result:
\begin{equation}
 \lim_{n\to\infty}\sup_{\norm{\btheta}_\infty \le \delta_n, C_0 \ge \norm{\btheta}_2 \ge c_0} |R_n^\bx(\btheta)  -R_n^\bg(\btheta)| = 0
\end{equation}
where $\delta_n \to 0 $ and $C_0 \ge c_0 >0$ are deterministic. 
To simplify notation, we let $\cS_n$ be the collection of $\btheta$ defined
by the constraints above.
%
We work conditionally on the weight matrix $\bW = \bW_\RF$ with  $\bW\in \Omega_1$.
Given $\delta>0$, let $\psi_\delta$ be a Lipcshitz function so that for all $t\in\R$,
\begin{equation}
\label{eq:indicator_lipschitz_sandwich}
 \one_{t\le -\delta} \le \psi_{-\delta}(t)  \le  \one_{t\le 0} \le \psi_{\delta}(t) \le \one_{t\le \delta},
\end{equation}
where the Lipschitz constant will of course depend on $\delta$. 
First note that for any Lipschitz test function $\varphi$ we have 
\begin{equation}
\label{eq:lipschitz_test_error}
\sup_{\btheta\in \cS_n, \bW \in \Omega_1}\left|\E\left[\varphi(y_\new \bx_\new^\sT\btheta)|\bW\right]  -\E\left[\varphi(y_\new \bg_\new^\sT\btheta)|\bW\right]\right| \to 0
\end{equation}
as $n\to\infty$. 
Indeed, since we can write $y_\new = 2\one_{ \{f(\bz_\new^\sT \bbeta^\star)  \ge u_\new \}} - 1$
for $u_\new\sim \Unif([0,1])$, we have for any $\delta>0$,
\begin{align*}
 &\left|\E\left[\varphi(y_\new \bx_\new^\sT\btheta)|\bW\right]  -\E\left[\varphi(\psi_\delta(u_\new - f(\bz_\new^\sT \bbeta^\star)) \bx_\new^\sT\btheta)|\bW\right]\right|\\
 &\le\norm{\varphi}_\Lip\E[(\bx_\new^\sT\btheta)^2 \big|\bW]^{1/2} \P\left(u_\new \in[0,\delta]\right)^{1/2}\\
 &\le C_0\norm{\varphi}_\Lip  \norm{\bSigma_{\bx|\bW}}_\op^{1/2} \delta^{1/2}.
\end{align*}
But
$ \sup_{\bW\in \Omega_1}\big|\norm{\bSigma_{\bx|\bW}}_\op -\norm{\mu_1 \bW\bW^\sT + \mu_2 \bI}_\op \big| \to 0$ on $\Omega_1$
by~\cite[Lemma 5]{HuLu22}
and $\bzeta := (\btheta^\sT,\bbeta^\sT)^\sT$ satisfies the assumptions of Lemma~\ref{lemma:pointwise-gaussianity}. Furthermore, 
\begin{equation}
(\btheta^\sT \bx_\new, \bbeta^{\star\sT} \bz_\new)  \mapsto \E\left[\varphi(\psi_\delta(u_\new - f(\bz_\new^\sT \bbeta^\star)) \bx_\new^\sT\btheta)|\bW,\bz_\new,\bx_\new\right]
\end{equation}
is Lipschitz, 
so applying a similar argument to $\bg_\new$, then
taking $n\to\infty$ followed by $\delta\to 0$ proves Eq.~\eqref{eq:lipschitz_test_error}.
Now to deduce Proposition~\ref{proposition:asymptotic-equivalence-for-regular-estimators}, we use Eq.~\eqref{eq:indicator_lipschitz_sandwich}
to write
\begin{equation}
 \E\left[\psi_{-\delta}(y_\new \bx_\new^\sT\btheta)|\bW\right]\le
 \P\left( y_\new \bx_\new^\sT \btheta \le 0|\bW \right)  \le \E\left[\psi_\delta(y_\new \bx_\new^\sT\btheta)|\bW\right].
\end{equation}
Taking $n\to\infty$ and using Lemma~\ref{lemma:pointwise-gaussianity} and Eq.~\eqref{eq:indicator_lipschitz_sandwich} again we have
\begin{align}
\lim_{n\to\infty} \sup_{\btheta\in \cS_n,\bW \in\Omega_1} 
\label{eq:sandwich}
\P\left( y_\new \bg_\new^\sT \btheta \le -\delta\Big|\bW\right)
&\le
\lim_{n\to\infty}  
\sup_{\btheta\in\cS_n,\bW \in\Omega_1} 
\P\left( y_\new \bx_\new^\sT \btheta \le 0\Big|\bW\right)\\
&\le
\lim_{n\to\infty}  
\sup_{\btheta\in\cS_n,\bW \in\Omega_1} 
\P\left( y_\new \bg_\new^\sT \btheta \le \delta\Big|\bW\right).
\end{align}
But the upper and lower bounds in the previous display can be shown to be close since
\begin{align}
&\sup_{\btheta\in\cS_n, \bW\in\Omega_1} 
\left|\P\left( y_\new \bg_\new^\sT \btheta \le \delta\Big|\bW\right)-
\P\left( y_\new \bg_\new^\sT \btheta \le -\delta\Big|\bW\right)\right|\\
&\hspace{40mm}\le \sup_{\btheta\in\cS_n, \bW\in\Omega_1}  \P\left( \norm{\left(\mu_1 \bW\bW^\sT + \mu_2 \bI_p\right)^{1/2}\btheta}_2 |G| \le \delta\Big|\bW \right) \\
&\hspace{40mm}\le \sup_{\btheta\in\cS_n, \bW\in\Omega_1}  2\delta \left(\sigma_{\min}\left(\mu_1 \bW\bW^\sT + \mu_2 \bI_p\right)^{1/2}\norm{\btheta}_2\right)^{-1} \le \frac{2\delta}{\mu_2^{1/2}  c_0}
\end{align}
Since this holds for all $\delta>0$, and noting that Eq.~\eqref{eq:sandwich} holds with $\bg_\new$ replacing $\bx_\new$, we have
\begin{equation}
\lim_{n\to\infty}
\sup_{\btheta\in\cS_n, \bW\in\Omega_1} 
\left|\P\left( y_\new \bg_\new^\sT \btheta \le 0\big|\bW\right)
-
\P\left( y_\new \bx_\new^\sT \btheta \le 0 \big|\bW\right)\right| = 0.
\end{equation}
Now since $\Omega_1$ holds with high probability by Lemma~\ref{lem:basic:op:norm:RF}, the claim of the proposition follows.

\subsection{Proof of Proposition~\ref{proposition:universality-of-the-geometric-structure}}
Once again, the proofs for the random features model and the independent entries model 
are similar modulo additional conditioning arguments required for the random features. 
We hence provide the proof under this model.

Let us fix $\sgamma>\sgamma^\star$ so that the data is separable with high probability for both $\bD=\bX_{\RF}$ and $\bD=\bG_\RF$.
We will work conditionally on the event
 $\{\lambda_{\max}(\bSigma_{\bg})\leq \oslambda\}$. By Lemma~\ref{lem:basic:op:norm:RF}, this event has high probability.
To establish the proposition, we follow the standard perturbation argument~\cite{MontanariNg17}.
Namely, for $l\in\{1,2\}$, $s$ sufficiently small and $b>0$ consider the following ``perturbed'' quantity:
\begin{equation}\label{eq:def:H}
\widehat H_{l}(\bD; s,\kappa) := \min_{\{\btheta : y_i\langle \bd_i, \btheta\rangle  \ge \kappa \, \forall i \leq n,\norm{\btheta}_\infty \le \epsilon_n\}} \norm{\btheta}_2^2 + s \pi _{l}(\btheta)
\end{equation}
where $\epsilon_n$ in this definition is taken so that 
\begin{equation}
 \P\left(\norm{\widehat\btheta_\MM}_\infty  > \epsilon_n \right) \to 0.
\end{equation}
Let us denote by $\widehat\btheta_l(\bD;s,\kappa)$ the minimizer of the above loss
and
note that the max-margin classifier satisfies
$\widehat\btheta_l(\bD; 0, \widehat \kappa)  = \widehat \btheta_\MM(\bD).$ Since $\widehat\kappa \to \kappa^*$, we expect that 
$$\pi_l( \widehat \btheta_\MM(\bD)) \to \pi_l(\widehat\btheta_l(\bD; 0, \kappa^\star).$$
Indeed, Lemma~\ref{lemma:MM_l2_limit} in Section~\ref{sec:proof-universality-of-the-geometric-structure} makes this precise.
This reduces our Proposition~\ref{proposition:universality-of-the-geometric-structure} to establishing the universality of the latter quantity in the previous display at a fixed $\kappa = \kappa^\star$.
Recalling the definitions of $\cC_l$ in~\eqref{def:cl} with $\epsilon_n$ taken as above and letting
\begin{align} 
\widehat F_{l}(\bD; s,\kappa, b)&:= \min_{\btheta \in \cC_{l}(s,b)} \frac1n \sum_{i=1}^n \left(\kappa - y_i\langle \btheta, \bd_i \rangle \right)_+^2,\\
\widehat\kappa_{l}(\bD; s, b) &:= \max_{\theta \in \cC_{l}(s,b)}  \min_{i \in[n]} y_i \langle \btheta, \bd_i\rangle,
\end{align}
we note that these quantities are related by the equivalence
\begin{equation}\label{eq:equiv:H:F}
    \widehat H_{l}(\bD; s,\kappa)\leq b \iff  \widehat\kappa_{l}(\bD; s, b)\geq \kappa \iff \widehat F_{l}(\bD; s,\kappa, b)=0.
\end{equation}
for $b>0$. Using this this relation along with the universality of the ERM-like problem $\widehat F_l$, we establish the universality of $\widehat H_l$.
Indeed, we prove the following lemma.

\begin{lemma}
\label{lemma:univ_H}
Let $\os$ be as in~Proposition~\ref{prop:univ_fixed_kappa}.
Fix $|s| \le \os, b>0, l\in\{1,2\}$, and $\kappa>0$. For any $\delta\in(0,b)$, we have
\begin{align}
 \lim_{n\to\infty} \P\left(\widehat H(\bX, s,\kappa) > b \right) &\le 
  \lim_{n\to\infty} \P\left(\widehat H(\bG, s,\kappa) > b - \delta - C(b,\oslambda)|s|^{1/2} \right) \\
 \lim_{n\to\infty} \P\left(\widehat H(\bX, s,\kappa) \le b \right) &\le 
  \lim_{n\to\infty} \P\left(\widehat H(\bG, s,\kappa) \le b + \delta + C(b,\oslambda)|s|^{1/2} \right)
\end{align}
where $C(b,\oslambda)$ is a constant depending only on $b,\oslambda$, where $\oslambda$ is as in Lemma~\ref{lem:basic:op:norm:RF}.
\end{lemma}
Meanwhile, we argue that
\begin{equation}
\label{eq:derivative_relation}
 \frac{\partial}{\partial s}  \widehat H_l(\bD, s, \kappa^\star)\big|_{s = 0}  =  \pi_l(\widehat \btheta_l (\bD, 0,\kappa^\star)).
\end{equation}
Hence, to establish the universality of the latter quantity, we establish the universality of the derivative (in an appropriate sense) using the universality of $\widehat H_l$.


In the following section, we prove Lemma~\ref{lemma:univ_H} asserting the universality of $\widehat H_l$ for a fixed $\kappa>0$.
Then, in Section~\ref{sec:universality-of-derivative}, we show that this allows us to deduce, in an appropriate sense, the universality of the derivative in~\eqref{eq:derivative_relation}.
Finally, in Section~\ref{sec:proof-universality-of-the-geometric-structure}, we use the continuity of $\widehat\btheta_l(\bD;s,\kappa)$
in $\kappa$ to conclude the proof of Proposition~\ref{proposition:universality-of-the-geometric-structure}.

\subsubsection{Universality of $\widehat H_l$: proof of Lemma~\ref{lemma:univ_H}}
The proof proceeds by using the relation in Eq.~\eqref{eq:equiv:H:F} to relate $\widehat H_l$ to $\widehat F_l$ through $\widehat \kappa_l$.
Hence, we first need the following lower bound on $\widehat \kappa_l$ whose proof 
is similar to the proof of the unperturbed margin in Section~\ref{sec:proof-of-universality-max-margin}.
\begin{lemma}
\label{lemma:lower_bound_kappa}
Let $|s|\leq \os, l\in \{1,2\},$ and $\kappa>0$. With high probability, we have
 \begin{equation}
   \widehat \kappa_{l}(\bD; s,b) \ge \kappa - C \widehat F_{l}(\bD; s,\kappa, b)^{1/2}
 \end{equation}
 for some $C>0$, which is independent of $n$.
\end{lemma}
Moreover, we will require the following continuity property of $\kappa \to \widehat{H}_{l}(\bD;s,\kappa)$.
\begin{lemma}\label{lem:continiuty:H}
Let $|s|\leq \os, l\in \{1,2\},$ and $\kappa>0$. There exists $C\equiv C(\kappa,\os)$ such that with high probability, we have for small enough $\eps>0$
\begin{equation*}
    \widehat{H}_{l}(\bD;s,\kappa)-C\eps \leq \widehat{H}_{l}(\bD;s,\kappa-\eps)\leq  \widehat{H}_{l}(\bD;s,\kappa)
\end{equation*}
\end{lemma}
Before establishing these lemmas, let us give the proof of~\ref{lemma:univ_H}.

\begin{proof}[Proof of Lemma~\ref{lemma:univ_H}]
Both inequalities can be proven in a similar manner. We will only give a proof of the second one.
Fix arbitrary small enough $\eps>0$.
 Then we have by the equivalence in \eqref{eq:equiv:H:F} that
\begin{equation}
\begin{split}
 \lim_{n\to\infty} \P\left(\widehat H(\bX; s,\kappa) \le b \right) &=
 \lim_{n\to\infty} \P\left(\widehat F(\bX; s,\kappa, b) = 0\right) \\
 &\stackrel{(a)}\le\lim_{n\to\infty} \P\left(\widehat F(\bG; s,\kappa, b ) \le \eps \right) \\
 &\stackrel{(b)}\le \lim_{n\to\infty} \P\left( \widehat \kappa(\bG; s,b) \ge \kappa - C \eps^{1/2}  \right)\\ 
 &= \lim_{n\to\infty}\P\left(\widehat H(\bG; s,\kappa - C \eps^{1/2}) \le b\right)\\
 &\stackrel{(c)}\le \lim_{n\to\infty}\P\left(\widehat H(\bG; s,\kappa)  \le b + C^\prime \eps^{1/2}\right).
\end{split}
\end{equation}
where $(a)$ follows from Proposition~\ref{prop:univ_fixed_kappa}, $(b)$ follows from
Lemma~\ref{lemma:lower_bound_kappa}, and $(c)$ follows from Lemma~\ref{lem:continiuty:H}. Since $\eps>0$ was arbitrary, this completes the proof.
\end{proof}

Let us now return to the proofs of the auxiliary Lemmas~\ref{lemma:lower_bound_kappa} and~\ref{lem:continiuty:H}.
First, we give the following lemma.
\begin{lemma}
\label{lemma:S-l_inclusion}
Let $l\in\{1,2\}$,  $|s| \le \oslambda$ and $b>0$. 
On the event where $\norm{\bSigma}_\op \le \oslambda$  we have
\begin{equation}
B_2^p(c_l)  \subseteq \{\btheta: \norm{\btheta}_2^2 + s \pi_l(\btheta) \le b \} \subseteq B_2^p(C_l).
\end{equation}
for some constants $c_l\equiv c_l(\os,b), C_l\equiv C_l(\os, b) >0$ depending only on $\os, b >0$.
\end{lemma}

\begin{proof}
Let $|s|\leq \os$. For $l=1$, note that $\norm{\btheta}_2^2 + s \pi_l(\btheta) \le b$
implies
 \begin{equation}
  \norm{\btheta}_2^2 - |s| \norm{\btheta}_2 \oslambda\le b
 \end{equation}
 and hence
 \begin{equation}
  \norm{\btheta}_2 \le \frac{\os \oslambda}{2} + \frac{\sqrt{\os^2\oslambda + 4 b}}{2}\le C_1(\os, b)
 \end{equation}
 for some constant $C_1$ since $\oslambda$ is upper bounded with high probability.
Meanwhile,
 let $c_1\equiv c_1(\os, b)>0$ be the positive solution $c_1^2 + \os \osla c_1 = b$. If $\norm{\btheta}_2 \le c_1$ then we have
 \begin{equation}
  \norm{\btheta}_2^2 + s \pi_1(\btheta) \le \norm{\btheta}_2^2 + \os \oslambda \norm{\btheta}_2  \le c_1^2 + \os \osla c_1  = b.
 \end{equation}
For $l=2$, note that $\norm{\btheta}_2^2 + s\norm{\btheta}_\Sigma^2 = \norm{\btheta}_{(\bI + s \bSigma)}^2$ and hence
Since for $|s| \leq \os $ we have 
\begin{equation}
\frac12 \le \lambda_{\min}(\bI + s\bSigma) \le \lambda_{\max}(\bI + s\bSigma) \le 2,
\end{equation}
we conclude that 
the statement holds with $c_l = b/2$ and $C_l = 2b$.
\end{proof}
\begin{proof}[Proof of Lemma~\ref{lemma:lower_bound_kappa}]

Working on the event in the previous lemma (which is a high probability event by Lemma~\ref{lem:basic:op:norm:RF},
we have
\begin{align}
\widehat\kappa_{l}(\bD; s, b) &\ge  \max_{\theta \in B_2^p(c_l) \cap \{\norm{\btheta}_\infty \le \epsilon_n\}}
\min_{i \in[n]} y_i \langle \btheta, \bd_i\rangle
\ge
\min_{\bu^\sT \one = 1,\bu \ge 0 } c_l\widehat{G}_{n,\eps_n}^{\bD}(\bu).
\end{align}
Letting $\bu^\star$ be the minimzer, we apply Lemma~\ref{lemma:RSC_inf_norm} to conlude that
\begin{equation}
	\widehat\kappa_{l}(\bD; s, b)
	\ge c_1 \sqrt{n} \norm{\bu^\star}_2 -  c_2 \norm{\bu^\star}_1
\end{equation}
on a high probability event, which allows us to conlude that $\norm{\bu^\star}_2 \le \frac{C}{\sqrt{n}}$ on this event.
Hence, with high probability we have
\begin{align*}
 \widehat \kappa_l (\bD; s,b) &\ge
 \max_{\theta \in B_2^p(c_l) \cap \{\norm{\btheta}_\infty \le \epsilon_n\}}
\sum_{i=1}^n y_i u^\star_i\langle \btheta, \bd_i\rangle \ge \kappa -
C \left(\min_{\theta \in B_2^p(c_l) \cap \{\norm{\btheta}_\infty \le \epsilon_n\}}
 \frac1n \sum_{i=1}^n (\kappa - y_i \btheta^\sT \bd_i)_+^2\right)^{1/2}\\
 &\ge \kappa - C \frac{C_l}{c_l} \left(\widehat F_l(\bD; s,\kappa,b\right)^{1/2}
\end{align*}
as desired.
\end{proof}
We move onto the proof of Lemma~\ref{lem:continiuty:H}.
\begin{proof}[Proof of Lemma~\ref{lem:continiuty:H}]
The upper bound in the lemma follows directly from the definition. To prove the lower bound for $l=1$, let $\widehat\btheta_\epsilon$ denote the minimizer of $\widehat H_1(\bD; s, \kappa -\epsilon)$,
then
 \begin{align*}
  \widehat H_1(\bD; s,\kappa - \epsilon) 
  &= \norm{\widehat\btheta_\epsilon}_2^2 \left(1-\frac\epsilon\kappa\right)^2 + s \pi_1(\widehat\btheta_\epsilon)\left(1-\frac\epsilon{\kappa}\right)\\
     &\ge \left(1-\frac\epsilon{\kappa}\right) \widehat H_1(\bD;s,\kappa) 
     - \frac\epsilon{\kappa} \norm{\widehat\btheta_\epsilon}_2^2\left(1-\frac\epsilon{\kappa} \right)\\
     &= \widehat H_1(\bD; s,\kappa)  - \frac\epsilon{\kappa}\left( \widehat H_1(\bD; s,\kappa)+ \norm{\widehat\btheta_\epsilon}_2^2\left(1-\frac\epsilon{\kappa} \right) \right)\\
     &\ge \widehat H_1(\bD;s,\kappa)  -\epsilon C(\kappa,s )
 \end{align*}
 where the last inequality holds  with high probability.
 A similar argument holds for $l=2$ by noting that
 \begin{equation}
  \widehat H_2(\bD; s,\kappa - \epsilon) 
  =\left( 1- \frac{\epsilon}{\kappa}\right)^2  \widehat H_2(\bD; s,\kappa).\\
 \end{equation}
\end{proof}

\subsubsection{Universality of the difference quotient}
\label{sec:universality-of-derivative}
Since the proof for both $l=1$ and $l=2$ is the same, we suppress the index $l$ throughout and use the notation $\widehat \btheta_s^\bD(\kappa)$ for minimizer of the optimization problem for $\widehat H_l(\bD; s)$ in \eqref{eq:def:H}.
Furthermore, we introduce the compact notation $\widehat \btheta_s^\bD = \widehat \btheta_s^\bD(\kappa^\star)$ in what follows.

The starting point in our perturbation argument is to show that the minimizer $\widehat\btheta_s^\bD$ is stable under small changes in $s$ due to the strong
convexity of the problem $\widehat H$.

\begin{lemma}
\label{lemma:theta_hat_lipschitz_s}
With high probability, we have for $|s|\leq \os\equiv (2\oslambda)^{-1}$,
\begin{equation}
 \norm{\widehat\btheta_s^\bD -\widehat\btheta_0^\bD}_2 \le C |s|
\end{equation}
for some $C=C(\os)>0$.
\end{lemma}
\begin{proof}
Let $f(\btheta;s):= \norm{\btheta}_2^2+s\pi(\btheta)$. By strong convexity of $f(\cdot;s)$ for $|s|\leq \os$, we have
 \begin{equation}\label{eq:lower:f}
  \widehat H(\bD; s,\kappa^\star)=f(\widehat \btheta^{\bD}_s;s)\geq f(\widehat \btheta^{\bD}_0;s)+\mu\norm{\widehat \btheta^{\bD}_s-\widehat \btheta^{\bD}_0}_2^2.
 \end{equation}
 where $\mu>0$ is a universal constant. Meanwhile, since $\widehat{\btheta}^{\bD}_0$ minimize $f(\btheta; 0)$ in the constraint set $\{\btheta: y_i\langle \btheta, \bd_i\rangle \geq \kappa^\star\}$, we have
 \begin{equation}\label{eq:tech:1}
     \begin{split}
         f(\widehat \btheta^{\bD}_s;s)-f(\widehat \btheta^{\bD}_0;s)
  &= f(\widehat \btheta^{\bD}_s;0)-f(\widehat \btheta^{\bD}_0;0)+s\left( \pi(\widehat\btheta^\bD_s) - \pi(\widehat\btheta^\bD_0)\right)\\
  &\leq s\left( \pi(\widehat\btheta^\bD_s) - \pi(\widehat\btheta^\bD_0)\right)\\
  &\le |s|\oslambda \norm{\widehat\btheta^\bD_s-\widehat\btheta^\bD_0}_2\left(1\vee\left(\norm{\widehat\btheta^\bD_s}_2+\norm{\widehat\btheta^\bD_0}_2\right)\right),
     \end{split}
 \end{equation}
where the last inequality holds since it is straightforward to check that for $\btheta_1,\btheta_2\in \R^p$,
\begin{equation}\label{eq:pi:pseudo:lipschitz}
    \left|\pi(\btheta_1)-\pi(\btheta_2)\right|\leq \oslambda \norm{\btheta_1-\btheta_2}_2\left(1\vee(\norm{\btheta}_1+\norm{\btheta}_2\right)
\end{equation}
for both $\pi(\btheta)=\btheta^{\sT}\bSigma \btheta^\star$ and $\pi(\btheta)=\norm{\btheta}_{\bSigma}^2$. Note that $\norm{\widehat\btheta^\bD_0}_2\leq 2$ holds w.h.p. by Lemma \ref{lemma:MM_l2_limit}. To upper bound $\norm{\widehat\btheta^\bD_s}_2$, we use $f(\widehat \btheta^{\bD}_s;s)\leq f(\widehat \btheta^{\bD}_0;s)$:
\begin{equation}\label{eq:upper:2:norm:s}
    \norm{\widehat{\btheta}^{\bD}_{s}}_2^2+s\pi(\widehat{\btheta}^{\bD}_s)\leq \norm{\widehat{\btheta}^{\bD}_{0}}_2^2+s\pi(\widehat{\btheta}^{\bD}_0)\leq C\left(\norm{\widehat{\btheta}^{\bD}_{0}}_2\vee\norm{\widehat{\btheta}^{\bD}_{0}}_2^2\right)).
\end{equation}
Since $\norm{\widehat\btheta^\bD_0}_2\leq 2$ holds w.h.p. by Lemma \ref{lemma:MM_l2_limit}, \eqref{eq:upper:2:norm:s} implies that $\norm{\widehat{\btheta}^{\bD}_{s}}_2\leq C$ w.h.p. for some constant $C>0$ (for exact detail, see Lemma \ref{lemma:S-l_inclusion} below). Thus, plugging this into the right hand side of \eqref{eq:tech:1} shows that
\begin{equation}\label{eq:upper:f}
  f(\widehat \btheta^{\bD}_s;s)-f(\widehat \btheta^{\bD}_0;s)\leq C |s|\norm{\widehat\btheta^\bD_s-\widehat\btheta^\bD_0}_2,
\end{equation}
for some constant $C\equiv C (\os)>0$. Therefore, combining \eqref{eq:lower:f} and \eqref{eq:upper:f} concludes the proof.
\end{proof}

Now, define
\begin{equation}
\widehat \Delta(\bD;s) := \frac{\widehat H(\bD;s,\kappa^\star) - \widehat H(\bD;0, \kappa^\star) }{s}.
\end{equation}
An elementary, but crucial observation is that for any $s>0$,
\begin{equation}\label{eq:pi:Delta}
\begin{split}
\widehat\Delta(\bD; s)&\leq \frac{ \norm{\widehat{\btheta}^{\bD}_{0}}_2^2+s\pi(\widehat{\btheta}^{\bD}_0)-\norm{\widehat{\btheta}^{\bD}_{0}}_2^2}{s}\\
&=\pi(\widehat\btheta_0^\bD(\kappa))=\frac{ \norm{\widehat{\btheta}^{\bD}_{0}}_2^2-s\pi(\widehat{\btheta}^{\bD}_0)-\norm{\widehat{\btheta}^{\bD}_{0}}_2^2}{-s}\leq \widehat\Delta(\bD;-s).
\end{split}
\end{equation}
The following lemma establishes Lipschitzness and universality of $\widehat \Delta(\bD;s)$ as a consequence of Lipschitzness of $\widehat\btheta^\bD(s)$ and universality of $\widehat H(\bD; s,\kappa^\star)$.
\begin{lemma}
\label{lemma:properties_of_delta}
Let $0< s\leq \os$. With high probability, we have
\begin{equation}
 \widehat\Delta(\bD;-s) - \widehat\Delta^\bD(\bD;s) \le C s
\end{equation}
for some $C>0$ depending only $\os$. Furthermore, for any $r\in\R$ and $\delta>0$, we have
\begin{align*}
 \lim_{n\to\infty}\P\left(\widehat\Delta(\bX;-s) \ge r + \delta\right) 
&\le 
 \lim_{n\to\infty}\P\left(\widehat\Delta(\bG;-s) \ge r \right) \\
 \lim_{n\to\infty}\P\left(\widehat\Delta(\bX;s) \le r - \delta\right) 
&\le 
 \lim_{n\to\infty}\P\left(\widehat\Delta(\bG;s) \le r \right) 
\end{align*}
\end{lemma}
\begin{proof}
For the first assertion, w.h.p., we have
 \begin{align*}
 \widehat\Delta(\bD;-s) - \widehat\Delta(\bD;s) 
 &= -\frac1s \left(\norm{\widehat\btheta_{-s}^{\bD}}_2^2 - \norm{\widehat\btheta_0^{\bD}}_2^2  \right)
 - \frac1s\left(\norm{\widehat\btheta^\bD_s}_2^2 - \norm{\widehat\btheta^\bD_0}_2^2 \right)\\
 &+ \pi(\widehat\btheta_{-s}^\bD) -\pi(\widehat\btheta_{s}^\bD)\\
 &\stackrel{(a)}{\le} \pi(\widehat\btheta_{-s}^\bD) -\pi(\widehat\btheta_{s}^\bD)\\
 &\stackrel{(b)}{\le} C\norm{\widehat\btheta_{-s}^\bD-\widehat\btheta_{s}^\bD}_2\\
 &\stackrel{(c)}{\le} C^\prime|s|
 \end{align*}
 where $(a)$ holds since $\norm{\widehat\btheta_{-s}^{\bD}}_2\wedge \norm{\widehat\btheta_{s}^{\bD}}_2\geq \norm{\widehat\btheta_{0}^{\bD}}_2$ by definition of $\widehat\btheta_{0}^{\bD}$, $(b)$ holds on the high probability event $\{\norm{\widehat\btheta^{\bD}_0}_2\leq 2\}$ (cf. Lemma \ref{lemma:MM_l2_limit} by \eqref{eq:pi:pseudo:lipschitz} and \eqref{eq:upper:2:norm:s}, and $(c)$ is due to Lemma \ref{lemma:theta_hat_lipschitz_s} and triangle inequality.
 
 For the second part of the lemma, we have
 \begin{equation}
 \begin{split}
  \lim_{n\to\infty}\P\left(\widehat\Delta(\bX;-s) \ge r + 3\delta \right)
  &\le 
  \lim_{n\to\infty}\P\left( \frac{\widehat H(\bX;-s,\kappa^\star) - 1 }{-s}\ge r + 2\delta \right)
  +\lim_{n\to\infty}\P\left(\left|\widehat H(\bX;0,\kappa^\star) - 1\right| > s\delta\right)\\
  &\stackrel{(a)}=
  \lim_{n\to\infty}\P\left( \widehat H(\bX;-s,\kappa^\star) \le 1- s(r + 2\delta) \right)\\
  &\stackrel{(b)}\le\lim_{n\to\infty}\P\left( \widehat H(\bG;-s,\kappa^\star)\le 1- s(r + \delta) + C(\oslambda,b)|s| \right)\\
  &= \lim_{n\to\infty}\P\left( \frac{\widehat H(\bG;-s,\kappa^\star)- 1 }{-s}\ge r+\delta \right)\\
  &\stackrel{(c)}\le 
  \lim_{n\to\infty}\P\left( \frac{\widehat H(\bG;-s,\kappa^\star) - \widehat H(\bG;0,\kappa^\star)}{-s}\ge r\right)\\
  &=
  \lim_{n\to\infty}\P\left(\widehat\Delta(\bG;-s) \ge r \right),
 \end{split}
 \end{equation}
 where in $(a)$ and $(c)$ we used Lemma~\ref{lemma:MM_l2_limit}, in $(b)$ we used Lemma~\ref{lemma:univ_H}. The other inequality follows analogously.
\end{proof}

\subsubsection{Proof of Proposition~\ref{proposition:universality-of-the-geometric-structure}}
\label{sec:proof-universality-of-the-geometric-structure}
The following lemma first makes the connection between $\widehat \btheta^{\bD}_0 = \widehat\btheta^\bD_0(\kappa^\star)$ and $\bthetaMM^{\bD}$.
\begin{lemma}
\label{lemma:MM_l2_limit}
With high probability, we have
\begin{equation}\label{lemma:MM_l2_limit:1}
    \widehat \btheta_0^{\bD}=\frac{\kappa^\star}{\widehat \kappa_n(\bD)}\bthetaMM^{\bD}.
\end{equation}
Thus, we have that
 \begin{equation}\label{lemma:MM_l2_limit:2}
  \widehat H(\bD; 0,\kappa^\star)=\norm{\widehat\btheta_0^\bD}_2^2 \pto 1.
 \end{equation}
\end{lemma}
\begin{proof}
Since $\sgamma>\sgamma^\star$, the data is seperable w.h.p., thus \eqref{lemma:MM_l2_limit:1} is straightforward from the duality between the optimization problems that defines the max-margin estimator $\bthetaMM^{\bD}$ and that defines $\widehat H(\bD; 0,\kappa^\star)$ in \eqref{eq:def:H}. The second assertion \eqref{lemma:MM_l2_limit:2} is a consequence of the first assertion \eqref{lemma:MM_l2_limit:1} and 
Theorem~\ref{theorem:universality-of-the-margin} since $\norm{\bthetaMM^{\bD}}_2=1$ holds whenever the data is separable.
\end{proof}

We are now ready to give the proof of Proposition \ref{proposition:universality-of-the-geometric-structure}.
\begin{proof}[Proof of Proposition \ref{proposition:universality-of-the-geometric-structure}]

To start with, we make two reductions: first, by Proposition \ref{proposition:gaussian-angles}, it suffices to prove that $\pi_{l}(\bthetaMM^{\bX})\pto c_{l}^\star$ for $l\in \{1,2\}$. 
Second, since by Lemma \ref{lemma:MM_l2_limit} we have for $l \in \{1,2\}$
\begin{equation*}
    \pi_{l}(\bthetaMM^{\bX})=\pi_{l}\left(\frac{\widehat{\kappa}_n(\bX)}{\kappa^\star}\widehat{\btheta}_0^{\bX}\right)=\left(\frac{\widehat{\kappa}_n(\bX)}{\kappa^\star}\right)^{l}\pi_{l}(\widehat{\btheta}_0^{\bX}),
\end{equation*}
and $\widehat{\kappa}_n(\bX)\pto \kappa^\star$ holds by Theorem~\ref{theorem:universality-of-the-margin}, it suffices to prove that $\pi_{l}(\widehat{\btheta}_0^{\bX})\pto c_{l}^\star$ for $l\in \{1,2\}$.

Now, fix arbitrary $\delta>0$ and take $s\equiv s_{\delta}\in (0,\os \wedge \frac{\delta}{C})$, where $C$ is the constant appearing in Lemma~\ref{lemma:properties_of_delta}. Then, 
\begin{align}
\lim_{n\to\infty}  \P\left( \pi_l(\widehat\btheta_0^\bX) \ge c_l^\star + 3\delta \right)
&\stackrel{(a)}{\le} 
\lim_{n\to\infty}  \P\left( \widehat \Delta(\bX;-s) \ge \pi_l^\star + 3\delta \right)\\
&\stackrel{(b)}{\le}
\lim_{n\to\infty}  \P\left( \widehat \Delta(\bG;-s) \ge \pi_l^\star + 2\delta \right)\\
&\stackrel{(c)}{\le}
\lim_{n\to\infty}  \P\left( \widehat\Delta(\bG;s) + Cs \ge \pi_l^\star + 2\delta \right)\\
&\stackrel{(d)}{\le}
\lim_{n\to\infty}  \P\left( 
\pi_l(\widehat\btheta_0^\bG(\kappa^\star)) \ge \pi_l^\star + \delta \right)\\
&=0,
\end{align}
where $(a)$ and $(d)$ are due to \eqref{eq:pi:Delta}, and $(b)$ and $(c)$ are by Lemma \ref{lemma:properties_of_delta}. A similar argument establishes the lower bound in probability, concluding the proof.
\end{proof}

\section*{Acknowledgements}
We acknowledge the support of NSF through award DMS-2031883 (A.M., B.S., and Y.S.), the Simons Foundation through
Award 814639 for the Collaboration on the Theoretical Foundations of Deep Learning (A.M., B.S., and Y.S.), the NSF grant CCF-2006489 (A.M. and B.S.), the ONR grant N00014-18-1-2729 (A.M. and B.S.), and Vannevar Bush Faculty Fellowship award ONR-N00014-20-1-2826 (Y.S.). Part of this work was carried out
while A.M. was on partial leave from Stanford and a Chief Scientist at Ndata Inc
dba Project N. The present research is unrelated to A.M.’s activity while on leave.

\bibliography{bib}
\bibliographystyle{amsalpha}

\appendix
\newpage
\section{Characterization of the margin and the test error for Gaussian design from~\cite{montanari2019generalization}}
\label{sec:formula}
In this section, we review the characterization of the quantities $\kappa^\star_{\RF}(\gamma_1,\gamma_2), \pi^{\star}_{\ell,\RF}(\gamma_1,\gamma_2), \ell=1,2,$ and $\kappa^\star_{\ind}(\gamma;\mu), \pi^\star_{\ell,\ind}(\gamma;\mu), \ell=1,2,$ in Propositions~\ref{proposition:max-margin-gaussian} and \ref{proposition:gaussian-angles} from \cite{montanari2019generalization}. We first describe the characterization of the asymptotic margin $\kappa^\star(\gamma;\mu)$ and the asymptotic functionals $\pi^\star_{\ell}(\gamma;\mu), \ell=1,2,$ under a set of assumptions that generalizes Assumption~\ref{assumption:ind}. Then, we specialize such description to the Gaussian equivalent of the independent features model in Section~\ref{subsec:gaussian:indep}, and to the Gaussian equivalent of the random features model in Section~\ref{subsec:gaussian:rf}.

Consider the proportional asymptotics $n,p\to\infty$ with $\frac{p}{n}\to \gamma$ and the Gaussian covariates $(\bg_i)_{i\leq n} \stackrel{i.i.d}{\sim} \normal(0,\bSigma)$. The labels $\by=(y_i)_{i\leq n}\in \{\pm 1\}^n$ are distributed as
\begin{equation*}
\P(y_i=+1\,\mid\,\bg_i)=f(\bg_i^{\sT}\btheta^\star)\,,
\end{equation*}
where $\btheta^\star\equiv\btheta^\star_n\in \R^p$ satisfies $\|\btheta^\star\|_2=1$. We further assume that $(\bSigma,\btheta^\star,f)\equiv (\bSigma_n,\btheta^\star_n, f)$ satisfies the following, which generalizes ${\sf (B1)}$-${\sf (B3)}$. Let $\bSigma_n=\sum_{i=1}^{p}\lambda_i \bv_i \bv_i^{\sT}$ be the eigenvalue decomposition of $\bSigma_n$ with eigenvalues $(\lambda_i)\equiv \big(\lambda_i(\bSigma_n)\big)_{i\leq p}$ with $\lambda_1\geq\ldots \geq \lambda_p>0$.
\begin{itemize}
    \item ${\sf (C1)}:$ There exist constants $\lambda_{\textsf{M}}, L>0,$ and $\eps>0$ such that
	\begin{equation*}
	 \lambda_1(\bSigma_n)\le \lambda_{\sf M}\,,\quad\quad \frac{1}{p}\sum_{i=1}^{p} \left(\frac{1}{\la_i(\bSigma_n)}\right)^{1+\eps}\leq L.
 \end{equation*}
 \item ${\sf (C2)}:$ Let $\rho_n \equiv \langle \btheta^{\star}_n,\bSigma_n\btheta^{\star}_n\rangle^{1/2}$ and $\bar{w}_i \equiv  \sqrt{p\lambda_i} \<\bv_i,\btheta^{\star}_{n}\>/\rho_n$.
Then the empirical distribution of $\{(\lambda_i,\bar{w}_i)\}_{1 \leq i \leq n}$ converges in Wasserstein-$2$ distance to a 
probability distribution $\nu$ on $\reals_{>0}\times \reals$:
\begin{align*}
\frac{1}{p}\sum_{i=1}^p\delta_{(\lambda_i, \bar{w}_i)}\stackrel{W_2}{\Longrightarrow} \nu\, .
\end{align*}
In particular, $\int w^2 \nu(\de\lambda,\de w) = 1$, and we have that
\begin{equation*}
    \rho \equiv \lim_{n\to\infty} \rho_n = \int   (w^2/\lambda) \nu(\de\lambda,\de w).
\end{equation*}
\item ${\sf (C3)}:$ Let $T = YG$ where 
\begin{equation}\label{eq:Y:G:dist}
    \P(Y = 1 \mid G)= 1-\P(Y = -1 \mid G)= f(\rho \cdot G),\quad G \sim \normal(0, 1). 
\end{equation}
We assume $f: \R \to [0,1]$ to be continuous, and it satisfies the following non-degeneracy condition: 
	\begin{equation*}
		\inf \Big\{x: \P(T<x)>0 \Big\} = -\infty
			~~\text{and}~~
		\sup \Big\{x: \P(T>x)>0   \Big\} = \infty\, .
	\end{equation*}
\end{itemize}

Indeed, it is straightforward to check that if ${\sf (B1)}$, ${\sf (B2)}^\prime$ and ${\sf (B3)}$ are satisfied, then ${\sf (C1)}$-${\sf (C3)}$ are satisfied. For any $\kappa\ge 0$, 
define the function $F_{\kappa}: \R \times \R_{\ge 0} \to \R_{\ge 0}$ by
\begin{equation}
\label{eqn:def-F-kappa}
F_{\kappa}(c_1, c_2) = \left(\E \left[(\kappa - c_1 YG - c_2 Z)_+^2\right]\right)^{1/2}
~~\text{where}~
	\begin{cases}
		Z \perp (Y, G) \\
		Z \sim \normal(0, 1), G \sim \normal(0, 1) \\
		\P(Y = +1 \mid G) = f(\rho \cdot G) \\
		\P(Y = -1 \mid G) = 1-f(\rho\cdot G)
	\end{cases}
\end{equation}
For $(X,W)\sim \nu$, where $\nu$ is given in ${\sf (C2)}$, introduce the constants 
\begin{equation*}
\xi := \left(\E [X^{-1}W^2]\right)^{-1/2}\quad\text{and}\quad\omega := \left(\E [(1-\zeta^2 X^{-1})^2W^2]\right)^{1/2}. 
\end{equation*}
Define the functions $\gamma_+: \R_{>0} \to \R$~~\text{and}~$\gamma_-: \R_{>0} \to \R$ by
\begin{equation}
\begin{split}
\gamma_+(\kappa) &:= 
	\begin{cases}
		0~~&\text{if $\partial_1 F_{\kappa}(\xi, 0) > 0$,} \\
			\left(\partial_2 F_{\kappa}(\xi, 0)\right)^2 - \omega^2 \left(\partial_1 F_{\kappa}(\xi, 0)\right)^2~~~~~~~&\text{if otherwise,}
	\end{cases}
	\\ 
\gamma_-(\kappa) &:= 
	\begin{cases}
		0~~&\text{if $\partial_1 F_{\kappa}(-\xi, 0) > 0$,} \\
			\left(\partial_2 F_{\kappa}(-\xi, 0)\right)^2 - \omega^2 \left(\partial_1 F_{\kappa}(-\xi, 0)\right)^2~~&\text{if otherwise.}
	\end{cases}
\end{split}
\end{equation}
Here, we denoted $\partial_i\equiv \frac{\partial}{\partial c_i}$ for $i=1,2$. Finally, we define $\gamma^{\star}(\nu)$ and $\gamma^{\low}:\R_{>0}\to \R_{\ge 0}$ by 
\begin{align}
	\gamma\opt(\nu)&:= \min_{c\in \R} F_0^2 (c, 1)\, ,\label{eq:Psistar0}\\
\gamma^{\low}(\kappa) &:= \max\{\gamma\opt(0), \gamma_+(\kappa), \gamma_-(\kappa)\}.\label{eq:PsilowDef}
\end{align}

The next proposition is needed to define $\kappa^\star(\gamma;\mu)$ and $\pi_{\ell}^\star(\gamma;\mu)$ for $\ell=1,2,$.
\begin{proposition}[{\cite[Proposition 5.1]{montanari2019generalization}}]
\begin{enumerate}
\item[$(a)$]  For any $\gamma > \gamma^{\low}(\kappa)$, the following system of equations has unique solution $(c_1, c_2, s) \in \R \times \R_{>0} \times \R_{>0}$
(here expectation is taken with respect to $(X,W)\sim \nu$):
	\begin{equation}
	\label{eqn:system-of-equations-c-1-c-2-s-with-no-G}
		\begin{split}
		-c_1 &= \E\left[\frac{ \left(\partial_1 F_{\kappa}(c_1, c_2) -  
			c_1 c_2^{-1} \partial_2 F_{\kappa}(c_1, c_2)\right) W^2 X^{1/2}}
				{ c_2^{-1} \partial_2 F_{\kappa}(c_1, c_2) X^{1/2} + \gamma^{1/2} s X^{-1/2}}\right] \, , \\
		c_1^2 + c_2^2 &=  \E \left[\frac{\gamma X + 
			 \left(\partial_1 F_{\kappa}(c_1, c_2) -  c_1 c_2^{-1} \partial_2 F_{\kappa}(c_1, c_2)\right)^2 W^2 X}
				{( c_2^{-1} \partial_2 F_{\kappa}(c_1, c_2) X^{1/2} + \gamma^{1/2} s X^{-1/2})^2}\right] \, ,\\
		1 &= \E
			\left[\frac{\gamma + 
			 \left(\partial_1 F_{\kappa}(c_1, c_2) -  c_1 c_2^{-1} \partial_2 F_{\kappa}(c_1, c_2)\right)^2 W^2}
				{( c_2^{-1} \partial_2 F_{\kappa}(c_1, c_2) X^{1/2} + \gamma^{1/2} s X^{-1/2})^2}\right]\, .
		\end{split}
	\end{equation}
\item[$(b)$] Define the function $T: (\gamma, \kappa) \to \R$ (for any $\gamma > \gamma^{\low}(\kappa)$) by
	\begin{equation}
	\label{eqn:def-T}
		T(\gamma, \kappa)\equiv T(\gamma, \kappa;\nu) :=  \gamma^{-1/2} \left(F_{\kappa}(c_1, c_2) - c_1 \partial_1 F_{\kappa}(c_1, c_2) -	
		c_2 \partial_2 F_{\kappa}(c_1, c_2) \right) - s\, ,
	\end{equation}
	where $(c_1(\gamma, \kappa), c_2(\gamma, \kappa), s(\gamma, \kappa))\equiv(c_1(\gamma, \kappa; \nu), c_2(\gamma, \kappa;\nu), s(\gamma, \kappa;\nu))$ is the unique solution of Eq
	\eqref{eqn:system-of-equations-c-1-c-2-s-with-no-G} in $\R \times \R_{>0} \times \R_{>0}$. 
	Then we have 
	\begin{enumerate}
	\item[$(i)$] $T(\cdot, \cdot), c_1(\cdot, \cdot), c_2(\cdot, \cdot), s(\cdot, \cdot)$ are continuous functions in the domain 
		$\left\{(\gamma, \kappa): \gamma > \gamma^{\low}(\kappa)\right\}$. 
	\item[$(ii)$] For any $\kappa > 0$, the mapping $T(\, \cdot\, , \kappa)$ is strictly monotonically decreasing, and satisfies
	\begin{equation}
		\lim_{\gamma \nearrow +\infty} T(\gamma, \kappa) < 0 < \lim_{\gamma \searrow \gamma^{\low}(\kappa)} T(\gamma, \kappa). \label{eq:Tpsi}
	\end{equation}
	\item[$(iii)$] For any $\gamma > 0$, the mapping $T(\gamma, \cdot)$ is strictly monotonically increasing, and satisfies
	\begin{equation}
		\lim_{\kappa \nearrow +\infty} T(\gamma, \kappa) = \infty. 
	\end{equation}
	\end{enumerate}
\end{enumerate}
\end{proposition}

We can now define $\kappa^\star(\gamma;\mu)$ and $\pi_{\ell}^\star(\gamma;\mu)$ for $\ell=1,2$.
\begin{definition}\label{def:KappaE}
For $\gamma\geq \gamma^{\star}(\nu)$, define the \emph{asymptotic max-margin}
as 
\begin{equation*}
\kappa\opt(\gamma;\nu) = \inf\left\{\kappa \ge 0: T(\gamma, \kappa; \nu) = 0\right\}. 
\end{equation*}
For $\gamma<\gamma^\star(\nu)$, we define $\kappa^{\star}(\gamma;\nu):=0$. For $\gamma\geq \gamma^{\star}(\nu)$, we define $\pi_{\ell}\opt(\gamma;\nu) := c_{\ell}(\gamma, \kappa\opt(\gamma;\mu);\nu))$ for $\ell=1,2$.
\end{definition}
\cite[Theorem 3]{montanari2019generalization} established that the asymptotic margin and the asymptotic generalization error under assumptions ${\sf (C1)}$-${\sf (C3)}$ are respectively as follows:
\begin{equation*}
    \widehat \kappa_n^{\bG}\pto \kappa^\star(\gamma;\nu)\quad\textnormal{and}\quad R_n^{\bG}\pto R^\star(\gamma;\nu):=Q\left(\frac{\pi_1^\star(\gamma;\nu)}{\rho\cdot \pi_2^\star(\gamma;\nu)}\right),
\end{equation*}
where $Q$ is explicitly given as follows: for $G,Z\stackrel{i.i.d.}{\sim} \normal(0,1)$ and $\nu \in [0,1]$, let
\begin{equation}\label{eq:def:Q}
	Q(\nu) := \P\left(\nu YG+ \sqrt{1-\nu^2} Z \le 0\right)~~\text{where}~~
		\begin{cases}
			\P(Y= +1\mid G) = f(\rho\cdot G) \\
			\P(Y= -1\mid G) = 1 -f(\rho\cdot G).
		\end{cases}
\end{equation}
\subsection{Gaussian equivalent of the independent features model}
\label{subsec:gaussian:indep}
Observe that the Gaussian design $\bG_{\ind}\in \R^{n\times p}$ under Assumption~\ref{assumption:ind} satisfies ${\sf (C1)}$-${\sf (C3)}$ above, and it can be seen as a special case where $\bSigma$ is a diagonal matrix. In particular, recalling $\mu\in \mathscr{P}(\R_{>0}\times \R)$ from \eqref{eq:empirical:convergence}, the condition ${\sf (C2)}$ is satisfied for $\nu_{\mu}\in \mathscr{P}(\R_{>0}\times \R)$, where $\nu_{\mu}$ is obtained from $\mu$ by a change of measure: for any Borel-measurable $A\subseteq \R_{>0}$ and $B\subseteq \R$, we have
\begin{equation*}
\nu_{\mu}\big(\lambda\in A\,,\,w\in B\big)= \mu\Big(\lambda\in A\,,\,\frac{\sqrt{\lambda}\theta}{\rho} \in B\Big)\,,\quad\textnormal{where}\quad \rho:=\int \lambda \theta^2 \mu(\de \lambda, \de \theta)\,.
\end{equation*}
Thus, we have the following definition.
\begin{definition}
We define $\kappa^{\star}_{\ind}(\gamma;\mu):=\kappa^{\star}(\gamma;\nu_{\mu})$ and $\gamma^{\star}_{\ind}(\mu):=\gamma^{\star}(\nu_{\mu})$. Further, for $\ell=1,2$, we define $\pi^\star_{\ell,\ind}(\gamma;\mu):=\pi^{\star}_{\ell}(\gamma;\nu_{\mu})$, and $R^\star_{\ind}(\gamma;\mu):=R^\star(\gamma;\nu_{\mu})$.
\end{definition}
With the definition above, the statements for $\bG=\bG_{\ind}$ in Proposition~\ref{proposition:max-margin-gaussian} and Proposition~\ref{proposition:gaussian-angles} is satisfied by \cite[Theorem 3, Proposition 6.4]{montanari2019generalization}.
\subsection{Gaussian equivalent of the random features model}
\label{subsec:gaussian:rf}
The Gaussian equivalent of the random features model $\bG_{\RF}$ also satisfies ${\sf (C1)}$-${\sf (C3)}$ conditional on the weights $\bW$ by the Marchenko-Pastur's law. To see this, fix $\gamma_1,\gamma_2, \mu_1,\mu_2$, and denote $\gamma:= \frac{\gamma_1}{\gamma_2}$, the limit of $\frac{p}{n}$. First, define the following probability measure on $(0,\infty)$:
\begin{align}
\nu_{\mp}(\de x) &= \begin{cases}
(1-\gamma_1^{-1})\delta_{0} +\gamma_1^{-2}\nu\Big(\frac{x}{\gamma_1};\frac{1}{\gamma_1}\Big)\de x&\;\;\mbox{if $\gamma_1>1$,}\\
\nu(x;\gamma_1)\de x&\;\;\mbox{if $\gamma_1\in (0,1]$,}
\label{eq-MPlaw}
\end{cases}\\
\nu(x; \lambda)& = \frac{\sqrt{(\lambda_+-x)(x-\lambda_-)}}{2\pi \lambda x} \, \bfone_{x\in [\lambda_-,\lambda_+]}\, ,\\
\lambda_{\pm} & =  (1\pm\sqrt{\lambda})^2\, .
\end{align}
By Marchenko-Pastur's law, the empirical spectral distribution of $\bW\bW^{\sT}$ converges in Wasserstein $2$-distance $W_2$ to $\nu_{\mp}$ almost surely as $p,d\to\infty$ with $\frac{p}{d}\to \gamma_1$ (see e.g. \cite{bai2010spectral}).

Recall $\bSigma\equiv \bSigma_{\RF}$ in \eqref{eq:def:Sigma} and $\btheta^\star$ in \eqref{eq:def:theta:star}. Let $\bSigma=\sum_{i=1}^{p}\lambda_i \bv_i\bv_i^{\sT}$ be the eigenvalue decomposition of $\bSigma$ with $\lambda_1\geq \lambda_2\geq...\ge \lambda_p$. Moreover, let $\rho_n \equiv \norm{\btheta^\star}_{\bSigma}$ and $\bar{w}_i \equiv \sqrt{p\lambda_i} \langle \bv_i, \btheta^\star \rangle / \rho_n$. Then, we can determine the limit of joint empirical distribution of $\{(\lambda_i, \bar{w}_i)\}_{i\leq p}$ by Marchenko-Pastur's law: if we let $\tX \sim \nu_{\mp}$ be independent of $G\sim \normal(0,1)$, then

\begin{align}
\label{eq:RF:nu}
\frac{1}{p}\sum_{i=1}^p\delta_{(\lambda_i,\bar{w}_i)}\stackrel{W_2}{\Longrightarrow} \nu_{\gamma_1}\equiv {\sf Law}(X,W)\, ,
\end{align}
where
\begin{align}
X = \mu_1^2\tX+\mu_2^2\, ,\;\;\;\;\;\;\;\;\;\;
W = \frac{\mu_1\sqrt{\gamma_1 \tX}\, G}{C_0 (\mu_1^2\tX+\mu_2^2)^{1/2}}\, ,\;\;\;\;\;\;\;\;\;\;
C_0 = \left(\E\Big[\frac{\mu_1^2 \gamma_1 \tX}{\mu_1^2\tX+\mu_2^2}\Big]\right) ^{1/2} \, .
\end{align}
Moreover, we have that $\rho_n:= \norm{\btheta^\star}_{\bSigma}\to \rho\equiv \rho_{\RF}$ and $\alpha_n\equiv \mu_1\norm{\bSigma^{-1}\bW\beta^\star}_2 \to \alpha\equiv \alpha_{\RF}$ (defined in \eqref{eq:def:theta:star}) as $n\to\infty$, where
\begin{equation}\label{eq:def:rho:star}
\begin{split}
    \rho:=\left(\E\Big[\frac{\tX}{\mu_1^2\tX+\mu_2^2}\Big]\bigg/\E\Big[\frac{\tX}{(\mu_1^2\tX+\mu_2^2)^2}\Big]\right)^{1/2},\quad\alpha:=\bigg(\E\Big[\frac{\mu_1^2 \gamma_1 \tX}{\left(\mu_1^2\tX+\mu_2^2\right)^2}\Big]\bigg) ^{1/2}.
\end{split}
\end{equation}
Then, we define $f_{\RF}(x)$, the limit of $\bar{f}(x)$ by
\begin{align}
\label{eq:def:f:star}
f_{\RF}(x):=\E_{g\sim \normal(0,1)}f(\alpha x+\tau g),\quad\textnormal{where}\quad
\tau:= \left(1-\gamma_1\E\Big[\frac{\mu_1^2 \tX}{\mu_1^2\tX+\mu_2^2}\Big]\right)^{1/2}.
\end{align}
Since $\tau^2>0$, it follows that $f_{\RF}(\cdot)$ satisfies ${\sf (C3)}$. Thus, we have the following definition.
\begin{definition}
    We define $\kappa^{\star}_{\RF}(\gamma_1,\gamma_2):=\kappa^\star\big(\frac{\gamma_1}{\gamma_2},\nu_{\gamma_1}\big)$ and $\tau^\star_{\RF}(\gamma_1,\gamma_2):=\gamma^{\star}(\nu_{\gamma_1})$. Further, we define $\pi^\star_{\ell,\RF}(\gamma_1,\gamma_2):=\pi^{\star}_{\ell}\big(\frac{\gamma_1}{\gamma_2};\nu_{\gamma_1}\big)$ for $\ell=1,2$, and $R^\star_{\RF}(\gamma_1,\gamma_2):=R^\star\big(\frac{\gamma_1}{\gamma_2};\nu_{\gamma_1}\big)$.
\end{definition}
With the definition above, the statements for $\bG=\bG_{\RF}$ in Proposition~\ref{proposition:max-margin-gaussian} and Proposition~\ref{proposition:gaussian-angles} is satisfied by \cite[Theorem 3, Proposition 6.4]{montanari2019generalization}.
\section{Auxiliary lemmas}
\label{section:aux_lemmas}
We state here a collection of auxiliary lemmas used throughout the proofs. 
\subsection{High probability bounds}
We begin with high probability bounds on the norms for the random features model.
\begin{lemma}
\label{lem:basic:op:norm:RF}
Let $\overline W = 2$
and let $\Omega_1 := \left\{\norm{\bW_\RF}_\op \le \overline W\right\}$. 
Then the following hold:
\begin{enumerate}
 \item
 There exists constant $c$ such that $\P(\Omega_1^c ) \le e^{-c n}$.
\item On $\Omega_1$,
 $\norm{\bSigma_{\RF}}_{\op}$ is bounded by $\osla = \mu_1 \overline W^2 + \mu_2$. 
\item $\bx_{\RF} \one_{\Omega_1}$ is $O(1)$ subgaussian, and hence 
$\norm{\bX_{\RF}\one_{\Omega_1}}_\op$ and $\norm{\bx_{\RF} \one_{\Omega_1}}_2$ are bounded by $Cn^{1/2}$ for some constant $C$.
\end{enumerate}
\end{lemma}
\begin{proof}
The proof of the first two items follows from standard results (see~\cite{vershynin2018high}), 
meanwhile, the last item was shown in~\cite{HuLu22}.
\end{proof}

\subsection{Asymptotic pointwise-Gaussianity of the non-Gaussian features}
We establish the pointwise asymptotic Gaussianity of the non-Gaussian features.
We note that this result was apriori established for the random features distribution in~\cite{HuLu22}. 
However, to establish the delocalization of $\widehat\btheta_\MM$ (Proposition~\ref{proposition:delocalization-for-general-max-margin}),
we require a faster rate of convergence requiring a different approach for pointwise-Gaussianity.

\paragraph{The random features model}
Let us first define the augmented random features and their Gaussian equivalent. 
Recall the distributional setting of Section~\ref{section:ass_RF}.
In what follows, $\bx,\bg$ have the same distribution as a row of $\bX_\RF$ and $\bG_\RF$ respectively.
\begin{equation}
  \bxi_i := (\bx_i^\sT ,\bz_i^\sT )^\sT \in \R^{p + d}\quad
   \bgamma_i := (\bg_i^\sT ,\bz_i^\sT )^\sT \in \R^{p + d}.
 \end{equation}
We prove the following lemma.

\begin{lemma}
\label{lemma:pointwise-gaussianity}
Let $\Omega_1$ be as in Lemma~\ref{lem:basic:op:norm:RF}.
Let $\delta_n$ be any sequence such that $\delta_n\to0$. Then
for any Lipschitz function $\varphi$, we have
 \begin{equation}
  \lim_{n\to\infty} \sup_{\norm{\bzeta_1}_\infty \le \delta_n, \norm{\bzeta}_2 \le C ,\bW \in \Omega_1}
  \left|
  \E[\varphi( \bzeta^\sT\bxi) | \bW]
  -\E[\varphi( \bzeta^\sT\bgamma) |\bW]
  \right|=0
 \end{equation}
 where $\bzeta = (\bzeta_1^\sT,\bzeta_2^\sT)^\sT$, $\bzeta_1\in\R^p$, $\bzeta_2\in\R^d$.
\end{lemma}
\begin{proof}
Define 
$$h(\bz) := \bzeta^\sT \bxi(\bz)= \sum_{j=1}^p\xi_j \sigma(\bw_j^\sT \bz) + \sum_{j=p+1}^{p+d} \xi_jz_j.$$
We compute the derivatives of $h(\bz)$:
\begin{align}
 \grad h(\bz) = \bW^\sT \bM_1 \bzeta_1 + \bzeta_2\quad
\grad^2 h(\bz) =  \bW^\sT \bM_2 \diag\{\bzeta_1\} \bW.
\end{align}

So by the second order Poincar\'e inequality, we have
\begin{equation}
  \left|
  \E[\varphi( \bzeta^\sT\bxi)|\bW ]
  -\E[\varphi( \bzeta^\sT\balpha)|\bW ]
  \right| \le  C \frac{\left(\norm{\bW}_\op \norm{\sigma'}_\infty \norm{\bzeta_1}_2 + \norm{\bzeta_2}_2 \right)
  (\norm{\bW}_\op^2 \norm{\sigma''}_\infty \norm{\bzeta_1}_\infty)}{\Var(\bzeta^\sT\bxi)}
\end{equation}
where $\balpha$ is a Gaussian vector with mean $0$ and covariance matching that of $\bxi$.
So for any $\delta>0$, we have
\begin{equation}
 \sup_{\norm{\bzeta_1}_\infty \le \delta_n, \norm{\bzeta}\le C,\Var(\bzeta^\sT\bxi) > \delta} 
  \left|
  \E[\varphi( \bzeta^\sT\bxi)|\bW ]
  -\E[\varphi( \bzeta^\sT\balpha)|\bW ]
  \right| \le  C_1 \frac{\norm{\bW}_\op^3 \delta_n}{\delta} 
\end{equation}
Meanwhile, 
\begin{equation}
 \sup_{\Var(\bzeta^\sT\bxi) \le \delta} 
  \left|
  \E[\varphi( \bzeta^\sT\bxi)|\bW ]
  -\E[\varphi( \bzeta^\sT\balpha) |\bW ]
  \right|  
  \le C_2 \norm{\varphi}_\Lip \Var(\bzeta^\sT \bxi)^{1/2} \le C_3 \delta^{1/2}.
\end{equation}
Since $\norm{\bW}_\op \le W^\star$ on $\Omega_1$, with probability at least $1-e^{-c_0 p}$, we have

\begin{align}
 \sup_{
 \norm{\bzeta_1}_\infty \le \delta_n, \norm{\bzeta}_2\le C
 } 
  \left|
  \E[\varphi( \bzeta^\sT\bxi) |\bW]
  -\E[\varphi( \bzeta^\sT\balpha)|\bW ]
  \right|  \le
  C_5 \frac{\delta_n}{\delta} + 
   C_3 \delta^{1/2} + \norm{\varphi}_\infty (e^{-c_0 p}).
\end{align}
Sending $n \to\infty$ then $\delta \to 0$ shows that the statement of the lemma holds with $\bgamma$ replaced with $\balpha$.
Now by \cite[Lemma 5]{HuLu22},
we have
\begin{equation}
 \lim_{n\to\infty}\sup_{\bW\in\Omega_1}\norm{\Cov(\balpha | \bW)- \Cov(\bgamma|\bW) }_\op  = 0.
\end{equation}
Hence,
\begin{equation}
 \sup_{\bW\in\Omega_1}\Big|\E\left[\varphi(\bzeta^\sT \balpha) - \varphi(\bzeta^\sT\bgamma)\big|\bW\right] \Big|
 \le
\norm{\varphi}_\Lip
\norm{\bzeta}_2^2 \sup_{\bW\in\Omega_1} \norm{\Cov(\alpha|\bW) - \Cov(\gamma|\bW) }_\op.
\end{equation}
Optimizing over $\bzeta$ then taking $n\to\infty$ completes the proof.

\end{proof}

\paragraph{Independent entries}
We state the following lemma whose proof is a direct consequence of the classical Lindeberg CLT.
Recall the distributional setting of Section~\ref{section:ass_RF}.
In what follows, $\bx,\bg$ have the same distribution as a row of $\bX_{\ind}$ and $\bG_{\ind}$ respectively.
\begin{lemma}
\label{lemma:pointwise-gaussianity_iid}
Let $\delta_n$ be any sequence such that $\delta_n\to0$.
Then
for any bounded Lipschitz function $\varphi$, we have
 \begin{equation}
  \lim_{n\to\infty} \sup_{\norm{\btheta}_\infty \le \delta_n}
  \left|
  \E[\varphi( \btheta^\sT\bx)]
  -\E[\varphi( \btheta^\sT\bg)]
  \right|=0.
 \end{equation}
\end{lemma}

\subsection{Proof of Lemma~\ref{lemma:crazy-lemma-random-feature-models}}
\label{sec:proof-crazy-lemma-random-feature-models}
Throughout this proof, we work conditionally on $\bZ$. 
Clearly, the columns of $\bX$ are independent copies of $\sigma(\bZ^T \bw_i)$. Note 
$\bw \mapsto \langle \bv, \sigma(\bZ^T \bw)\rangle$ is Lipschitz with Lipschitz constant 
at most $\norm{\bZ}_\op \norm{\bv}_2 \norm{\sigma}_{{\rm lip}}$. Hence, 
$\bX$ is subgaussian with parameter 
\begin{equation}
	\nu(\bZ) = \frac{C}{\sqrt{d}} \cdot \norm{\bZ}_\op \norm{\sigma}_{{\rm lip}}.
\end{equation} 
for some universal 
constant $C > 0$. Furthermore, the covariance matrix of $\bX$, denoted by $\bSigma_{\bZ}$, is given by 
\begin{equation*}
	\bSigma_{\bZ} \equiv \Big(\bSigma_{\bZ}(k,\ell)\Big)_{k,\ell\leq n}\equiv \bigg(\E_{\bw} \Big[\sigma\big(\langle \bw, \bz_{k}\rangle\big)\cdot \sigma(\big\langle \bw, \bz_{\ell}\big\rangle)\Big]\bigg)_{k,\ell\leq n}\,,
\end{equation*} 
where $\bz_{1},\ldots,\bz_{n}$ are columns of $\bZ$ and $\E_{\bw}$ denotes the expectation with 
respect to $\bw$ conditioned on $\bZ$.

Let us denote $\bSigma_{\bZ}^\prime\equiv \frac{\mu_1^2}{d}\bZ^{T}\bZ + \mu_2^2 \bI_n$.
Introduce $\kappa_n =  \max_{1\le i, j \le n} \sqrt{d}\cdot |\frac{1}{d}\langle \bz_i, \bz_j\rangle - \mathbf{1}_{i = j}|$, 
and $\zeta_n = \norm{\bZ}_\op/\sqrt{d}$. According to \cite[Lemma 5]{HuLu22} (see Remark F.1. in~\cite{MontanariSa22} for applicability), 
there is a constant $C > 0$ independent of $n$ such that with probability one there holds
\begin{equation*}
	\norm{\bSigma_{\bZ} - \bSigma_{\bZ}^\prime}_\op \le \delta_n(\bZ)
\end{equation*}
where $\delta_n(\bZ) = C \cdot \frac{ 1+ \kappa_n^3 + \zeta_n^4 }{\sqrt{n}}$. Hence, we may use 
triangle inequality to obtain the bound 
\begin{equation}
	\mu_2 - \delta_n(\bZ) \le \lambda_{\min}(\bSigma_{\bZ}) \le \lambda_{\max} (\bSigma_{\bZ}) \le 
		\mu_1 \cdot \frac{\norm{\bZ}_\op}{\sqrt{d}} + \mu_2 + \delta_n(\bZ).
\end{equation}

To finish the proof, we give bounds on $\nu^2(\bZ)$, as well as $\delta_n(\bZ)$ that holds with high probability. 

\begin{itemize}
\item 
By Theorem~\ref{theorem:matrix-concentration-covariance-operator}, $\norm{\bZ}_\op \le C (\sqrt{p} + \sqrt{d})$
with probability at least $1-e^{-cp}$. Here, $C, c$ are universal constants. 

\item 
By Theorem 3.1.1 of \cite{vershynin2018high}, $|\norm{\bz}_2 - \sqrt{d}|$ is subgaussian with parameter 
$c$ where $c > 0$ is universal. A union bound then shows 
$\max_{i \le n} |\norm{\bz_i}_2 - \sqrt{d}| \le C \sqrt{\log n}$ holds 
with probability at least $1-\frac{1}{2n}$ for some universal constant $C$. Also, one can show that 
$\max_{i, j \le n, i\neq j} |\langle z_i, z_j\rangle| \le C \log n$ with probability at least $1-\frac{1}{2n}$
for some large universal constant $C > 0$. 

This shows that $|\delta_n(\bZ)| \le \frac{C}{\sqrt{n}}((1+\sqrt{p/d})^4 + \log^4 n)$ with probability at least $1-n^{-1}$. 
\end{itemize}

Now we simply define $\Omega_{\bZ}$ as the event for which the conclusions of the above two bullet points 
hold. Lemma~\ref{lemma:crazy-lemma-random-feature-models} then follows from the above discussions.

\end{document}